\documentclass[11pt]{amsart}
\usepackage{amsmath,amssymb,amsthm}
\usepackage[latin1]{inputenc}
\usepackage{version,tabularx,multicol}
\usepackage{graphicx,float,psfrag}
 \usepackage{stmaryrd}
 \usepackage{mathrsfs}

\newcommand{\reff}[1]{(\ref{#1})}

\theoremstyle{plain}
\newtheorem{theo}{Theorem}[section]
\newtheorem*{theo*}{Theorem}
\newtheorem{cor}[theo]{Corollary}
\newtheorem{prop}[theo]{Proposition}
\newtheorem{lem}[theo]{Lemma}
\newtheorem{defi}[theo]{Definition}
\newtheorem*{defi*}{Definition}
\theoremstyle{remark}
\newtheorem{rem}[theo]{Remark}

\newcommand{\ca}{{\mathcal A}}

\newcommand{\cc}{{\mathcal C}}
\newcommand{\cd}{{\mathcal D}}
\newcommand{\ce}{{\mathcal E}}

\newcommand{\cf}{{\mathcal F}}
\newcommand{\cg}{{\mathcal G}}
\newcommand{\ch}{{\mathcal H}}
\newcommand{\ci}{{\mathcal I}}

\newcommand{\ck}{{\mathcal K}}
\newcommand{\cl}{\varphi}
\newcommand{\ff}{f}

\newcommand{\cm}{{\mathcal M}}

\newcommand{\cs}{{\mathscr S}}

\newcommand{\cu}{{\mathcal U}}

\newcommand{\C}{{\mathbb C}}

\newcommand{\E}{{\mathbb E}}

\newcommand{\N}{{\mathbb N}}
\newcommand{\Nz}{{\mathbb N}}
\newcommand{\Np}{{\mathbb N}^*}
\newcommand{\Ni}{\bar \Nz}
\renewcommand{\P}{{\mathbb P}}

\newcommand{\R}{{\mathbb R}}

\newcommand{\T}{{\mathbb T}}
\newcommand{\Tf}{{\mathbb T}_{\rm f}}

\newcommand{\rn}{{\rm n }}

\newcommand{\re}{{\rm e}}

\newcommand{\rs}{{\rm s}}

\newcommand{\bt}{{\mathbf t}}
\newcommand{\bs}{{\mathbf s}}

\newcommand{\ind}{{\bf 1}}
\newcommand{\fc}{\mathfrak{c}}
\newcommand{\fa}{\mathfrak{a}}
\newcommand{\fb}{\mathfrak{b}}
\newcommand{\fg}{\mathfrak{g}}
\newcommand{\fm}{\mathfrak{m}}
\newcommand{\fp}{\mathfrak{p}}
\newcommand{\fR}{\mathfrak{R}}

\newcommand{\ske}{{\rm Ske}} 
\newcommand{\anc}{{\rm Anc}}

\newcommand{\val}[1]{\mathop{\left| #1 \right|}\nolimits}
\newcommand{\inv}[1]{\mathop{\frac{1}{ #1}}\nolimits}
\newcommand{\expp}[1]{\mathop {\mathrm{e}^{ #1}}}

\renewcommand{\mod}{{\rm mod}\;}

\newcommand{\lb}{[\![}
\newcommand{\rb}{]\!]}

\title[Expansive GW trees]{Asymptotic properties of expansive Galton-Watson trees}



\date{\today}
\author{Romain Abraham} 

\address{
Romain Abraham,
Laboratoire MAPMO, CNRS, UMR 7349,
F\'ed\'eration Denis Poisson, FR 2964,
 Université d'Orléans,
B.P. 6759,
45067 Orléans cedex 2,
France.
}
  
\email{romain.abraham@univ-orleans.fr}

\author{Jean-François Delmas}

\address{
Jean-Fran\c cois Delmas,
Université Paris-Est, \'Ecole des Ponts, CERMICS, 6-8
av. Blaise Pascal, 
  Champs-sur-Marne, 77455 Marne La Vallée, France.}

\email{delmas@cermics.enpc.fr}

\begin{document}

\begin{abstract}
  We consider a super-critical  Galton-Watson tree whose non-degenerate
  offspring distribution has  finite mean. We consider  the random trees
  $\tau_n$ distributed  as $\tau$ conditioned on  the $n$-th generation,
  $Z_n$, to be of size $a_n\in \N$.   We identify the possible local limits of
  $\tau_n$  as $n$  goes to  infinity according  to the  growth rate  of
  $a_n$. In  the low regime,  the local  limit $\tau^0$ is  the Kesten
  tree, in the moderate regime  the family of local limits, $\tau^\theta$
  for $\theta\in (0, +\infty )$,  is distributed as $\tau$ conditionally
  on  $\{W=\theta\}$,  where  $W$  is the  (non-trivial)  limit  of  the
  renormalization  of $Z_n$.   In the  high regime,  we prove  the local
  convergence towards $\tau^\infty $ in  the Harris case (finite support
  of  the offspring  distribution)  and  we give  a  conjecture for  the
  possible limit  when the  offspring distribution has  some exponential
  moments.  When the offspring distribution  has a fat tail, the problem
  is  open.   The  proof  relies   on  the  strong  ratio   theorem  for
  Galton-Watson processes. Those latter results  are new in the low regime
  and high regime,  and they can be used to  complete the description of
  the (space-time) Martin boundary of Galton-Watson
  processes. Eventually, we consider the continuity in distribution of
  the local limits $(\tau^\theta, \theta\in [0, \infty ])$. 
\end{abstract}

\maketitle

\section{Introduction}

The study  of Galton-Watson  (GW) processes and  more generally  GW trees
conditioned to be  non extinct goes back to  Kesten \cite{k:sbrwrc}, see
Lemma 1.14 therein. In the sub-critical and non-degenerate critical case
the extinction event $\ce$ being of  probability one, there are many non
equivalent limiting  procedure to  define a GW  tree conditioned  on the
non-extinction  event. Those  so-called local  limits of  GW trees  have
received  a renewed  interest  recently because  of  the possibility  of
condensation phenomenon:  a node  in the limiting  tree has  an infinite
degree.  This  appears when  conditioning sub-critical  GW trees  to be
large,   see   Jonsson   and  Stef\'anson   \cite{js:cnt}   and   Janson
\cite{j:sgtcgwrac}  when  conditioning  on large  total  population  and
Abraham  and  Delmas  \cite{ad:llcgwtcc},  when  conditioning  on  large
sub-population or \cite{ad:hapi} for a survey from the same authors. The
other typical behavior for the local limit  of GW trees is to exhibit an
infinite spine  on which  are grafted  independent finite  GW sub-trees,
such as  in \cite{k:sbrwrc}.  Various  conditionings lead to  such local
limit, which we  call Kesten tree, for critical or  subcritical GW tree,
see Abraham and Delmas \cite{ad:llcgwtisc}  and references therein for a
general study  and \cite{ad:hapi}  for other  recent references  in this
direction also. Intuitively, the local limit is the Kesten tree when the
events approximating the non-extinction event decrease in probability at
polynomial rate.  One of the  motivations of the current work is to
present local  limits of sub-critical  GW trees with  different behavior
(that is other than an infinite spine or a node of infinite degree), see
the  partial  results  from  Section  \ref{sec:sub-critical},  where  we
present a family  of local limits with an infinite  backbone not reduced
to a spine.

Recently,  with Bouaziz,  we  considered in  \cite{abd:llfgt} the  local
limits of   GW trees  $\tau$ with geometric offspring  distribution (see
Section \ref{sec:geom} for a precise  definition) conditioned on the size $Z_n$ of
the population at  generation $n$ being equal to $a_n\in  \N$.  Because the
distribution  of   $Z_n$  is   explicit  for  the   geometric  offspring
distribution, it  is possible to  compute all the possible  local limits
(if any) for the sub-critical, critical and super-critical cases and for
all the  possible sequences  $(a_n, n\in  \N^*)$. The  local limit  is a
random   tree   which   depends   on  the   rate   of   convergence   of
$(a_n, n\in  \N^*)$ towards  infinity.  When  this sequence  is positive
bounded or grows slowly to infinity,  the limit is still the Kesten tree.
This result already appears in the critical case in \cite{ad:llcgwtisc},
see Section 6.  When the growth  to infinity is moderate, then the local
limit  can be  described as  an infinite  random backbone  on which  are
grafted independent finite GW trees.  Surprisingly the backbone does not
enjoy the branching  property as the numbers of  children of individuals
at generation $n$ on the backbone are not independent and depend also on
the size of  the backbone at generation $n$.  If  the growth to infinity
is high, then the local limit exhibits the condensation phenomenon: the
root, and only  the root, of the  local limit has an  infinite number of
children. The aim of the present  work is to extend those results mainly
to  general  super-critical  offspring distribution  and  marginally  to
sub-critical offspring distribution.

\subsection{The main results}

Let $p=(p(k), k\in \N)$ be  a non-degenerate offspring distribution with
finite  mean   $\mu=\sum_{k\in  \N}   k  p(k)$.   Let  $f$   denote  the
corresponding generating  function so that $f'(1)=\mu$, and let $R_c\geq
1$ be its radius of convergence.  We shall mainly
consider   the    super-critical   case    $\mu\in (1, +\infty )$,   but    in   Section
\ref{sec:sub-critical}  where  we  consider  a  particular  sub-critical
offspring distribution  (that is  $\mu\in (0,1)$).

We  recall the  local convergence  of random  ordered rooted  tree.  The
ordered rooted trees, defined in  Section \ref{sec:tree}, are subsets of
the    set     of    finite     sequences    of     positive    integers
$\cu=\bigcup_{n\geq      0}(\Np)^{n}$      with      the      convention
$(\Np)^{0}=\{\partial\}$,  and $\partial$  being the  root of  the tree.
For    a    tree    $\bt$    and    $u\in    \cu$,    we    denote    by
$k_u(\bt)\in \bar  \N=\N\bigcup \{\infty  \}$ the  out-degree of  a node
$u\in \bt$ or equivalently the number  of children of $u$ in $\bt$, with
the  convention that  $k_u(\bt)=-1$ if  $u\not \in  \bt$.  We  denote by
$z_h(\bt)$ the size of $\bt$ at generation  $h\in \N$.  A sequence of trees
$\bt_n$ converges locally  to a tree $\bt$ if  $k_u(\bt_n)$ converges to
$k_u(\bt)$ for  all $u\in \cu$.   And we say  that a sequence  of random
trees $T_n$  converges locally in distribution  to a random tree  $T$ if
$(k_u(T_n), u\in \cu)$ converges in distribution to $(k_u(T), u\in \cu)$
for the finite dimensional  marginals. See Section \ref{sec:cv-of-t} for
a precise  setting.

We  consider  the  random  tree  $\tau$ defined  as  the  GW  tree  with
super-critical non-degenerate offspring distribution $p$ and finite mean
$\mu$, and we  define $Z=(Z_n=z_n(\tau), n\in \N)$  the corresponding GW
process, with $Z_n$ being the size  of $\tau$ at generation $n$, strating at
$Z_0=1$.  Let $\fa\in \N$ and $\fb\in \bar \N$ be respectively the lower
and upper  bound of  the support of  $p$.  We have  $\fa<\fb$ as  $p$ is
non-degenerate.  Let $\fc=\P(\ce)$ be  the probability of the extinction
event.  We  recall that $\fc\in [0,1)$  is the only root  of $f(r)=r$ on
$[0,  1)$.  Notice  that  $\fc=0$  if and  only  if  $\fa\geq 1$.   When
$\P(Z_n=a_n)>0$,  we denote  by $\tau_n$  a random  tree distributed  as
$\tau$ conditioned on $\{Z_n=a_n\}$.  We  study the local convergence in
distribution of $(\tau_n, n\in \N^*)$ according to the rate of growth of
the sequence $(a_n,  n\in \N^*)$.  According to  Seneta \cite{s:sgwp} or
Asmussen  and Hering  \cite{ah:bp}, we  shall consider  the Seneta-Heyde
norming  $(c_n,  n\in \N)$  which  is  a  sequence such  that  $Z_n/c_n$
converges a.s.  to a limit $W$  and $\P(W=0)=\fc$, see its definition in
Section \ref{sec:gen-gen}.   When $\mu=+\infty  $, then such a normalization does not exists and  when the  $L\log(L)$ condition  holds, that  is
$\sum_{k\in \N^*} p_k \log(p_k)<+$, then  $c_n$ is equivalent to $\mu^n$
up  to  an  arbitrary   positive  multiplicative  constant,  see  Seneta
\cite{s:fegwp}.  However, we stress out that the $L\log(L)$ condition is
not  assumed in  this paper  and that  we only  consider the  case $\mu$
finite. It  is well known  that the  distribution of $W$,  restricted to
$(0, +\infty )$,  has a continuous positive density $w$  with respect to
the Lebesgue measure, see the seminal work of Harris \cite{h:bp} and the
general result from Dubuc \cite{d:pri}. However, $w$ is explicitly known
in only  two cases:  the geometric  offspring distribution,  see Section
\ref{sec:geom} below and the  example developed by Hambly \cite{h:ctbl}.
\medskip

We now  introduce the possible local limiting
trees. 
\begin{defi}
\label{defi:tau}
  Let $\tau$ be a  GW tree with non-degenerate super-critical offspring
  distribution $p$ with finite mean.
\begin{itemize}
   \item If $\fc>0$, we denote by $\tau^{0, 0} $ a random
  tree distributed as $\tau$ conditionally on the extinction event
  $\ce$.
\item If $\fc>0$,  we denote by $\tau^0$ the  corresponding Kesten tree,
  see Definition \ref{def:Kesten}. If $\fc=0$  (that is $\fa\geq 1$), we
  denote by $\tau^0$ the deterministic regular $\fa$-ary tree.
\item For $\theta\in (0, +\infty )$, we denote by $\tau^\theta$ a random
  tree distributed as $\tau$ conditioned on $\{W=\theta\}$. 
\item If $\fb<\infty $, we denote by $\tau^\infty $ the deterministic regular
  $\fb$-ary tree.  If $\fb=+\infty $ and $R_c>1$, we denote by $\tau^\infty
  $ the random tree $T^{(\lambda_c)}$ given in Section
    \ref{sec:tinfini}. 
\end{itemize}
\end{defi}
According   to   Remark    \ref{rem:decomp-E},   the   distribution   of
$\tau^\theta$,  which is  defined  in Section  \ref{sec:extremal}, is  a
regular  version   of  the  distribution  of   $\tau$  conditioned  on
$\{W=\theta\}$ for  $\theta\in (0,  +\infty )$.  The  tree $\tau^\theta$
can  be intuitively  described  as a  random  (non-homogeneous in  time)
infinite backbone on which, if $\fc>0$, are grafted independent GW trees
distributed as  $\tau^{0, 0}$.  The  description of the backbone  and of
its  offspring distribution  is one  of  the main  contribution of  this
paper.  The infinite backbone does not enjoy the branching property, and
the offspring  distribution $\rho_{\theta,r}$ of the  individuals of the
current generation  depends on the  size $r$ of the  current generation,
see   Definition   \reff{eq:def-rho}.   The   probability   distribution
$\rho_{\theta, r}$ is  a function of the density  $w$.  The distribution
of  $\tau^\theta$ is  in a  sense a  generalization of  the Kesten  tree
distribution. 

The  tree   $\tau^\infty  $  appears   as  a  natural  local   limit  of
non-homogeneous   GW  trees   $T^{(\lambda)}$   introduced  in   Section
\ref{sec:tau-lambda}  and with  a nice  representation given  in Section
\ref{sec:tau-2-type} using two-type GW trees.

The  condensation   holds  for   $\tau^\infty  $,  defined   in  Section
\ref{sec:tinfini},  at  least   at  the  root  if   $\fb=+\infty  $  and
$f(R_c)=+\infty  $.   Furthermore,  the   tree  $\tau^\infty  $   is  not
homogeneous in general.

We can now give the first main result on the local convergence in
distribution of $\tau_n$ according to the growth rate of $(a_n,n\in \N^*)$. 
\begin{theo}
\label{theo:main-cv-loc}
  Let $\tau$ be a  GW tree with non-degenerate super-critical offspring
  distribution $p$ with finite mean.
 We assume that the
  sequence $(a_n, n\in \N^*)$ is such that $\tau_n$ is well defined for
  all $n\in \N ^*$. 
\begin{itemize}
\item   \textbf{Extinction  case:}     $a_n=0$  for   all
  $n\ge n_0$ for some $n_0\in\N^*$.   If $\fc=0$, then $\tau_n$ is not
  defined. If $\fc>0$, then $\tau_n$ is well defined and  we have:
\[
\tau_n\; \xrightarrow[n\rightarrow \infty ]{\textbf{(d)}}\; \tau^{0,0}.
\]
\item         \textbf{Low          regime:}  
  $\lim_{n\rightarrow\infty } a_n  /c_n=0$ and $a_n>0$ for  all $n\in \N^*$.
 Then, we have:
\[
\tau_n\; \xrightarrow[n\rightarrow \infty ]{\textbf{(d)}}\; \tau^{0}.
\]
\item        \textbf{Moderate        regime:} 
  $\lim_{n\rightarrow\infty  } a_n  /c_n=\theta\in  (0,  +\infty )$.  
 Then, we have:
\[
\tau_n\; \xrightarrow[n\rightarrow \infty ]{\textbf{(d)}}\; \tau^{\theta}.
\]
\item         \textbf{High         regime:} $\lim_{n\rightarrow\infty  }
  a_n  /c_n=+\infty  $.  (Partial results.)        If  $\fb<\infty $
  (Harris case) or if  $p$ is
  geometric,  then we have:
\[
\tau_n\; \xrightarrow[n\rightarrow \infty ]{\textbf{(d)}}\; \tau^{\infty }.
\]
\end{itemize}
\end{theo}

The  local  convergence  is  well  known in  the  extinction  case,  see
Proposition \ref{prop:cv-killed}.  For  the low regime, it  is stated in
Proposition \ref{prop:cv-not-fat-gene}.  For the  moderate regime, it is
stated in  Proposition \ref{prop:cv-fat-gene}.  For the  high regime, it
is  stated in  \cite{abd:llfgt}  for $p$  geometric  and in  Proposition
\ref{prop:cv-not-vfat-gene} for the Harris case.  All the proofs rely on
the  strong  ratio  theorem,  see  Section  \ref{sec:srthm}  below.   We
conjecture that the local convergence of $\tau_n$ towards $\tau^\infty $
in  the high  regime  holds if  $R_c>1$ (or  equivalently  $W$ has  some
positive   exponential  moments,   see   the  first   part  of   Section
\ref{sec:tau-lambda}).  The existence of local limits in the high regime
when $R_c=1$ is open.

\begin{rem}
  \label{rem:deftn} 
  We  recall from  Dubuc  \cite{d:eapgw} some  sufficient conditions  on
  $x\in  \N$  such  that   $\P_k(Z_n=x)>0$,  where  $\P_k$  denote,  for
  $k\in  \N^*$, the  distribution of  the  GW process  $Z$ started  from
  $Z_0=k$.  Notice first that if $x=0$, then $\P_k(Z_n=x)>0$ if and only
  if $\fc>0$, that is $\fa=0$.

  We now consider the case $x>0$. The offspring distribution
  $p$ is said to be of type  $(L_0,r_0)$, if $L_0$ is the period  of $p$, that is
  the           greatest           common           divisor           of
  $\{n-\ell; \, n>\ell  \text{ and } p(n)p(\ell)\neq 0\}$,  and $r_0$ is
  the residue  $(\mod L_0)$ of  any $n$ such  that $p(n)\neq 0$.   It is
  clear  that  $\P_k(Z_n=x)>0$  implies  $ x=k  \,  r_0^n  (\mod  L_0)$.
  According to \cite{d:eapgw}, for any  $b>\fa$ such that $p(b)>0$ (take
  $b=\fb$ if  $\fb<\infty $), there exists  $d\in \N$ such that  for all
  $k\in \N ^*$ and $x\in \lb  k\fa^n +  d, kb^n -d\rb$  with $x=kr_0^n(\mod  L_0)$, we
  have $\P_k(Z_n=x)>0$.

Taking $k=1$ and $x=a_n$, this provides sufficient
  conditions for $\tau_n$  to  be well  defined.  In particular, notice
  that there exist sequences $(a_n, n\in \N ^*)$ in all the regime such that
  $\tau_n$ is well defined. 
\end{rem}

Moreover, we have  the following continuity result in distribution for the family of limiting trees.
\begin{theo}
   \label{theo:cont}
   Let  $p$ be  a non-degenerate  super-critical offspring  distribution
   with        finite mean.       The      family
   $(\tau^\theta, \theta\in [0, +\infty ))$  is continuous for the local
   convergence in distribution.  Furthermore, if  
   $\fb<\infty $ or if $p$ is geometric, then we have:
\[
\tau^\theta\; \xrightarrow[\theta\rightarrow +\infty ]{\textbf{(d)}}\;
\tau^{\infty }. 
\]
\end{theo}

The continuity of $(\tau^\theta, \theta\in  [0, +\infty ))$ is proven in
Section \ref{sec:cont-0} and more precisely in Corollary \ref{cor:cvttq}
for the  continuity at  $0$.  The continuity at  $0$ allows to explain and extend
Corollary  3 from  Berestycki,  Gantert  and Mörters  \cite{bgms:gwtvml}
on  the
convergence  in distribution  of $\tau_{(\varepsilon)}$ (distributed as
$\tau$ conditionally  on $\{0<W\leq \varepsilon\}$) towards  $\tau^0$ as
$\varepsilon$ goes down to 0, see Corollary \ref{cor:link}.
\medskip

The  continuity at  infinity is  proven in
\cite{abd:llfgt}   for   the   geometric   case   and   in   Proposition
\ref{prop:cvtq-tinfty-gen}  for the  Harris  case. If  $R_c>1$, we  also
conjecture that  the local convergence in  distribution of $\tau^\theta$
towards $\tau^\infty $  as $\theta$ goes to infinity.  We  get a partial
result in this direction in  Section \ref{sec:cont-infty}, as if $R_c>1$
and if $\tau^\theta$ converges locally  in distribution as $\theta$ goes
to  infinity, then  the limit  is indeed  $\tau^\infty $,  see Corollary
\ref{cor:cv-t-inf}.
If  $R_c=1$,  then  we  have   not  hint  concerning  the  existence  or
non-existence of possible  limits for $\tau^\theta$ as  $\theta$ goes to
infinity. Notice that it is not  clear that $\tau^\theta$  is stochastically
non-decreasing with $\theta$.

\begin{rem}
   \label{rem:sub-intro}
Partial results concerning the sub-critical case are presented in
Section \ref{sec:sub-critical}, under the assumption that $R_c>1$ and
the equation $f(r)=r$ has a finite root in $(1, R_c]$. This assumption
is equivalent to assume that the sub-critical GW tree is distributed as
a super-critical GW tree conditioned on the extinction event. In this case,
we can use the previous results in the  super-critical case  to get results
in the sub-critical case.
\end{rem}

\subsection{Strong ratio theorem for super-critical GW process}
\label{sec:srthm}
We set for $k,h\in \N^*$:
\begin{equation}
   \label{eq:defH}
H_n(h,k)=\frac{\P_{k}(Z_{n-h}=a_n)}{\P(Z_n=a_n)},
\end{equation}
where $Z$ is under $\P_k$ a GW process starting from $Z_0=k$.  The proofs
of Theorem \ref{theo:main-cv-loc}, when  there is non condensation, rely
on   the   elementary   identity    \reff{eq:ph}   which   states   that
$  \P(r_{h}(\tau_n)=\bt) =  H_n(h,  z_h(\bt)) \P(r_h(\tau)=\bt)$,  where
$r_h(\bs)$ denotes  the restriction of  the tree $\bs$ up  to generation
$h\in \N^*$, and  $\bt$ is a tree with height  $h$ (that is $z_h(\bt)>0$
and $z_{h+1}(\bt)=0$).  Since the  local convergence in  distribution of
$\tau_n$  towards  a  tree  with  finite  nodes  is  equivalent  to  the
convergence of $ \P(r_{h}(\tau_n)=\bt)$ for all $h\in \N^*$ and all tree
$\bt$ of  height $h$, up to  the identification of the  limit, the local
convergence can be deduced from the  convergence as $n$ goes to infinity
of $H_n(h,k)$ for all $h,k\in \N^*$. The result is in the same spirit as
the strong ratio theorem for random walks. Notice that  all the regimes
described in the following theorem are valid thanks to Remark
\ref{rem:deftn}. 

\begin{theo}
   \label{theo:srt}
Let  $p$ be  a non-degenerate  super-critical offspring  distribution
  with        $\mu$           finite.  We assume that the
  sequence $(a_n, n\in \N^*)$ is such that $\P(Z_n=a_n)>0$  for
  all $n\in \N ^*$. 
\begin{itemize}
\item  \textbf{Extinction  case:}  $a_n=0$  for   all
  $n\ge n_0$ for some $n_0\in\N^*$.  If  $\fc=0$, then $\P(Z_n=0)=0$ for  all $n\in \N$,
  and thus $H_n$ is not defined.  If $\fc>0$,  then  we
  have:
\begin{equation}
   \label{eq:H00}
H^{0,0}(h,k):=\lim_{n\rightarrow\infty } H_n(h,k)= \fc^{k-1}.
\end{equation}

\item \textbf{Low regime:}  $\lim_{n\rightarrow\infty } a_n  /c_n=0$ and
  $a_n>0$ for  all $n\in \N^*$. We  have\footnote{Notice that $\fa=0$, resp.
    $\fa=1$,  resp.   $\fa\geq  2$,  is equivalent  to  $\fc>0$,  resp.
    $\fc=0$ and $f'(\fc)>0$, resp. $f'(\fc)=0$.}:
\begin{equation}
   \label{eq:H0}
H^{0}(h,k):=\lim_{n\rightarrow\infty } H_n(h,k)=
\begin{cases}
 k\fc^{k-1} f'(\fc)^{-h} & \text{ if $\fa=0$},\\
   f'(\fc)^{-h} \ind_{\{k=1\}}  & \text{ if $\fa=1$},\\
p(\fa)^{-(\fa^{h} 
  -1)/(\fa-1)}\ind_{\{k=\fa^h\}} & \text{ if $\fa\geq 2$}.\\
\end{cases}
\end{equation}
\item        \textbf{Moderate        regime:}
  $\lim_{n\rightarrow\infty  } a_n  /c_n=\theta\in  (0,  +\infty )$.     
We  have, with the notation $w_k(\theta)=\sum_{i=1}^k \binom{k}{i} \fc^{k-i}
w^{*i}(\theta)$:
\begin{equation}
   \label{eq:Hq}
H^\theta(h,k):=\lim_{n\rightarrow\infty } H_n(h,k)=
\mu^h \, \frac{
  w_k\left(\mu^h \theta \right)}{w(\theta)}\ind_{\{k=r_0^h (\mod L_0)\}},
\end{equation}
where $(L_0, r_0)$ is the type of $p$. 
\item         \textbf{High         regime:}  $\lim_{n\rightarrow\infty  }
  a_n  /c_n=+\infty  $.  (Partial results.)  We have: 
\begin{equation}
   \label{eq:Hi}
H^\infty (h,k):=\lim_{n\rightarrow\infty } H_n(h,k)=
\begin{cases}
p(\fb)^{-(\fb^{h} 
  -1)/(\fb-1)}\ind_{\{k=\fb^h\}} & \text{ if $\fb<\infty $,}\\
0 & \text{ if $p$ is geometric}.
\end{cases}
\end{equation}
\end{itemize}
\end{theo}

Contrary to the short  proof of the  strong ratio theorem for
random walks given by  Neveu \cite{n:tece}, the proof presented here for
the strong ratio theorem   rely on explicit
equivalent of $\P_{k}(Z_{n-h}=a_n)$ for $n$ large. The well known
extinction  case is given in  Remark \ref{rem:cond-onE}. 

The  result  for  the  low  regime  is  much  more  delicate.  We  shall
distinguish between the Schröder case  $f'(\fc)>0$ and the Böttcher case
$f'(\fc)=0$,   and   in   those   two  cases   consider   the   sequence
$(a_n,  n\in\N^*)$ bounded  or unbounded.   The case  $a_n$ bounded  and
$\fa=0$  can be  found  in Papangelou  \cite{p:lgwpsc}.  The case  $a_n$
bounded, $\fa=1$ is an easy extension  of \cite{p:lgwpsc}, see Case I in
the  proof of  Proposition  \ref{prop:cv-not-fat-gene}  in the  Schröder
case.  The case  $a_n$ unbounded and $\fa\leq 1$ (Schröder  case) can be
derived, see Lemma \ref{lem:srt-schroder},  from the precise asymptotics
of  $\P_\ell(Z_n=a_n)$ given  by Fleischmann  and Wachtel  \cite{fw:ld}.
The   case   $\fa\geq   2$   (Böttcher  case)   is   given   in   Lemma
\ref{lem:rapp-Za}  and Lemma  \ref{lem:srt-bottcher}.  The former  lemma
relies on a  precise approximation of $\P_\ell(Z_n=a_n)$  given in Lemma
\ref{lem:bott-Pz}  for $a_n$  unbounded, which  is an  extension of  the
precise asymptotics given by Fleischmann and Wachtel \cite{fw:lta}.

The moderate regime  is a direct consequence of the local limit theorem in
Dubuc and  Seneta  \cite{ds:lltgwp}, see Lemma \ref{lem:srt-moderate} here. 

The      high     regime      in      the      Harris     case      when
$\limsup_{n\rightarrow\infty  }   a_n/\fb^n<1$  is  detailed   in  Lemma
\ref{lem:rapp-Zb}  with $\ell=1$.   It relies  on techniques  similar to
those developed in \cite{fw:lta} or in Flajolet and Odlyzko \cite{fo:ld}
to get an  equivalent to $\P_k(Z_n=a_n)$, see  Lemma \ref{lem:H-Pz}. The
proof  is  however  given  in  details because  the  adaptation  is  not
straightforward.    The  high   regime  for   the  geometric   offspring
distribution is given in \cite{abd:llfgt}.  \medskip

If $\fb=\infty  $ and  $f(R_c)=+\infty $, we  conjecture that  $\tau_n$ converges
locally in distribution towards a limit $\tau^\infty $ whose root has an
infinite number of children. Using the elementary identity \reff{eq:ph},
we deduce  the following conjecture  that if $\fb=\infty $  and $f(R_c)=+\infty $,
then:
 \begin{equation}
   \label{eq:Hii}
H^\infty (h,k):=\lim_{n\rightarrow\infty } H_n(h,k)=0.
\end{equation}
If  $\fb=+\infty $  and $f(R_c)<+\infty  $, then  $\tau^\infty $  has no
condensation  and thus  $H^\infty(n,k) $  might exists  and be  given by
$f_{-h+1}(R_c)^k/f(R_c)$,  where, for  $n\in  \N^*$,  $f_n$ denotes  the
$n$-th iterate  of $f$ and $f_{-n}$  its inverse (which is  well defined
because $f_n$ is increasing).  See the martingale term in the right hand
side of \reff{eq:t-t=mart-t-gen} with $\lambda=\lambda_c$.

If  $R_c=1$, the  possible existence  of a  limit for  $H_n$ is  an open
question.  See  Wachtel, Denisov and Korshunov  \cite{wdk:tasgwph} for a
first step in the study of this so-called heavy-tailed case.

\subsection{Link with the Martin boundary of  super-critical  GW
  process} 
\label{sec:extreme}
Recall  that $Z$  is  a super-critical  GW  process with  non-degenerate
offspring distribution $p$ with finite  mean $\mu$.  The Martin boundary
$\cm$  of the  non-negative  space-time GW  process  corresponds to  all
extremal  non-negative  space-time  harmonic  functions  $H$  defined  on
$\N^2$,  and  is  related  to  the  set  of  all  extremal  non-negative
martingales  $N=(N_n=H(n,Z_n), n\in  \N)$.   Considering  only the  case
$Z_0=1$, then Remark \ref{rem:deftn} implies  that the functions $H$ are
only  defined  for  $(n,k)$  such   that  $k=r_0^n  (\mod  L_0)$,  where
$(L_0,  r_0)$  is  the  type  of  $p$.  Let  $\ch$  denote  the  set  of
non-negative space-time function  $H$ such that there  exists a sequence
$(a_n,                 n\in                  \N^*)$                 with
$H(h,k)=\lim_{n\rightarrow\infty }\P_k(Z_{n-h}=a_n)/\P(Z_n=a_n)$ for all
$h,k\in  \N$.   According  to  Kemeny, Snell  and  Knapp  \cite{ksk:dmc}
Chapter 10, we have $\cm\subset \ch$.

Consider the  collection $\ch^*=\{H^\theta,  \theta\in [0,  \infty )\}$.
We deduce  from Section  \ref{sec:srthm} that $\ch^*\subset  \ch$.  This
appears already  in Athreya  and Ney  \cite{an:lltra}, see  also Section
II.9 from  Athreya and  Ney \cite{an:bp}.  We  also deduce  from Section
\ref{sec:srthm} that $H^{0, 0}\in \ch$ if  and only if $\fa=0$. We get a
complete description of $\ch$ and $\cm$ in the Harris case and geometric
case.  To our knowledge, the results  for the Harris case in the present
work  and for  the  geometric  case in  \cite{abd:llfgt}  are the  first
complete  descriptions  of the  Martin  boundary  for super-critical  GW
process.   This (partially) answers  a  question raised  in  \cite{an:lltra}, on  the
identification of $\ch \backslash \ch^*$.

\begin{theo}
   \label{theo:martin}
Let  $p$ be  a non-degenerate  super-critical offspring  distribution
  with        $\mu$           finite.  If $\fb<\infty $, then we have:
\[
\cm=\ch=\begin{cases}
   \ch^*\cup \{H^{0, 0}, H^\infty \} & \text{ if $\fa=0$,}\\
\ch^*\cup \{ H^\infty \} &\text{ if $\fa\geq 1$.}
\end{cases}
\]
If $p$ is geometric, then we have $\cm=\ch=\ch^*\cup \{H^{0, 0}\}$ if $\fa=0$ and
$\cm=\ch=\ch^*$ if $\fa>0$.
\end{theo}
In the  previous theorem, the description  of $\ch$ is a  consequence of
Theorem \ref{theo:srt}; the equality $\cm=\ch$ follows directly from the
fact    that   for    all   $\theta\in    (0,   +\infty    )   $    a.s.
$\lim_{n\rightarrow\infty  } z_n  (\tau^\theta)/c_n=\theta$, see  Remark
\ref{rem:decomp-E}  or  Lootgieter  \cite{l:sagw}, Corollary  2.3.II  c)
which states  that all the functions  in $\ch^*$ are extremal  under the
$L \log(L)$ condition and the  aperiodic condition, that is $L_0=1$.  In
the  same   spirit,  Overbeck   \cite{o:mbbp}  has  given   an  explicit
description of  the Martin boundary  for some time-continuous branching
processes, see for example Theorem 2 therein.  \medskip

We conjecture that $\ch=\ch^*$ or  $\ch= \ch^*\cup \{H^{0, 0}\}$ as soon as
$\fb=+\infty $ and $f(R_c)=+\infty $, keeping  $H^{0, 0}$ if and only if
$\fa=0$.    Otherwise,   existence   of  a  limit   function   $H$   when
$\lim_{n\rightarrow\infty  }  a_n/c_n=+\infty $  is  still  open in  the
general case.  \medskip

We end  this section with some  works related to Martin  boundary for GW
process. We  refer to  Dynkin \cite{d:btmp} or  to \cite{ksk:dmc}  for a
presentation  of  the Martin  boundary.  For  the extremal  non-negative
harmonic functions (space only) of GW  process, we refer to Theorem 3 in
Cohn \cite{c:hfcmc}, which is stated under the $L \log(L)$ condition and
an  aperiodic condition.   (Notice that  the $L  \log(L)$ and  aperiodic
conditions  are   indeed  required  in   the  proof  of  Theorem   3  in
\cite{c:hfcmc} as it relies on  Corollary 2.3.II a) from \cite{l:sagw}.)
For  the  Martin entrance  boundary  of  GW  process, see  Alsmeyer  and
R\"osler \cite{ar:mebgwp}.

\subsection{The geometric offspring distribution case}
   \label{sec:geom}
We consider the geometric super-critical offspring distribution. We
collect results   developed  in  \cite{abd:llfgt}  and in this paper. 
\medskip

Let $0<q<\eta\leq  1$ and  define the $\cg(\eta,q)$  geometric offspring
distribution  by 
\[
\begin{cases}
p(0)=1-\eta,\\
p(k)  =  \eta q  (1-q) ^{k-1} & \mbox{for }\in  \N^*.
\end{cases}
\]
We  have $\fa=0$  if  $\eta<1$ and  $\fa=1$ if  $\eta=1$. Moreover, we have
$\fb=+\infty  $,  $(L_0,  r_0)=(1,0)$, $\mu=\eta/q\in  (1,  +\infty  )$. It is easy to compute
\[
f(s)=\frac{(1-\eta)   -    s   (1-q-\eta)}{1-    s(1-q)},
\]
and deduce that   $R_c=1/(1-q)$,
$f(R_c)=+\infty $, $\fc=(1-\eta)/(1-q)\in  [0, 1)$ and $f'(\fc)=q/\eta$.
It is also easy to check that 
\[
w(\theta)=(1-\fc)^2 \expp{-(1-\fc)\theta}
\]
 for
$\theta>0$                            and                           thus
$\lambda_c= \sup\{\lambda \in  \R; \, \E[\exp (\lambda  W)] <+\infty \}=
1-                                                               \fc>0$.
If $\fc>0$  or equivalently $\eta<1$,  then $\tau^{0, 0}$  has geometric
offspring     distribution      $\cg(q,\eta)$.      We      have     for
$\theta\in (0, +\infty )$, $r\in \N^*$:
\[
\rho_{\theta, r}(s)=\frac{(r-1)!}{(|s|_1-1)!} \, 
\Big(\theta(1-\fc)(\mu-1)\Big)^{|s|_1-r} \expp{ -\theta
  (1-\fc)(\mu-1) }, \quad s\in (\N^*)^r,
\]
with $|s|_1=\sum_{i=1}^r s_i$ for $s=(s_1, \ldots, s_r)$; and 
\[
H^{\theta}(h,k)=
\mu^h \expp{-\theta(1-\fc)(\mu^h -1)} \sum_{i=1}^k \binom{k}{i}
\fc^{k-i} \frac{\Big(\theta (1-\fc)^2\mu^h \Big)^{i-1}}{(i-1)!}\cdot
\]
Notice that the definition of $H^\theta$  is similar  to the extremal space-time
harmonic functions  given in Theorem 2  from \cite{o:mbbp} for binary
splitting in continuous time. 
\medskip

We have  that $(\tau_n,  n\in \N^*)$  converges locally  in distribution
towards                         $\tau^\theta$                         if
$\lim_{n\rightarrow  \infty }  a_n/\mu^n=\theta\in  [0,  +\infty ]$  and
$a_n>0$      for       all      $n\in      \N^*$.        The      family
$(\tau^\theta, \theta\in  [0, +\infty ])$ is  continuous in distribution
for the local convergence. The random tree
$\tau^\infty $ has only one node  of infinite degree which happens to be
the root.  The space-time Martin  boundary is $\cm=\ch=\ch^*$ if $\fc=0$
and $\cm=\ch=\ch^*\cup \{H^{0, 0}\}$ if $\fc>0$.

\subsection{Organization of the paper}
We recall  the definition of trees, the local convergence and the distribution of the Galton-Watson
tree $\tau$ in Section \ref{sec:not}. Section \ref{sec:Kesten} is
devoted to the Kesten tree associated with $\tau$. We introduce in Section
\ref{sec:gen-gen} a probability distribution $\rho_{\theta, r}$ in
\reff{eq:def-rho} which plays a crucial role to describe the local
limits in the moderate regime. We present the local limits in the
moderate regime in Section \ref{sec:extremal}. The statements of the
local convergence are in Section \ref{sec:loc-cv}. The continuity of the
local limits is studied in Section \ref{sec:cont-0} and the partial
results on the continuity at $\theta=+\infty $ are presented in Section
\ref{sec:tinfini}. Section \ref{sec:sub-critical} is devoted to the
sub-critical case (when it is seen as the super-critical case
conditioned to the extinction event). After some ancillary results given
in Section \ref{sec:anc}, we give detailed proofs in the
technical 
Section  \ref{sec:Harris} for the Harris case and state the
 results  for the Böttcher case in Section \ref{sec:bottcher}.

\section{Notations}
\label{sec:not}
We denote by $\Nz=\{0,1,2,\ldots\}$ the set of non-negative integers, by
$\Np=\{1,2,\ldots\}$    the    set     of    positive    integers    and
$\Ni=\Nz\cup\{+\infty\}$.   For  any  finite   set  $E$,  we  denote  by
$\sharp E$ its cardinal.

We say that a function $g$ defined on $(0, +\infty )$ is  multiplicatively
periodic  with period $c>0$  if $g(c x)=g(x) $ for
all $x>0$. Notice that $g$ is also multiplicatively
periodic  with period $1/c$. 

\subsection{The set of discrete trees}
\label{sec:tree}

We recall Neveu's formalism \cite{n:apghw} for ordered rooted trees. Let
$\cu=\bigcup_{n\geq 0}(\Np)^{n}$ be the  set of finite sequences
of positive integers with the convention $(\Np)^{0}=\{\partial\}$. We
also 
set $\cu^*=\bigcup_{n\geq 1}(\Np)^{n}= \cu\backslash\{\partial\}$. 

For $u\in \cu$, let $|u|$  be the  length or the generation  of $u$
defined as the  integer $n$ such that $u \in(\Np)^{n}$.   If $u$ and $v$
are two sequences of $\cu$,  we denote by $uv$ the concatenation
of two sequences,  with the convention that $uv=vu=u$  if $v=\partial$.
The set of strict ancestors of $u\in  \cu^*$ is
defined by
\[
\anc(u)=\{v \in \mathcal{U},\ \exists  w \in \mathcal{U}^*,\  u=vw\}, 
\]
and    for    $\cs\subset    \cu^*$,    being    non-empty,    we    set
$\anc(\cs)=\bigcup _{u\in \cs} \anc(u)$.

A tree $\bt$ is a subset of $\mathcal{U}$ that satisfies :
\begin{itemize}
\item  $\partial \in \bt$.
\item  If $u\in \bt\backslash\{\partial\}$, then $\anc(u)\subset \bt$.
\item  For every $u\in \bt$, there exists $k_{u}(\bt)\in\bar\Nz$ such
  that, for every $i\in \N ^*$, $ui \in \bt \iff 1\leq i\leq k_{u}(\bt)$.
\end{itemize}

We denote by $\T_\infty $ the set of trees.  For $r\in \bar \N$, $r\geq
1$, we denote by $\bt_r$ the regular $r$-ary tree, defined by
$k_u(\bt_r)=r$ for all $u\in \bt_r$. Let $\bt \in \T_\infty $ be
a tree. The vertex $\partial$ is called  the root of the tree $\bt$ and
we  denote by  $\bt^*=\bt\backslash\{\partial\}$ the  tree without  its
root.  For a  vertex $u\in\bt$, the integer  $k_{u}(\bt)$ represents the
number  of  offsprings  (also  called   the  out-degree)  of  the  vertex
$u  \in   \bt$.  By   convention,  we   shall  write   $k_u(\bt)=-1$  if
$u\not\in \bt$.  The height $H(\bt)$ of the tree $\bt$ is defined by:
\[
 H(\bt) =  \sup \{|u|,\ u \in \bt\} \in \Ni.
\]
For $n\in\Nz$, the size of the $n$-th generation of $\bt$ is defined by:
\[
z_{n}(\bt)=\sharp\{u \in \bt ,\vert u\vert=n\}.
\]

We denote by $\Tf^*$ the subset of trees with finite out-degrees except the root's:
\[
\Tf^{*} = \{\bt \in \T_\infty ;\,  \forall u\in\bt^*,\ k_u(\bt) < + \infty
\}
\]
and by $\Tf= \{\bt \in \Tf^*;\,   k_\partial(\bt) < + \infty \}$ the
subset of trees with finite out-degrees. 

Let   $h,k\in\Np$. We define $\Tf^{(h)}$ the subset of finite trees with
height $h$: 
\[
\Tf^{(h)}=\{\bt \in \Tf;\,  H(\bt)= h \}
\]
and $\T^{(h)}_{k}= \{\bt \in \Tf^{(h)};\, k_\partial(\bt)= k\}$ the subset of finite trees with height   $h$
and out-degree of the root equal to  $k$. 
The restriction operators $r_h$ and $r_{h,k}$ are defined, for
every $\bt\in\T_\infty $, by:
\[
r_h(\bt)  =\{ u\in\bt;\ |u|\le h\}
\quad\text{and}\quad 
r_{h,k}(\bt)  =\{\partial\} \cup  \{ u\in r_h(\bt) ^*; \, \anc(u)
               \cap \{1, \ldots, k\} \neq \emptyset\},
\]
so that, for  $\bt\in\Tf$,
if $H(\bt)\geq h$, then $r_h(\bt)\in \Tf^{(h)}$; and 
for  $\bt\in\Tf^*$,
if $H(\bt)\geq h$ and 
$k_\partial(\bt)\geq k$, then $r_{h,k}(\bt)\in\T^{(h)}_{k}$.

\subsection{Convergence of trees}
\label{sec:cv-of-t}
Set $  \N_{1}=\{-1\}\cup \Ni $, endowed  with the usual topology  of the
one-point compactification of the  discrete space $\{-1\}\cup \Nz$.  For
a  tree  $\bt\in\T_\infty $,  recall that  by  convention  the  out-degree
$k_u(\bt)$ of $u$ is  set to -1 if $u$ does not belong  to $\bt$. Thus a
tree  $\bt\in\T_\infty   $  is  uniquely  determined   by  the
$\N_1$-valued sequence
$(k_u(\bt),   u\in\cu)$  and   then  $\T_\infty   $  is   a  subset   of
$\N_{1}^{\cu}$.  By Tychonoff  theorem, the  set $\N_{1}^{\cu}$  endowed
with the product topology is compact. Since $\T_\infty $ is closed it is
thus compact. In fact, the set $\T_\infty$ is a Polish space (but we don't need
any precise metric at this point). The local convergence of sequences of trees
is then  characterized as follows. Let  $(\bt_n, n\in \N)$ and  $\bt$ be
trees in $\T_\infty $. We say that $\lim_{n\rightarrow\infty } \bt_n=\bt$
if and only if  $\lim_{n\rightarrow\infty }
k_u(\bt_n)=k_u(\bt)$ for all $u\in \cu$. 
It is easy to see that:
\begin{itemize}
   \item  If  $(\bt_n, n\in \N)$ and $\bt$ are trees  in
$\Tf$, then we have $\lim_{n\rightarrow\infty } \bt_n=\bt$ if and only if 
$\lim_{n\rightarrow\infty } r_h(\bt_n)=r_h(\bt)$ for all $h\in
  \N^*$.
\item  If  $(\bt_n, n\in \N)$ and $\bt$ are trees  in
$\Tf^*$, then we have $\lim_{n\rightarrow\infty } \bt_n=\bt$ if and only if 
$\lim_{n\rightarrow\infty } r_{h,k}(\bt_n)=r_{h,k}(\bt)$ for all $h,k\in
  \N^*$.
 \end{itemize}

 If  $T$ is   a  $\Tf$-valued (resp.   $\Tf^*$-valued) random  variable,
 then its   distribution    is
 characterized                                                        by
 $\left(\P(r_h(T)=\bt);  \,  h\in  \N^*,  \,  \bt  \in  \Tf^{(h)}\right)$
 (resp.
 $\left(\P(r_{h,k}(T)=\bt);    \,    h,k\in     \N^*,    \,    \bt    \in
   \T_{k}^{(h)}\right)$).
 Using the Portmanteau theorem, we deduce the following characterization of  the convergence  in
  distribution:
\begin{itemize}
\item Let  $(T_n, n\in  \N)$ and $T$  be $\Tf$-valued  random variables.
  Then, if a.s. $H(T)=+\infty $,   we have:
\begin{equation}
   \label{eq:cv-loi}
T_n\; \xrightarrow[n\rightarrow \infty ]{\textbf{(d)}}\; T
\iff 
\lim_{n\rightarrow\infty } \P(r_h(T_n)=\bt)=\P(r_h(T)=\bt) 
\quad \text{for all $h\in \N^*, \, \bt\in \Tf^{(h)}$}.
\end{equation}
   \item  Let  $(T_n, n\in \N)$ and $T$ be $\Tf^*$-valued random
     variables.  Then,  if
     a.s. $H(T)=+\infty $ and  $k_\partial(T)=+\infty $, we have: 
\begin{equation}
   \label{eq:cv-loi*}
T_n\; \xrightarrow[n\rightarrow \infty ]{\textbf{(d)}}\; T
\iff
\lim_{n\rightarrow\infty } \P(r_{h,k}(T_n)=\bt)=\P(r_{h,k}(T)=\bt) 
\quad \text{for all $h,k\in \N^*, \, \bt\in \T_{k}^{(h)}$}.
\end{equation}
\end{itemize}

\subsection{Galton-Watson trees}
\label{sec:GW}
Let $p=(p(n), n\in \Nz)$ be a probability distribution on $\Nz$. A
$\Tf$-valued random variable $\tau$ is called a GW tree with
offspring distribution $p$ if for all $h\in \Np$ and $\bt\in \Tf$ with
$H(\bt)\leq h$:
\[
\P(r_h(\tau)=\bt)=\prod_{u\in r_{h-1}(\bt)} p(k_u(\bt)).
\]
The generation size process defined  by $(Z_n=z_n(\tau), \, n\in \N)$ is
the so-called GW process. We  refer to \cite{an:bp} and \cite{ah:bp} for
a general study of GW processes.
\medskip 

We recall here the classical result on the extinction probability of the
GW tree  and introduce some notations. We denote by
$\ce=\{H(\tau)<+\infty\}=\bigcup _{n\in \N} \{ Z_n=0\}$ the extinction
event and denote by $\fc$ the extinction probability:
\begin{equation}\label{eq:def-c}
\fc=\P(\ce).
\end{equation} 
Then, if $\ff$ denotes the generating function of $p$, $\fc$ is the
smallest non-negative root of $\ff(s)=s$. We denote by $\mu$ the mean
of $p$ i.e. $\mu=\ff'(1)$. We recall the  three following cases:
\begin{itemize}
\item The sub-critical case ($\mu<1$):  $\fc=1$.
\item The critical case ($\mu=1$):  $\fc=1$ (unless $p(1)=1$ and then $\fc=0$).
\item The super-critical case ($\mu>1$):  $\fc\in [0, 1)$, the process has a
  positive probability  of non-extinction. Notice 
  that $\fc=0$ if and only if $\fa\geq 1$.
\end{itemize}
We  consider the lower and upper bounds of the support of $p$:
 \begin{equation}
   \label{eq:def-ab}
\fa=\inf\{n\in \N; \, p(n)>0\}
\quad\text{and}\quad
\fb=\sup  \{k;  \,  p(k)>0\}\in  \bar  \N.
\end{equation} 
We say that $p$ is non-degenerate if  $\fa<\fb$. We define $f_n$  the $n$-th
iterate of  $f$, which is  the generating  function of $Z_n$.  We recall
that $\lim_{n\rightarrow\infty } f_n(0)=\fc$.   We also introduce in the
supercritical  case ($\mu>1$)   the  Schröder
constant $\alpha$  defined by:
\begin{equation}
   \label{eq:def-alpha}
f'(\fc)=\mu^{-\alpha}, \quad \alpha\in (0, +\infty ].
 \end{equation}

 We set $\P_k$ the probability under which the GW process $(Z_n, n\ge 0)$ starts with $Z_0=k$ individuals and write $\P$ for $\P_1$ so that:
 \[
 \P_k(Z_n=a)=\P(Z_n^{(1)}+\cdots+Z_n^{(k)}=a),
 \]
 where the $(Z^{(i)},1\le i\le k)$ are independent random variables
 distributed as  $Z$ under $\P$.
 
We consider a sequence $(a_n, n\in \Np)$ of elements of
 $ \N$ and, when $\P(Z_n=a_n)>0$, 
$\tau_n$    a  random  tree   distributed  as  the  GW   tree  $\tau$
conditioned on $\{Z_n= a_n\}$.
Let $n\geq h\geq 1$ and $\bt\in \Tf^{(h)}$. We have, with $k=z_h(\bt)$:
\begin{equation}
   \label{eq:ph}
  \P(r_{h}(\tau_n)=\bt)
= \P(r_h(\tau)=\bt)\frac{\P_k(Z_{n-h}=a_n)}{\P(Z_n=a_n)} \cdot
\end{equation}

\section{The Kesten tree}
\label{sec:Kesten}
In  this  section,   we  consider  a  GW  tree   $\tau$  with  offspring
distribution  $p=(p(n), n\in  \Nz)$ having  mean $\mu\in  (0, +\infty  )$.
Recall  that  $\fc\in [0,  1]$  denotes  the extinction  probability  of
$\tau$. We define  an associated probability distribution  $\fp$ on $\N$
as follows:

\begin{defi}\label{def:def-fp}
\begin{itemize}
   \item[(i)] If $\fc=0$, we define $\fp$ as  the Dirac mass at point $\fa$. 
\item[(ii)]  If $\fc>0$, we  define  the   probability  distribution
$\fp=(\fp(n), n\in \Nz)$ by:
\begin{equation}
   \label{eq:def-fp}
\fp(n)=\fc^{n-1} p(n) \quad\text{for $n\in \Nz$}.
\end{equation}
\end{itemize}
\end{defi}

We denote  by $\fm$  the mean  of $\fp$.  If $\mu\leq  1$ and $p(1)\ne 1$, as $\fc=1$, we have
$\fp=p$ and $\fm=\mu$.    If    $\fc>0$,   we    have
$\fm=\ff'(\fc)\in (0,1]$. 

\begin{rem}
   \label{rem:cond-onE}
   If $\fc>0$, let $\tau^{0,0}$ be a GW tree with offspring distribution
   $\fp$ defined in \reff{eq:def-fp}.  It is well known that the GW tree
   $\tau$ conditioned  on the extinction  event $\ce$ is  distributed as
   $\tau^{0,0}$.  Indeed, we have using  the branching property that, for
   $h\in   \N^*$,  $\bt\in \Tf^{(h)}$,  and setting
   $k=z_h(\bt)$:
\[
\P(r_h(\tau)=\bt|\, \ce)
= \P(r_h(\tau)=\bt) \frac{\P_k(\ce)}{\P(\ce)}= \fc^{k-1}
\P(r_h(\tau)=\bt)=\P(r_h(\tau^{0,0})=\bt). 
\]
\end{rem}

Let $k\in \Np$. If $\ff^{(k)}(1)\in (0, +\infty )$, that is   $p$ has
finite moment of order $k$ and the support of $p$ is not a subset of $\{0,
\ldots, k-1\}$,  then we define the $k$-th order size-biased
probability distribution of $p$ as $p_{[k]}=(p_{[k]}(n), n\in \Nz)$ with:
\begin{equation}
   \label{eq:def-biased-p}
p_{[k]}(n)=  \ind_{\{n\geq k\}}\, \binom{n} {k}\frac{k!}{\ff^{(k)}(1)}\, p(n) .
\end{equation}
The generating function of $p_{[k]}$ is
$\ff_{[k]}(s)=s^k\ff^{(k)}(s)/ \ff^{(k)}(1)$. 
The probability distribution $p_{[1]}$  is the so-called size-biased
probability distribution  of $p$. 

\medskip
We now  define the so-called  Kesten tree $\hat \tau^0$  associated with
the  offspring distribution  $p$.   
\begin{defi}[Kesten tree]
\label{def:Kesten}
\begin{itemize}
\item[(i)] If $\fc>0$, the
Kesten tree $\hat\tau^0$ is a two-type  GW tree where the  vertices are either  of type
$\rs$  (for survivor)  or of  type $\re$  (for extinction).  Its distribution is characterized as follows.
\begin{itemize}
 \item The root is of type $\rs$.
  \item The number of offsprings of a vertex depends, conditionally on the
  vertices of lower or same height, only on its own type (branching property).
   \item A vertex of type $\re$ produces only vertices of type $\re$ with
     offspring distribution $\fp$. 
   \item The random number  of children of a vertex of  type $\rs$ has the
     size-biased distribution  of $\fp$  that is $\fp_{[1]}$  defined by
     \reff{eq:def-biased-p}  with  $k=1$.  (Notice that $\fp_{[1]}$  is  well
     defined as $\fc>0$.)  Furthermore, all of the children  are of type
     $\re$  but  one,  uniformly  chosen  at random  which  is  of  type
     $\rs$. 
\end{itemize}
\item[(ii)] If  $\fc=0$, the  (degenerate) Kesten
tree $\hat\tau^0$ is given by $\bt_\fa$ the  regular  $\fa$-ary  tree, with  $\fa\geq  1$  defined  by
\reff{eq:def-ab}. It   can  be  seen  as a  GW  tree  with  degenerate
offspring distribution the Dirac mass at point $\fa$. In this case all
the individuals have type $\rs$. 
\end{itemize}
\end{defi}

Informally, when $\fc>0$,  the
individuals of  type $\rs$ in  $\hat \tau^0$  form an infinite  spine on
which   are  grafted   independent  GW   trees  distributed  (see Remark
\ref{rem:cond-onE}) as  $\tau$
conditionally   on    the   extinction    event   $\ce$. 
\medskip

We define $\tau^0=\ske(\hat \tau^0)$ as  the tree $\hat \tau^0$ when one
forgets the  types of  the vertices.  If $\fc=0$,  then $\tau^0$  is the
regular  $\fa$-ary tree.  If $\fc>0$,  the distribution  of $\tau^0$  is
given in the following classical result.
\begin{lem}
   \label{lem:=distrib}
   Let  $p$ be  an offspring  distribution with  finite positive mean
   such that $\fc>0$. The distribution of $\tau^0$ is
   characterized by: for all   $h\in \N^*$   and 
   $\bt\in  \Tf^{(h)}$  with  $k=z_h(\bt)$:
\begin{equation}
   \label{eq:ph0}
  \P(r_{h}(\tau^0)=\bt)
= k \fc^{k-1} \fm^{-h} \, \P(r_h(\tau)=\bt).
\end{equation}
\end{lem}
If $\mu\le 1$, this is the usual link between Kesten tree and the size-biased GW tree. If $\mu>1$, the lemma just means that the Kesten tree is the sized biased tree associated with the tree conditioned on extinction (which is the subcritical GW tree with offspring distribution $\fp$).
We give a short proof of this  well-known result.
\begin{proof}
  According to  Section \ref{sec:cv-of-t}, the distribution  of $\tau^0$
  is   characterized  by   \reff{eq:ph0}   for  all   $h\in  \N^*$   and
  $\bt\in \Tf^{(h)}$ with $k=z_h(\bt)$.

Let $h\in \N^*$,  $\bt\in \Tf^{(h)}$ and
   $v\in \bt$ such that $|v|=h$. Let
   $V$ be the vertex of type $\rs$ at level $h$ in $\hat \tau^0$. 
We
   have, with $k=z_h(\bt)$:
\begin{align*}
  \P(r_{h}(\tau^0)=\bt, V=v)
  &= \prod_{u \in r_{h-1}(\bt) \backslash \anc(\{v\})} \fp(k_u(\bt))
    \prod_{u \in \anc(\{v\})} \inv{k_u(\bt)} \, \fp_{[1]}(k_u(\bt))\\
  &= \fm^{-h} \fc^{\sum_{u\in r_{h-1}(\bt)} (k_u(\bt) -1)}\,
    \prod_{u\in r_{h-1}(\bt)} p(k_u(\bt))\\
  &= \fm^{-h} \fc^{k-1}\, \P(r_h(\tau)=\bt) ,
\end{align*}
where we  used \reff{eq:def-biased-p} (with $k=1$,  $n=k_u(\bt)$ and $p$
replaced  by $\fp$)  and  \reff{eq:def-fp} (with  $n=k_u(\bt)$) for  the
second equality  and that  $\sum_{u\in r_{h-1}(\bt)}  (k_u(\bt) -1)=k-1$
for the last  one.  Summing over all $v\in \bt$  such that $|v|=h$ gives
the result.
\end{proof}

\section{A probability distribution associated with super-critical GW
  trees}
\label{sec:gen-gen}

In  this section,  we  consider  a super-critical  GW  tree $\tau$  with
non-degenerate offspring  distribution $p=(p(n), n\in \Nz)$  with finite
mean  $\mu\in  (1,  +\infty  )$.   We  recall  that  $\ff$  denotes  the
generating function of $p$ and $\fc$  is the smallest root in $[0,1)$ of
$\ff(s)=s$.   Notice that  $\fa=0$ is  equivalent to  $\fc>0$.  \medskip
Following \cite{s:sgwp}  or \cite{ah:bp},  we consider  the Seneta-Heyde
norming:    $(c_n,    n\in   \Nz)$    is    a    sequence   such    that
$\left(\expp{-Z_n/c_n},   n\in   \Nz\right)$   is   a   martingale   and
$c_0\in  (  -1/\log(\fc),  +\infty  )$.   This  sequence  is  increasing
positive and unbounded.  Furthermore, we have that $\fa<c_{n+1}/c_n<\mu$
for all  $n\in \Nz$ and that  the sequence $(c_{n+1}/c_n, n\in  \Nz)$ is
increasing\footnote{We  provide  a short  proof  of  the fact  that  the
  sequence $(c_{n+1}/c_n, n\in \Nz)$ is  increasing, as we didn't find a
  reference.  Define  $g_1(\lambda)=\log(f(\expp{-\lambda}))/\lambda$ so
  that $g_1(1/c_{n+1})=-c_{n+1}/c_{n}$.   So to prove that  the sequence
  $(c_{n+1}/c_n, n\in  \Nz)$ is increasing,  it is enough to  check that
  $g_1$   is   increasing,  or   more   generally   that  the   function
  $g_2(\lambda)=\log   (\E[\expp{-\lambda   X}])/\lambda$  defined   for
  $\lambda>0$ is  increasing, where  $X$ is  a non  constant real-valued
  random  variable  with  finite  Laplace transform.   Indeed,  we  have
  $g_2'(\lambda)>0$                                                   as
  $\E[Y\expp{-Y}] +  \E[\expp{-Y}]\log(\E[\expp{-Y}])<0$ for  any random
  variable $Y$  such that $Y\expp{-Y}$  is integrable, thanks  to Jensen
  inequality with  the strictly concave  function $-x \log(x)$  applied to
  $\expp{-Y}$.}   and  converges  towards  $\mu$.   We  also  have  that
$(Z_n/c_n,  n\in \Nz)$  converges  a.s.  towards  a non-negative  random
variable          $W$           with          Laplace          transform
$\cl(\lambda)=\E\left[\expp{-\lambda     W}\right]    $     such    that
$\cl(+\infty )=\P(W=0)=\fc$ and for all $\lambda\geq 0$:
\begin{equation}
   \label{eq:Lap-W}
\ff(\cl(\lambda/\mu))= \cl(\lambda). 
\end{equation}
The probability distribution of $W$, up to a multiplicative constant, is
the unique probability distribution solution of \reff{eq:Lap-W}.

\begin{rem}
   \label{rem:KS}
   If one  assumes that $p$ satisfies  $\E[Z_1\log(Z_1)]<+\infty $, then
   Kesten and  Stigum results  asserts that  $(\mu^{-n} Z_n,  n\in \Nz)$
   converges a.s. towards $W$ up to a scaling factor and that
   $\lim_{n\rightarrow\infty } \mu ^{-n}c_n$ exists and belongs to $(0, +\infty )$. 
\end{rem}

\begin{rem}
   \label{rem:def-lc}
Let $R_c=\sup\{r\geq 1;\, \ff(r)<+\infty \}\geq 1$ be the convergence
radius of the generating function $\ff$ of $p$. Set 
\begin{equation}
   \label{eq:def-ck}
\ck=\{\lambda\in \R; \,
\E[\expp{\lambda W}]<+\infty \},
\end{equation}
and $\lambda_c= \sup \ck\geq 0$.
According to Theorem 8.1 in \cite{l:apsbphbrw} (see also
\cite{r:fptd}), we have that $\lambda_c>0$ if and only if $R_c>1$. 
We then deduce  that \reff{eq:Lap-W} holds for $\lambda\in \C$ such that 
$\fR(\lambda)\in \ck$. We get that
$\ff(R_c)=\varphi(-\lambda_c)\in [1, +\infty ]$ and thus that:
\begin{equation}
   \label{eq:Rc-Lc}
R_c= \varphi(-\lambda_c/\mu).
\end{equation}
\end{rem}
%Utiliser W=sum_{i=1}^Z_i W^i

According to  \cite{ds:lltgwp} and references therein,  the distribution
of $W$  is $\fc \delta_0(dt)  + w(t)\ind_{\{t>0\}}  dt$, where $w$  is a
positive  continuous   function  defined   on  $(0,  +\infty   )$.   
Let
$(W_\ell, \ell\in  \Np)$ be independent random  variables distributed as
$W$.    The     distribution    of    $\sum_{\ell=1}^k     W_\ell$    is
$\fc^k \delta_0 (dt)  + w_k(t) dt $, where (by  decomposing according to
the number $k-i$ of random variables $W_\ell$ which are equal to $0$): 
\begin{equation}
   \label{eq:wk=w}
w_k(\theta)=\sum_{i=1}^k \binom{k}{i} \fc^{k-i}
w^{*i}(\theta)\quad\text{for $\theta>0$},
\end{equation}
and $w^{*i}$ denotes  the $i$-fold convolution of the  function $w$.  We
now define a  new probability distribution related to  the function $w$.
For    $r\in   \Np$,    $s=(s_1,    \ldots,    s_r)\in   (\Np)^r$    and
$\theta\in (0, +\infty )$, we set $|s|_1=\sum_{i=1}^rs_i$ and:
\begin{equation}
   \label{eq:def-rho}
\rho_{\theta, r}(s)=\mu \frac{w^{*|s|_1}(\mu
  \theta)}{w^{*r} (\theta)}  \, \prod_{i=1}^r
\frac{\ff^{(s_i)}(\fc)}{s_i!}\cdot
\end{equation}

\begin{lem}
   \label{lem:r=proba}
Let $p$ be a non-degenerate super-critical offspring distribution with
finite mean. Let $\theta\in (0, +\infty )$, $r\in \Np$. Then $\rho_{\theta,
  r}=(\rho_{\theta, r} (s), \, s \in
(\Np)^r)$ defines a probability distribution  on $(\Np)^r$. 
\end{lem}

\begin{proof}
  For  convenience,  we  shall  prove that  $\rho_{\theta/\mu,r}$  is  a
  probability distribution.   Let $\hat w$ denote  the Laplace transform
  of $w$:  $\hat w(\lambda)=\int_0^\infty w(t)\expp{-\lambda t}  dt$ for
  $\lambda\geq    0$.     We    deduce   from    \reff{eq:Lap-W}    that
  $\ff(\fc  + \hat  w(\lambda))=\fc+\hat w  (\mu\lambda)=\ff(\fc)+\hat w
  (\mu\lambda)$.  We deduce that for $r\in \Np$:
\begin{align*}
\hat w(\mu \lambda)^r   
&= \left( \ff(\fc+\hat w(\lambda)) - \ff(\fc)\right)   ^r\\
&= \sum_{k_1, \ldots, k_r\in \Np} \prod_{i=1}^r p(k_i) \left( (\fc+\hat
  w(\lambda))^{k_i} - \fc^{k_i}\right)\\
&= \sum_{k_1, \ldots, k_r\in \Np} \prod_{i=1}^r p(k_i)
  \sum_{s_i=1}^{k_i} \binom{k_i}{s_i} \fc^{k_i- s_i}
  \hat w(\lambda)^{s_i} \\
&=\sum_{s=(s_1, \ldots, s_r)\in (\Np)^r}  \hat w(\lambda)^{|s|_1}
\prod_{i=1}^r \sum_{k_i=s_i}^{+\infty } \binom{k_i}{s_i}  \fc^{k_i-s_i} p(k_i)\\
&=\sum_{s=(s_1, \ldots, s_r)\in (\Np)^r}  \hat w(\lambda)^{|s|_1}
\prod_{i=1}^r \frac{\ff^{(s_i)}(\fc)}{s_i!},
\end{align*}
where we used for the last equality that for $s\in \Np$, $x\in [0,1]$:
\[
\ff^{(s)}(x)=\sum_{k=s}^{+\infty } \frac{k!}{(k-s)!} x^{k-s} p(k)
= s! \sum_{k=s}^{+\infty } \binom{k}{s} x^{k-s} p(k).
\]
Since   $\hat  w   (\mu  \lambda)^r$   is  the   Laplace  transform   of
$w^{*r}(t/\mu)/\mu$,  by uniqueness  of  the Laplace  transform and  the
continuity  of  $w$  (and  thus  of  $w^{*i}$),  we  get  using the
definition \reff{eq:def-rho} of $\rho_{\theta,r}$ that  for  all
$\theta\in (0, +\infty )$:
\begin{equation}
   \label{eq:wr}
\inv{\mu} w^{*r}(\theta/\mu)
=\sum_{s=(s_1, \ldots, s_r)\in (\Np)^r}  w(\theta)^{*|s|_1}
\prod_{i=1}^r \frac{\ff^{(s_i)}(\fc)}{s_i!}
= \inv{\mu} w^{*r}(\theta/\mu)\sum_{s\in (\Np)^r} \rho_{\theta/\mu,r}(s).
\end{equation}
Since    $w$    is    non-zero,   we get   that $\sum_{s\in
  (\Np)^r} \rho_{\theta/\mu,r}(s)=1$ and thus $\rho_{\theta/\mu,r}$ is a
probability distribution as 
$\rho_{\theta/\mu,r}(s)$ is non-negative.
\end{proof}

We end this section with the limit of $\rho_{\theta,r}$ as $\theta$ goes
to  0 and  in  a  particular case  to  $+\infty  $.  Recall  Definitions
\reff{eq:def-ab} and  \reff{eq:def-alpha}.  One  has to  distinguish two
cases when  $\theta$ goes to 0:  the so-called Schröder  case $\fa\leq 1$
(equivalently $p(0)+p(1)\neq  0$, $f'(\fc)>0$ or  $\alpha<+\infty $) and
the  so-called  Böttcher  case  $\fa\geq 2$ (equivalently  $p(0)+p(1)=0$,
$f'(\fc)=0$  or $\alpha=+\infty $).   When $\theta$ goes to  infinity we
consider the  particular so-called  Harris case  where $p$ has  a finite
support (equivalently $\fb$ is finite).

\begin{lem}
   \label{lem:cv-rho}
   Let  $p$ be  a non-degenerate  super-critical offspring  distribution
   with finite  mean. 
\begin{itemize}
   \item[(i)]  In the  Schröder case  ($\fa\leq 1$), we  get that
   $\rho_{\theta,   1}$   converges   to  the  Dirac mass at point 1  as
   $\theta$ goes  down to 0.  
 \item[(ii)] In  the Böttcher case  ($\fa\geq 2$),  we get that,  for all
   $r\in  \Np$,  $\rho_{\theta,  r}$  converges to  the  Dirac  mass  at
   $(\fa, \ldots, \fa)\in \Nz^r$ as $\theta$ goes down to 0.
  \item[(iii)]  In  the Harris  case  ($\fb<\infty $),  we get that,  for all
   $r\in  \Np$,  $\rho_{\theta,  r}$  converges
   to  the  Dirac  mass  at 
   $(\fb, \ldots, \fb)\in \Nz^r$ as $\theta$ goes to infinity.
\end{itemize}
\end{lem}

\begin{proof}
We give the proof of (i). The technical proofs of (ii) and  (iii) are
postponed respectively to Sections \ref{sec:B-f(c)=0} and
\ref{sec:H-f(c)}. \medskip

  According   to  \cite{bb:ldsbp},  there
exists  a positive  continuous  multiplicatively  periodic function  $V$
defined on $(0, +\infty )$ with period $\mu$ such that for all $x>0$:
\begin{equation}
   \label{eq:w=V}
x^{1-\alpha} w(x)=V(x)+o(1) \quad\text{as $x\searrow 0$.}
\end{equation}
We have for $\theta>0$ as $\theta$ goes down to $0$:
\[
\rho_{\theta, 1}(1)=\mu \ff'(\fc) \frac{w(\mu\theta)}{w(\theta)} 
= \frac{V(\mu\theta) +o(1)}{V(\theta) +o(1)}=1 +o(1),
\]
where we  used Definition \reff{eq:def-alpha} of  the Schröder  constant for  the first
equality and that $V$ has multiplicative period $\mu$ for the last one.  This
implies that $\lim_{\theta\rightarrow 0} \rho_{\theta, 1}(1)=1$ and thus
$\rho_{\theta, 1}$ converges to  the Dirac mass at 1 as $\theta$ goes down to 0.
\end{proof}

\section{Extremal GW trees}
\label{sec:extremal}
We are in the setting of Section \ref{sec:gen-gen}. 
If $\fc>0$, we   define   the   sub-critical   offspring   distribution   $\fp$   by
\reff{eq:def-fp}  and,  see   \reff{eq:def-biased-p},  the  corresponding
size-biased distribution $\fp_{[\ell]}$ of order $\ell\in \Np$. For $\ell\in \Np$ such
that $\ff^{(\ell)}(\fc)>0$, we have:
\begin{equation}
   \label{eq:fp-2}
\fp_{[\ell]}(k)= \binom{k}{\ell} \frac{\ell!}{\ff^{(\ell)}(\fc)}\,
\fc^{k-\ell} p(k), \quad k\geq \ell.
\end{equation}

If $\fc=0$ but  $p(\ell)>0$ (or equivalently $\ff^{(\ell)}(\fc)>0$),
then we define $\fp_{[\ell]}$ as the Dirac mass at point $\ell$, so that
Definition  \reff{eq:fp-2}   is  consistent  for  $\fc\geq   0$.  Recall
Definition  \reff{eq:def-ab}  of  $\fa$  and   note  that 
$\fp=\fp_{[\fa]}$ if $\fc=0$.  \medskip

Let  $\theta\in  (0, +\infty  )$.   We  define  a two-type  random  tree
$\hat   \tau^\theta$  and   shall   consider   the  corresponding   tree
$\tau^\theta=\ske(\hat \tau^\theta)$  when one forgets the  types of the
vertices of $\hat \tau^\theta$.

\begin{defi}[Extremal tree]
\label{def:def-tau-theta} 
Let $p$ be a non-degenerate super-critical offspring distribution with
finite mean. 
The labeled random tree $\hat\tau^\theta$ is a two-type random tree where the  vertices  are  either  of type  $\rs$  (for
  survivor)    or     of    type    $\re$    (for     extinction)    and
  $\tau^\theta=\ske(\hat \tau^\theta)$ denotes  the corresponding random
  $\Tf $-valued tree when one  forgets the labels (or types).  The
  distribution of $\hat\tau^\theta$ is characterized as follows:
\begin{itemize}
   \item The root is of type $\rs$.
\item The number of offsprings of a vertex of type $\re$ does not depend on the
  vertices of lower or same height (branching property  for
  vertices of type $\re$).
\item A vertex   of type $\re$ produces only vertices of type $\re$
  with offspring distribution $\fp$ (as in the Kesten tree).
\item For every $h\ge0$, we set
\[
\cs_h=\{u\in\tau^\theta;\,  |u|=h \text{  and the vertex $u$ has type
  $\rs$ in $\hat \tau^\theta$} \}.
\]
For a vertex $u$ of type $\rs$, we denote by $\kappa^\rs (u)$ the number
of children  of $u$ with type $\rs$ and by $\kappa^\re(u)$ the number of
children of $u$ with type $\re$. Conditionally given
     $r_h(\tau^\theta)$  and  $(\cs_\ell,  0\leq  \ell\leq  h)$, we have:
\begin{itemize}
	\item[(i)] 
     $(    \kappa^\rs    (u),     u\in    \cs_h)$    has    distribution
     $\rho_{\mu^h  \theta ,  \sharp\cs_h}$. 
     \item[(ii)]  For every $u\in\cs_h$, conditionally         on
     $\{\kappa^\rs(v)=s_v\geq 1, v\in \cs_h\}$,  $\kappa^\re(u)$ is such            that
     $k_u(\tau^\theta)=\kappa^\rs(u)+\kappa    ^\re(u)$    has
     distribution $\fp_{[s_u]}$ and the  $s_u$ vertices  of type $\rs$
     are chosen uniformly at  random among the $k_u(\tau^\theta)$
     children. 
\end{itemize}
\end{itemize}
\end{defi}
Notice that Property (i) in the above definition   breaks down the  branching
     property. If $\fc=0$, then  a.s. $\kappa^\re(u)= 0$, so
     that there are no individuals of type $\re$. We    stress   out, and
     shall use later on, that
$\hat  \tau^\theta$  truncated  at  level  $h$  can  be  recovered  from
$r_{h}(\tau^\theta)$ and $\cs_h$ as all the ancestors of a vertex
of type $\rs$ are of type  $\rs$ and a vertex of type $\rs$ has at
least one children of type $\rs$.   
\medskip

Since all the vertices of type $\rs$ have at least one offspring of type
$\rs$, we get    $\sharp\cs_{h+1}\ge\sharp\cs_h$.   The   offspring
distribution of vertices of type $\rs$ can also be described as follows.
For every $h\ge 0$,  conditionally given $r_h(\tau^\theta)$ and $\cs_h$,
we compute the probability that
\begin{itemize}
\item we have $\sharp\cs_{h+1}-\sharp\cs_h=n$ for some $n\ge 0$ i.e. $n$ new vertices of type $\rs$ appear at generation $h+1$,
\item every node $u$ of $\cs_h$ has $k_u$ offspring, $s_u$ of them being
  of type  $\rs$, where the integers  $((s_u,k_u),\, u\in\cs_h)$ satisfy
  $1\le s_u\le k_u$ and $\sum_{u\in\cs_h}s_u= n+\sharp\cs_h$,
\item for every $u\in \cs_h$ and every subset $A_u\subset \{1,\ldots,k_u\}$ such that $\sharp A_u=s_u$, the positions of the offspring of $u$ of type $\rs$ among all the offspring of $u$, are given by $A_u$ i.e. $\cs_{h+1} \cap\{u1, \ldots, uk_u\}=uA_u$ where we recall that $uv$ denotes the concatenation of the two sequences $u$ and $v$.
\end{itemize}    
We    have:
\begin{multline}
\label{eq:pre-ch-general}
\P\left(\forall u\in \cs_h,\, \kappa^\rs(u)+\kappa^\re(u)=k_u \text{ and } \cs_{h+1} \cap\{u1,
    \ldots, uk_u\}=uA_u \,|\,
    r_h(\tau^\theta), \cs_h\right) 
\\
\begin{aligned}
 &  = \rho_{\mu^h \theta,\sharp\cs_h}((s_u, u\in \cs_h))\, \,  \prod_{u\in \cs_h}
 \inv{\binom{k_u}{s_u}}  \fp_{[s_u]}(k_u)\\
 &  = \mu \frac{w^{*(\sharp\cs_h+n)}(\mu^{h+1}
  \theta)}{w^{*\sharp\cs_h} (\mu^h \theta)}  \, \prod_{u\in \cs_h}
\fc^{k_u-s_u} p(k_u),
\end{aligned}
\end{multline}
where we used \reff{eq:def-rho} and \reff{eq:fp-2} for the last equality.

By construction,  a.s.  individuals of  type $\rs$ have a  progeny which
does not  suffer extinction whereas  individuals of type $\re$  (if any)
have a progeny which suffers  extinction.  Since the individuals of type
$\rs$  do   not  satisfy  the   branching  property,  the   random  tree
$\hat \tau^\theta$  is not  a two-type  inhomogeneous GW tree.    
\medskip

Using this definition,  it is easy to get that  the distribution of
the  tree $r_h(\tau^\theta)$  is absolutely  continuous with  respect to
those of the original GW tree $r_h(\tau)$.
\begin{lem}
   \label{lem:=distrib-w}
   Let  $p$ be  a non-degenerate  super-critical offspring  distribution
   with   finite   mean.    Let   $\theta\in  (0,   +\infty   )$.    Let
   $h\in \N^*$ and $\bt\in \Tf^{(h)}$. We have, with $k=z_h(\bt)$:
\[
  \P(r_{h}(\tau^\theta)=\bt)
= \P(r_h(\tau)=\bt)\,  \mu^h \frac{w_k(\mu^h \theta)}{w(\theta)}\cdot
\]
\end{lem}

\begin{proof}
  Let       $h\in        \Np$,       $\bt\in        \Tf^{(h)}$       and
  $S_h\subset \{u\in \bt;  \, |u|=h\}$ be non  empty.  Set $k=z_h(\bt)$.
  In order to shorten the  notations, we set $\ca=S_h\bigcup \anc(S_h)$.
  We      set,      for      $\ell\in     \{0,      \ldots,      h-1\}$,
  $S_\ell=\{u\in \ca, \, |u|=\ell\}$ the  vertices at level $\ell$ which
  have at  least one descendant  in $S_h$.  For $u\in  r_{h-1}(\bt)$, we
  set $s_u(\bt)=\sharp(\ca \bigcap u\Np)$, the number of children of $u$
  having  descendants  in  $S_h$.   We recall  that  $\hat  \tau^\theta$
  truncated at level $h$ can  be recovered from $r_{h}(\tau^\theta)$ and
  $\cs_h$.                           We                          compute
  $\cc_{S_h}=\P(r_{h} (\tau^\theta)=\bt ,\, \cs_h=S_h)$.  We have, using
  \reff{eq:pre-ch-general}:
\begin{align}
\nonumber
   \cc_{S_h}
&= \left[\prod_{u\in r_{h-1}(\bt), \, u\not\in \ca} \fp(k_u(\bt))
  \right]\, \prod_{\ell=0}^{h-1}
\mu \frac{w^{*(\sharp S_{\ell+1})}(\mu^{\ell+1}
  \theta)}{w^{*(\sharp S_h)} (\mu^\ell \theta)}  \, \prod_{u\in S_\ell}
\fc^{k_u(\bt)-s_u(\bt)} p(k_u(\bt))\\
\nonumber
&=\left[\prod_{u\in r_{h-1}(\bt)} p(k_u(\bt))
  \right]\, \left[\prod_{u\in r_{h-1}(\bt)}\!\!\! \fc^{k_u(\bt)-1}
  \right]\left[\prod_{u\in \ca } \fc^{-(s_u(\bt)-1)}
  \right]\,
\mu^h \frac{w^{*(\sharp S_h)}(\mu^h \theta)}{w(\theta)}\\
\label{eq:calcul-Cs}
&=\P(r_h(\tau) =\bt) \,
\fc^ {k - \sharp S_h} \mu^h \, \frac{w^{*(\sharp S_h)}(\mu^h \theta)}{w(\theta)},
\end{align}
where we used that for a tree $\bs$, we have $\sum_{u\in r_{h-1}(\bs)}
k_u(\bs) - 1= z_h(\bs) -1$ and that $\bs=\ca$ is  tree-like with
$z_h(\bs)=\sharp S_h$. 
Remark that $\cc_{S_h}$ depends only of $\sharp S_h$.
Since $\sharp \cs_h\geq 1$ as the
root is of type $\rs$, we obtain:
\begin{multline*}
\P(r_{h} (\tau^\theta)=\bt )
= \sum_{i=1}^k   \sum_{S_h\subset \{u\in \bt; \, |u|=h\}}\,
  \ind_{\{\sharp S_h=i\}}\, 
  \cc_{S_h}\\
= \sum_{i=1}^k   \binom{k}{i} \, \P(r_h(\tau) =\bt) \,
\fc^ {k - i} \mu^h \, \frac{w^{*i}(\mu^h \theta)}{w(\theta)}
=\P(r_h(\tau)=\bt)\, \mu^h \frac{w_k(\mu^h \theta)}{w(\theta)},
\end{multline*}
where we used \reff{eq:wk=w} for the last equality. 
\end{proof}

\begin{rem}
   \label{rem:decomp-E}
Let  $\ce^c=\{W>0\}$ denote the non-extinction event. Using Lemma
\ref{lem:=distrib-w}, we get for $h\in \N^*$, $\bt\in \Tf^{(h)}$, and
$g$ a non-negative measurable function defined on $\R_+$, that:
\[
\int_0^{+\infty}g(\theta)\P(r_h(\tau^\theta)=\bt)\,
w(\theta)d\theta=\E\left[g(W) \ind_{\{r_h(\tau)=\bt, \ce^c\}} \right]
\]
This implies that for every non-negative measurable function $G$ defined
on $\T_\infty\times \R_+ $, we have:
\[
\int_0^{+\infty} \E[G(\tau^\theta, \theta)]\,
w(\theta)d\theta=\E\left[G(\tau,W)\ind_{\{\ce^c\}} \right].
\]
Thus, the distribution probability of $\tau^\theta$ is a regular version
of  the distribution  of  $\tau$ conditionally  on $\{W=\theta\}$.  From
Lemma \ref{lem:=distrib-w},  we get that  this version is  continuous on
$\Tf^{(h)}$ for all $h\in \N^*$.  In particular, we deduce that for a.e.
$\theta\in           (0,            +\infty           )$,           a.s.
$\lim_{n\rightarrow\infty   }  z_n(\tau^\theta)/c_n=\theta$   (see  also
Theorem 2.II in  \cite{l:sagw} for an a.s. convergence for all
$\theta\in (0, +\infty )$  under
stronger hypothesis).   The distribution of  $\tau$ conditionally on
$\ce^c$ can be written as a mixture of distributions of $\tau^\theta$ as
for every Borel set $A$ of $\T_\infty $,
\[
\int_0^{+\infty}\P(\tau^\theta\in A)\, w(\theta)d\theta=\P(\{\tau \in
A\}\cap\ce^c).
\]
\end{rem}

\section{Convergence of conditioned super-critical GW trees}
\label{sec:loc-cv}
We are  in the setting  of Section  \ref{sec:gen-gen}, with $\tau$  a GW
tree with super-critical non-degenerate  offspring distribution $p$ with
finite  mean $\mu$.  We  consider a  deterministic $\N$-valued  sequence
$(a_n,n\in \N^*)$ such that $\P(Z_n=a_n)>0$  for every $n>0$. See Remark
\ref{rem:deftn} for conditions  on the existence of  such sequences.  We
denote  by $\tau_n$  a random  tree distributed  as the  GW tree  $\tau$
conditioned  on $\{Z_n=a_n\}$.  We study  the limit  in distribution  of
$\tau_n$  as $n$  goes to  infinity  and we  consider different  regimes
according to the growth speed  of the sequence $(a_n,n\in \N^*)$. Recall
that $Z_n$  is under $\P_k$ distributed  as a GW process  with offspring
distribution $p$ starting at $Z_0=k$.

We say  that the offspring distribution  $p$ is of type  $(L_0,r_0)$, when
$L_0$  is the  period  of $p$,  that  is the  greatest  common divisor  of
$\{n-\ell; \, n>\ell \text{ and } p(n)p(\ell)\neq 0\}$, and $r_0$ is the
residue $(\mod  L_0)$ of any  $n$ such  that $p(n)\neq 0$. 
See Remark \ref{rem:deftn} on sufficient  conditions to get $\P_k(Z_n=a)>0$. 

\subsection{The intermediate regime: $\lim_{n\rightarrow\infty }
  a_n/c_n\in (0, +\infty )$}
 We first state a strong ratio limit which is a direct consequence of 
the local limit
theorem in \cite{ds:lltgwp}. 

\begin{lem}\label{lem:srt-moderate}
Let $p$ be a non-degenerate super-critical offspring distribution with
finite mean and type $(L_0, r_0)$. Let $\theta\in(0,+\infty)$. Assume that
$\lim_{n\to\infty}a_n/c_n=\theta$ and that $a_n=r_0^n (\mod L_0)$
for all $n\in\N^*$. For all $h,k\in\N^*$, we have:
\[
   \lim_{n\rightarrow\infty}
\frac{\P_k(Z_{n-h}=a_n)}{\P(Z_n=a_n)}
=\mu^h \, \frac{
  w_k\left(\mu^h \theta \right)}{w(\theta)}\ind_{\{k=r_0^h(\mod L_0)\}}.
\]
\end{lem}

\begin{proof}
The local limit
theorem in \cite{ds:lltgwp} states that
for all  $k\in \Np$, $\theta\in (0,  +\infty )$ and $(a_n,  n\in \Nz)$ a
sequence       of      elements       of      $\Np$       such      that
$\lim_{n\rightarrow\infty } a_n/c_n=\theta$, we have:
\begin{equation}
   \label{eq:lim-Pk}
\lim_{n\rightarrow\infty } \left[c_n \P_k(Z_n=a_n) - 
 L_0 \ind_{\{ a_n=k r_0^n (\mod L_0)\}} w_k(\theta) \right]=0.
\end{equation}
We now assume that $a_n=k r_0^n (\mod L_0)$ and
$\lim_{n\rightarrow\infty } a_n/c_n=\theta\in (0, +\infty )$. 
Using Remark \ref{rem:deftn}, we deduce that $\P_k(Z_{n-h}=a_n)>0$ if and
only if $a_n=kr_0^{n-h}(\mod L_0)$ that is $k=r_0^h(\mod L_0)$. 
In this case, noticing that  $\lim_{n\rightarrow\infty
}a_n/c_{n-h}=\mu^h \theta$ as $\lim_{n\rightarrow\infty
}c_n/c_{n-h}=\mu^h$, using \reff{eq:lim-Pk}, we get that:
\[
   \lim_{n\rightarrow\infty}
\frac{\P_k(Z_{n-h}=a_n)}{\P(Z_n=a_n)}=  \lim_{n\rightarrow\infty} \frac{c_n}{c_{n-h}} \, \frac{
  w_k\left(\mu^h \theta \right)}{w(\theta)}=\mu^h \, \frac{
  w_k\left(\mu^h \theta \right)}{w(\theta)}\cdot
\]
\end{proof}

We deduce the following local convergence. 
\begin{prop}
   \label{prop:cv-fat-gene}
   Let  $p$ be  a non-degenerate  super-critical offspring  distribution
   with finite  mean.  Let $\theta\in (0,  +\infty )$.
   Assume that  $\lim_{n\rightarrow\infty }  a_n /c_n=\theta$ and  that
   $\tau_n$ is well defined for all $n$.  Then, we
   have the following convergence in distribution:
\[
\tau_n\; \xrightarrow[n\rightarrow \infty ]{\textbf{(d)}} \;  \tau^\theta. 
\]
\end{prop}
\begin{proof}
Assume  that $p$ is of type $(L_0,r_0)$, so that $\tau_n$ is well
defined for $n$ large if and only if $a_n=r_0^{n}(\mod L_0)$. 
Using  that a.s. $H(\tau^\theta)=+\infty $, 
the  characterization \reff{eq:cv-loi} of the convergence  in $\Tf$,
\reff{eq:ph} with $k=r_0^h(\mod L_0)$, and 
Lemmas  \ref{lem:=distrib-w} and \ref{lem:srt-moderate}, we directly get
the result. 
\end{proof}

\subsection{The high regime in the Harris case:
  $\lim_{n\rightarrow\infty }   a_n/c_n=+\infty  $}
\label{sec:H-R}
Let  $p$ be  a non-degenerate  super-critical offspring  distribution  with finite mean. Recall $\fb$ (the supremum of the
support of $p$) 
defined in \reff{eq:def-ab}. Notice that $\fb$ finite (Harris case) implies that $p$
has finite mean.
When  $\fb<\infty $, we define  $\tau^\infty$ as $\bt_\fb$,
the deterministic regular $\fb$-ary tree. 

\begin{prop}
   \label{prop:cv-not-vfat-gene}
   Let  $p$ be  a non-degenerate  super-critical offspring  distribution
   with $\fb<\infty $.  Assume that
   $     a_n     \leq    \fb^n$     for     all     $n\in    \N     ^*$,
   $\lim_{n\rightarrow\infty }  a_n /c_n=\infty  $ and that  $\tau_n$ is
   well  defined for  all $n$.   Then, we  have the  following
   convergence in distribution:
\[
\tau_n\; \xrightarrow[n\rightarrow \infty ]{\textbf{(d)}} \;  \tau^\infty . 
\]
\end{prop}

\begin{proof}
  We assume  that $\tau_n$  is well  defined, that is $\P(Z_n=a_n)>0$. For  $h\in \N^*$,  we have
  $\P(r_h(\tau)=r_h(\bt_\fb))=p(\fb)^{(\fb^{h} -1)/(\fb-1)}$.  We deduce
  from \reff{eq:ph} and \reff{eq:cv-loi},  using that $\bt_\fb$ has a.s.
  an    infinite    height,    that    the    proof    of    Proposition
  \ref{prop:cv-not-vfat-gene} is complete  as soon as we  prove that for
  all $k\leq \fb^h$:
\begin{equation}
   \label{eq:an=0-H}
\lim_{n\rightarrow\infty }  \frac{\P_{k}(Z_{n-h}=a_n)}{\P(Z_n=a_n)}=
p(\fb)^{-(\fb^{h} 
  -1)/(\fb-1)}\ind_{\{k=\fb^h\}}.
\end{equation}

In fact, it  is   enough   to   prove  \reff{eq:an=0-H}   for
$k=\fb^h$ as $\P(Z_h=\fb^h)=p(\fb)^{-(\fb^{h} 
  -1)/(\fb-1)}$ and:
\begin{equation}
   \label{eq:decomp-PZa}
\P(Z_n=a_n)=\P(Z_h=\fb^h)\P_{\fb^h} (Z_{n-h}=a_n)+ \sum_{k\leq  \fb^h-1} \P(Z_h=k)
\P_k(Z_{n-h}=a_n).
\end{equation}
It is also enough to consider the two cases:
$\lim_{n\rightarrow\infty       }      a_n/{\fb^n}=1$       or
$\limsup_{n\rightarrow\infty        }         a_n/{\fb^n}<1$        with
$\lim_{n\rightarrow\infty } a_n/{c_n}=+\infty $.  
\medskip 

We first  consider the case $\lim_{n\rightarrow\infty  } a_n/{\fb^n}=1$.
Notice  that  $\P_k( Z_{n-h}=a_n)=0$  for  $k\fb^{n-h}  < a_n$  as  each
individual produces at most $\fb$ children. For $k\leq \fb^h-1$, we have
$k\fb^{n-h}       \leq       \fb^n      -       \fb^{n-h}$.        Since
$\lim_{n\rightarrow\infty   }  a_n/{\fb^n}=1$,   we   deduce  that   for
$h,k\in \N ^*$, if $k\leq \fb^h -1$, then $k\fb^{n-h}<a_n$ for $n$ large
enough. Using \reff{eq:decomp-PZa}, we deduce that for $n$ large enough,
$\P(Z_n=a_n)=\P(Z_h=\fb^h)\P_{\fb^h}    (Z_{n-h}=a_n)$   as    soon   as
$\P(Z_n=a_n)>0$.  This gives \reff{eq:an=0-H}.  \medskip

The    case    $\limsup_{n\rightarrow\infty   }    a_n/{\fb^n}<1$    and
$\lim_{n\rightarrow\infty  } a_n/{c_n}=+\infty  $ is  proven in  Section
\ref{sec:upperZn}, see Lemma \ref{lem:rapp-Zb} with $\ell=1$.
\end{proof}

\subsection{The low regime: $\lim_{n\rightarrow\infty }
  a_n/c_n=0 $}
\label{sec:kK-R}
Let $p$  be a non-degenerate super-critical  offspring distribution with
finite mean. If $\fc>0$ (and  thus $\fa=0$), we recall that $\tau^{0,0}$
denote  the   distribution  of  the   GW  tree  $\tau$   with  offspring
distribution  $\fp$  given  in \reff{eq:def-fp}.   According  to  Remark
\ref{rem:cond-onE}, we have the following result for the extinction regime.

\begin{prop}
   \label{prop:cv-killed}
   Let  $p$ be  a non-degenerate  super-critical offspring  distribution
   with finite  mean such that $\fc>0$.   Assume that $ a_n  =0$ for $n$
   large enough  so that $\tau_n$ is  well defined for $n$  large enough.
Then, we have the following convergence in distribution:
\[
\tau_n\; \xrightarrow[n\rightarrow \infty ]{\textbf{(d)}} \;  \tau^{0,0}. 
\]
\end{prop}

Recall the Kesten tree $\tau^0$
from Definition \ref{def:Kesten}. 
Recall that  $\fa\geq 1$ implies that a.s. $\tau^0=\bt_\fa$,
the deterministic regular $\fa$-ary tree. 

\begin{prop}
   \label{prop:cv-not-fat-gene}
   Let  $p$ be  a non-degenerate  super-critical offspring  distribution
   with finite  mean.   Assume
   that $ a_n \geq 1 \vee \fa^n$  for all $n\in \N ^*$,
   $\lim_{n\rightarrow\infty } a_n  /c_n=0$ and  
that
   $\tau_n$ is well defined for all $n$.  Then, we have the following
   convergence in distribution:
\[
\tau_n\; \xrightarrow[n\rightarrow \infty ]{\textbf{(d)}} \;  \tau^0. 
\]
\end{prop}

\begin{proof}
We give the proof in the Schröder case ($\fa\leq 1$). The  Böttcher
case ($\fa\geq 2$) is more technical 
and its proof  is postponed to Section
\ref{sec:proof-B}. We suppose throughout the proof that $p$ is of type $(L_0,r_0)$.

\medskip\noindent
\textit{Case I: the sequence $(a_n,  n\in \N^*)$ is
bounded.}
We first  consider the  case $\fa=0$.   The ratio  theorem, see  (4) in
\cite{p:lgwpsc} (or \cite{an:bp} Theorem A.7.4), implies, that for all
$\ell, k, h\in \Np$, if $\P(Z_n=k)>0$ for $n$ large enough, then:
\[
\lim_{n\rightarrow\infty } \frac{\P_\ell(Z_{n-h}=k
  )}{\P(Z_{n}=k)}=\ell \fc^{\ell -1} \ff'(\fc)^{-h}.
\]
We deduce from \reff{eq:ph} and \reff{eq:ph0}, as $\fm=\ff'(\fc)$, 
that for $h\in \N^*$ and $\bt\in \Tf^{(h)}$, we have $\lim_{n\rightarrow\infty }  
\P(r_h(\tau_n)=\bt)=\P(r_h(\tau^0)=\bt)$. 
Since $\tau^0$ has a.s. an infinite height, we get that $\tau_n$
converges in distribution towards $\tau^0$  using the
convergence characterization  \reff{eq:cv-loi}. 
\medskip

We consider now  the case $\fa=1$. Recall that $\bt_\fa$  is the regular
$\fa$-ary  tree.  According to  Remark  \ref{rem:deftn},  for $k$  large
enough, we get that $\P(Z_{n}=k)>0$  and $\P(Z_{n-h}=k)>0$ for $n$ large
enough.  It is easy to check that for $h\in \N$, $k\in \Np$:
\[
\P(r_h(\tau)=r_h(\bt_\fa)|\, Z_n=k)=p(1)^h \frac{\P(Z_{n-h}=k)}{\P(Z_{n}=k)}\cdot
\]
For $k=1$, the left hand side member is equal to one. For $k>1$, it is
not difficult to get, by considering the lowest vertex of $\tau$ with
out-degree larger than one, 
that the sequence $(\P(Z_n=k)/\P(Z_n=1), n\in \N^*)$ is bounded. Then
arguing as in \cite{p:lgwpsc}, one gets that 
$\lim_{n\rightarrow\infty } \frac{\P(Z_{n-h}=k
  )}{\P(Z_{n}=k)}=p(1) ^{-h}$. 
This gives  that $\lim_{n\rightarrow\infty } \P(r_h(\tau)=r_h(\bt_\fa)|\,
Z_n=k)=1$. This implies that $\tau_n$
converges in distribution towards $\tau^0=\bt_\fa$ using the
convergence characterization  \reff{eq:cv-loi}.

\medskip\noindent
\textit{Case II: $\lim_{n\rightarrow\infty } a_n=+\infty $.}
We first consider the case $\fa=0$. Then we have $f_n(0)>0$ for all
$n\in \N^*$. Since 
 $\{\sum_{i=1}^\ell Z_{n-h}^{(i)}=a_n\}$ contains $\bigcup _{j=1}^\ell \left(\{Z^{(j)}_{n-h}=a_n\}\bigcap _{i\neq j}
\{Z_{n-h}^{(i)}=0\}\right) $, we deduce that 
$\P_\ell(Z_{n-h}=a_n
  )\geq  \ell \ff_{n-h}(0)^{\ell -1} \P(Z_{n-h}=a_n
  )$. Using that $\lim_{n\rightarrow\infty } \ff_{n-h}(0)=\fc$, we
  deduce from Lemma \ref{lem:srt-schroder}, stated below, 
that:
\[
\liminf_{n\rightarrow\infty }
\frac{\P_\ell(Z_{n-h}=a_n)}{\P(Z_n=a_n)}
\geq \ell \fc^{\ell -1} \ff'(\fc)^{-h}.
\]
As $\ff'(\fc)=\fm$, we deduce from \reff{eq:ph} and \reff{eq:ph0}
that 
\begin{equation}
   \label{eq:liminf}
\liminf_{n\rightarrow\infty }   
\P(r_h(\tau_n)=\bt)\geq \P(r_h(\tau^0)=\bt). 
\end{equation}
Since  $\tau^0$  has  a.s.  an   infinite  height,  we  deduce   that
\reff{eq:liminf}  holds for  all  $\bt\in \Tf ^{(h')}$  with $0\leq  h'\leq
h$.
Since  singletons   are open  subsets  of  the  closed  discrete  set
$\bigcup _{0\leq h'\leq h} \Tf^{(h')}$,  we deduce from the Portmanteau
theorem that $(r_h(\tau_n), n\in  \N)$ converges in distribution towards
$r_h(\tau^0)$. Since  this holds for  all $h\in \N^*$, and  since $\tau^0$
has  a.s.  an  infinite  height,   we  conclude  using  the  convergence
characterization \reff{eq:cv-loi}. 
\medskip    

We    now   consider the case      $\fa=1$.    Then   we    have
a.s.  $\tau^0=\bt_\fa$. We
deduce, as $\ff'(\fc)=p(1)$, that $\P(r_h(\tau)=r_h(\bt_\fa))=f'(\fc )^h$
and thus, using  \reff{eq:ph} and  Lemma \ref{lem:srt-schroder}:
\[
\P(r_h(\tau_n)=r_h(\bt_\fa))
= \P(r_h(\tau)=r_h(\bt_\fa))    \frac{\P(Z_{n-h}=a_n)}{\P(Z_n=a_n)}
\; \xrightarrow[n\rightarrow \infty ]{} \;1. 
\]
Since  this holds for  all $h\in \N^*$, and  since $\bt_\fa$
has  a.s.  an  infinite  height,   we  conclude  using  the  convergence
characterization \reff{eq:cv-loi}. 
\end{proof}

The proof of  the previous proposition in the Schröder  case is based on
the following strong ratio limit.

\begin{lem}\label{lem:srt-schroder}
Let $p$ be a non-degenerate super-critical offspring distribution with
finite mean in the Schröder case ($\fa\leq 1$). Assume that
$\lim_{n\to+\infty}a_n=+\infty$, $\lim_{n\to +\infty}a_n/c_n=0$ and
$\P(Z_n=a_n)>0$ for every $n\in\N^*$. Then we have for all $h\in \N^*$:
\begin{equation}
 \label{eq:liminfan0}
\lim_{n\to+\infty}\frac{\P(Z_{n-h}=a_n)}{\P(Z_n=a_n)}=\ff'(\fc)^{-h}.
\end{equation}
\end{lem}
Notice that according to Remark \ref{rem:deftn}, the condition
$\P(Z_n=a_n)>0$ in Lemma \ref{lem:srt-schroder} is
satisfied as soon as $a_n=r_0^n (\mod L_0)$ as
$\lim_{n\to+\infty}a_n=+\infty$ and $\lim_{n\to +\infty}a_n/c_n=0$. 

\begin{proof}
  Since $\fa\leq 1$, we have $r_0\in \{0,1\}$.  We deduce from Corollary
  5     in    \cite{fw:ld},     that    for     $k_n\leq    c_n$     and
  $\lim_{n\rightarrow\infty } k_n=+\infty $:
\begin{equation}
   \label{eq:schroder=0}
\lim_{n\rightarrow \infty }
\sup_{k\in [k_n, c_n], \, k=r_0 \, (\mod L_0), }
\val{ \frac{\mu^{n-\rho_k} c_{\rho_k}}{ L_0 w( k /\mu^{n-\rho_k}
    c_{\rho_k})}\, \P(Z_{n}=k) -1 
} =0,
\end{equation}
where $\rho_k=\min\{\ell\geq  1 ;  \, c_\ell  \geq k\}$.   Recall that
$\lim_{n\rightarrow \infty  } c_{n+1}/c_n=\mu$. The hypothesis  on $a_n$
imply thus  that  $\lim_{n\rightarrow\infty  } \rho_{a_n}=+\infty  $. 
 Set $\rho=\rho_{a_n}$ for
simplicity. Assume that $a_n=r_0^n (\mod L_0)$, so that
$\P(Z_n=a_n)>0$ for $n$ large enough.  For  $n$   large   enough,   we   have:
\[
   \frac{\P(Z_{n-h}=a_n)}{\P(Z_n=a_n)}
\sim \mu^{h} 
  \frac{w(a_n/\mu^{n-h-\rho} c_{\rho})}
  {w(a_n/\mu^{n-\rho} c_{\rho})}
\sim \mu^{\alpha h}\frac{V(a/c_\rho)}{V(a/c_\rho)}
 =\ff'(\fc)^{-h}, 
\]
where  we used  \reff{eq:schroder=0}  for the  first approximation,  the
representation \reff{eq:w=V} of $w$ in the Schröder case and that $V$ is
multiplicatively periodic with period $\mu$ for the second one.
\end{proof}

\section{Continuity in law of the extremal GW trees at $\theta=0$}
\label{sec:cont-0}
We  are  in  the  setting  of  Section  \ref{sec:gen-gen}.   Recall  the
definition of $\hat \tau^\theta$ given in Section \ref{sec:extremal} for
$\theta>0$ and  in Section  \ref{sec:Kesten} for $\theta=0$.   Since the
function  $w$   is  continuous,   we  get   that  the   distribution  of
$\hat  \tau^\theta$  and  thus  of   $\tau^\theta$,  as  a  function  of
$\theta\in (0,  +\infty )$  is continuous. From  the convergence  of the
offspring  distribution of  the individuals  of  type $\rs$  which is  a
consequence  of  Lemma \ref{lem:cv-rho},  we  deduce  the continuity  in
distribution of $\hat  \tau^\theta$ for $\theta\in [0,  +\infty )$. This
directly  gives   the  continuity  in
distribution of $\tau^\theta$ for $\theta\in [0, +\infty )$. We stress
in the next corollary that only the convergence at $0$ is non-trivial.

\begin{cor}
   \label{cor:cvttq}
 Let  $p$ be  a non-degenerate  super-critical offspring  distribution
   with   finite   mean.   We
   have the following convergence in distribution:
\[
\tau^\theta\; \xrightarrow[\theta\rightarrow 0 ]{\textbf{(d)}}\;
\tau^0.  
\]
\end{cor}

As  a  consequence of  Corollary  \ref{cor:cvttq},  we recover  directly
Corollary  3  from  \cite{bgms:gwtvml},  which is  stated  only  in  the
Böttcher case  ($p(0)+p(1)=0$) and extend  it to the Schröder  case, see
next  corollary. Recall  that  in  the Böttcher  case,  the random  tree
$\tau^0$ is  in fact  the (deterministic)  regular $\fa$-ary  tree.  For
$\varepsilon\in  (0,  1)$, let  $\tau_{(\varepsilon)}  $  be distributed  as
$\tau$  conditionally  on  $\{0<W\leq  \varepsilon\}$.  Notice  that  if
$\fc=0$, then conditioning  on $\{0<W\leq \varepsilon\}$ is  the same as
conditioning on $\{0\leq W\leq \varepsilon\}$.

\begin{cor}
   \label{cor:link}
   Let  $p$ be  a non-degenerate  super-critical offspring  distribution
   with finite mean.  We have the following convergence in distribution:
\[
\tau_{(\varepsilon)}\; \xrightarrow[\varepsilon\rightarrow 0 ]{\textbf{(d)}}\;
\tau^0.  
\]
\end{cor}

\begin{proof}
  Let  $h\in \N^*$  and  $\bt\in \Tf^{(h)}$  and  set $k=z_h(\bt)$.   We
  deduce    from    Lemma     \ref{lem:=distrib-w}    that    for    all
  $\theta\in (0, +\infty )$:
\[
   \P(r_h(\tau)=\bt)\,  \mu^h \, w_k(\mu^h \theta)
=\P(r_{h}(\tau^\theta)=\bt)\, w(\theta).
\]
Integrating with respect to $\theta\in (0, \varepsilon]$ for some
$\varepsilon>0$, we get:
\[
 \P(r_h(\tau)=\bt)\,  \P_k (0<W\leq \varepsilon\mu^h )
= \int_0^\varepsilon  \P(r_{h}(\tau^\theta)=\bt)\, w(\theta) \,
d\theta,
\]
where $W$ under $\P_k$ is distributed as $\sum_{\ell =1}^k W_\ell$,
where $(W_\ell, \ell\in \Np)$ are independent random variables  distributed as $W$
under $\P$.  Using Corollary \ref{cor:cvttq}, we get that:
\[
\lim_{\varepsilon\rightarrow 0}
\frac{\int_0^\varepsilon  \P(r_{h}(\tau^\theta)=\bt)\, w(\theta) \,
d\theta}{\P(0<W\leq \varepsilon)}=\P(r_{h}(\tau^0)=\bt). 
\]
This implies that:
\begin{equation}
   \label{eq:limrhte}
\lim_{\varepsilon\rightarrow 0}
\P(r_h(\tau)=\bt) \, \frac{\P_k (0<W\leq \varepsilon\mu^h )}{\P (0<W\leq
  \varepsilon)} = \P(r_{h}(\tau^0)=\bt). 
\end{equation}

On the other hand, we have:
\begin{align*}
\P(r_h(\tau)=\bt, 0<W\leq   \varepsilon) 
&=\P\left(r_h(\tau)=\bt, 0<\lim_{n\rightarrow\infty } \frac{Z_n}{c_n}\leq
  \varepsilon\right) \\
&=   \P(r_h(\tau)=\bt)\P_k\left( 0<\lim_{n\rightarrow\infty }
  \frac{Z_n}{c_{n+h}}\leq   \varepsilon\right) \\
&=   \P(r_h(\tau)=\bt)\P_k\left( 0< W\leq   \varepsilon\mu^{h}\right) ,
\end{align*}
where we used that $\lim_{n\rightarrow\infty } c_n/c_{n+h}=\mu^{-h}$
for the last equality. 
We deduce that:
\[
\P(r_h(\tau)=\bt\, |\, 0<W\leq   \varepsilon) 
=  \P(r_h(\tau)=\bt)\frac{\P_k\left( 0< W\leq
    \varepsilon\mu^{h}\right) }{\P(0<W\leq \varepsilon)}\cdot
\]
Then use \reff{eq:limrhte} and the characterization \reff{eq:cv-loi} of
the convergence  in $\Tf$ to conclude. 
\end{proof}

\section{Weak continuity   in   law   of   the  extremal   GW   trees   at
  $\theta=+\infty $}
\label{sec:tinfini}  
We are  in the setting  of Section \ref{sec:gen-gen}. The  continuity of
$(\hat    \tau^\theta,    \theta\in    [0,    +\infty    ))$    or    of
$(  \tau^\theta,  \theta\in   [0,  +\infty  ))$  at   infinity  is  more
involved.  And,  but  for  the  geometric  offspring  distribution,  see
\cite{abd:llfgt}, and  the Harris case, see Proposition
\ref{prop:cvtq-tinfty-gen}, we have a less precise result.

We first introduce a family of inhomogeneous  GW trees (whose offspring
distribution depends on the height of the vertex) which
converges in distribution toward a random tree $\tau^\infty$. These
trees are first constructed by absolute continuity with respect to the
distribution of $r_h(\tau)$ and can also be seen as two-type GW trees
generalizing the Kesten tree (see Subsection \ref{sec:tau-2-type}). The
tree $\tau^\infty$ will be a good candidate for the limit in distribution of $\tau^\theta$ as $\theta\to+\infty$ (as for the geometric case of \cite{abd:llfgt}) and the limit in distribution of $\tau_n$ when $a_n\gg c_n$. We prove only a weak limit in Subsection \ref{sec:lim-infty}.

\subsection{A family of inhomogeneous GW trees}
\label{sec:tau-lambda}

We set $\tilde \varphi(\lambda)=\varphi(-\lambda)=\E[\exp (\lambda
W)]$ for $\lambda\in \R$. 
Recall from Remark \ref{rem:def-lc} that:
$\lambda_c= \sup\{\lambda \in \R; \, \tilde \varphi(\lambda) <+\infty
\}\geq 0$, 
 $R_c=\tilde \varphi(\lambda_c/\mu)\geq 1$ is the convergence
radius of the generating function $\ff$ of $p$, see
\reff{eq:Rc-Lc}, and   $\ff(R_c)=+\infty $  if  and  only  if
$\tilde \varphi(\lambda_c)=+\infty $.
For $\lambda\in [-\infty , \lambda_c]$ and $h\in \N$, we set:
\begin{equation}
   \label{eq:def-ah}
\zeta_h(\lambda)=\tilde \varphi (\lambda \mu^{-h})=\E[\expp{\lambda
  \mu^{-h} W}]\in [\fc, +\infty ]. 
\end{equation}

  We  have for
$h,    \ell    \in   \Nz$    (with    an    obvious   convention    when
$\zeta_{h+\ell}(\lambda)=+\infty $) that:
\begin{equation}
   \label{eq:f(ah)}
\ff_h\left(\zeta_{h+\ell}(\lambda)\right)=\zeta_{\ell}(\lambda).
\end{equation}
The sequence
$(\zeta_h(\lambda), h\in \N)$ is bounded from below by $\fc$ and from above by $1$
if  $\lambda\leq   0$  and  from  below   by  $1$  and  from   above  by
$\zeta_0(\lambda)$  if $\lambda\geq 0$.  Notice that if $\lambda_c=+\infty $,
then we have  $\zeta_h(\lambda_c)=+\infty $ for all $h\in  \Nz$. Notice that
$\zeta_h(-\infty  )=\fc$ and thus $\zeta_h(-\infty  )=0 $ for all $h\in \Nz$ if
$\fa\geq  1$;  and $\zeta_h(-\infty  )>0  $  for  all $h\in  \Nz$  if
$\fa=0$.  We deduce that:
\begin{itemize}
   \item[(i)] $\zeta_h(\lambda)\in (0, +\infty )$ if and only if
     $\lambda\in(-\infty,\lambda_c)$,  or $\lambda=-\infty$ and $\fa=0$,
     or $\lambda=\lambda_c$ and
 $\zeta_0(\lambda_c)<+\infty$ (the latter
condition   being equivalent to $\ff(R_c)<+\infty $).
\item[(ii)] $\zeta_h(\lambda)=+\infty $ if and only if $\lambda=\lambda_c=+\infty $, or 
  $\lambda=\lambda_c$,  $h=0$ and $\zeta_0(\lambda_c)=+\infty $ (the latter
condition   being equivalent to $\ff(R_c)=+\infty $).
\item[(iii)] $\zeta_h(\lambda)=0$ if and only if $\lambda=-\infty $ and
  $\fa>0$. 
\end{itemize}

For  $h\in  \N$  and  $\lambda\in [-\infty,\lambda_c]$,  we  define  the
probability
$\tilde   p_h^{(\lambda)}   =\left(\tilde  p_h^{(\lambda)}   (k),   k\in
  \Ni\right)$
as  follows.  Recall  $\fa\in  \N$  and  $\fb\in  \bar  \N$  defined  in
\reff{eq:def-ab}.
\begin{itemize}
\item[(i)] If $\zeta_h(\lambda)\in (0, +\infty )$, we set for $k\in \N$: 
\begin{equation}
   \label{eq:def-tilde-p}
\tilde p_h^{(\lambda)} (k)=\frac{\zeta_{h+1}(\lambda)^k}{\zeta_h(\lambda)} p(k). 
\end{equation}
Thanks           to           \reff{eq:f(ah)},          we           get
$\sum_{k\in \Nz}  \tilde p_h^{(\lambda)} (k)=\ff(\zeta_{h+1}(\lambda))  / \zeta_h(\lambda)=1$,
so that  $\tilde p_h^{(\lambda)}  $ defined by  \reff{eq:def-tilde-p} is
 a probability distribution on $\Nz$.
\item[(ii)]     If     $\zeta_h(\lambda)=+\infty$     (which     implies
  $\lambda=\lambda_c$), we  set $\tilde p_h^{(\lambda)}$ the  Dirac mass
  at $\fb$.
\item[(iii)] If $\zeta_h(\lambda)=0$ (which implies $\lambda=-\infty $ and $\fa>0$), we set  $\tilde p_h^{(\lambda  )} $
  the Dirac mass at $\fa\in \N^*$. 
\end{itemize}
For     simplicity,    we     shall    write     $\tilde    p_h$     for
$\tilde p_h^{(\lambda)}  $, and specify  the value of $\lambda$  only if
needed.

\medskip

We define  $T^{(\lambda)}$ as a  GW tree with  offspring distribution
$\tilde   p_h$    at   generation   $h\in   \Nz$.     Since   the   case
$\lambda=\lambda_c$  will appears  later, we  will particularize  it and
write 
\begin{equation}
\tau^\infty=T^{(\lambda_c)}.
\end{equation}
If $\lambda_c=+\infty $,
then  the  tree  $\tau^\infty  $  is the  regular  $\fb$-ary  tree $\bt_\fb$,  where
$\fb\in  \bar \N$.   Notice  that  the root  of  $\tau^\infty  $ has  an
infinite number of children if and only if $\zeta_0(\lambda_c)=+\infty $ and
$\fb=\infty $, whereas  all the other individuals have  an infinite number
of children if and only if $\lambda_c=\fb=\infty $.

Notice  that  $T^{(\lambda)}$,  for $\lambda=0$,  is  distributed  as
$\tau$.  Since $\lambda\mapsto \tilde  p_h^{(\lambda)}$ is continuous on
the   set   of   probability   distributions   over   $\bar   \N$,   for
$\lambda \in  (-\infty ,  \lambda_c)$, we get  that the  distribution of
$T^{(\lambda)}$ is continuous for  the local convergence in distribution as
a function  of $\lambda$ over  $(-\infty ,  \lambda_c)$.  It is  easy to
check that  the tree-valued  random variable  $T^{(-\infty )}$  is in
fact  distributed  as  $\tau$  conditionally  on  the  extinction  event
$\ce=\{H(\tau)<+\infty \}$, that is $\tau^{0,0}$, if $\fa=0$  or  as the
regular  $\fa$-ary   tree $\bt_\fa$ if   $\fa\geq  1$.    In  the   latter  case,
$T^{(-\infty )}$ is thus defined  as the Kesten tree $\tau^0$ defined
in  Section  \ref{sec:Kesten}.   Taking  particular care  of  the  cases
$\fa\geq 1$ (when $\lambda$ goes  to $-\infty $), $0<\lambda_c<+\infty $
(when  $\lambda$ goes  to  $\lambda_c$), and  $\lambda_c=+\infty $  with
either $\fb$ finite  or not (when $\lambda$ goes to  $\lambda_c$), it is
not   difficult   to   check    that   the   probability   distributions
$\tilde p_h^{(\lambda)}$  over $\bar \N$  converge towards 
$\tilde
p_h^{(-\infty )}$ as $\lambda$  goes to $-\infty $, and to
$\tilde
p_h^{(\lambda_c)}$ as $\lambda$  goes to
 $\lambda_c$. This  implies the following result.

\begin{lem}
   \label{lem:cont-t-l}
Let $p$ be  a non-degenerate  super-critical offspring  distribution
with finite mean. 
We have that the family  $(T^{(\lambda)}, \lambda\in [-\infty , \lambda_c])$ is
continuous in distribution. 
\end{lem}

Let $\lambda\in [-\infty , \lambda_c]$. If $\zeta_0(\lambda)\in (0, +\infty )$, then for   $h\in   \N^*$  and
$\bt \in \Tf^{(h)}$, we have:
\begin{equation}
   \label{eq:def-t-t-gen}
\P(r_{h}(T^{(\lambda)})=\bt)= \prod_{u\in r_{h-1}(\bt)}  \tilde
p_{|u|}(k_u(\bt)) .
\end{equation}
If  $\lambda_c<+\infty$ and $\zeta_0(\lambda_c)=+\infty$,    then   for   $h\in   \Np$,    $k_0\in   \Np$   and
$\bt \in \T_{k_0}^{(h)}$, we have:
\begin{equation}
   \label{eq:def-t-t-infini}
\P(r_{h, k_0}(\tau^\infty)=\bt)= \prod_{u\in r_{h-1}(\bt)^*}  \tilde
p_{|u|}(k_u(\bt)) ,
\end{equation}
where    we    recall    that    for     a    tree    $\bs$    we    set
$\bs^*=\bs\setminus\{\partial\}$.               Remark             that
a.s. $T^{(\lambda)}\in \Tf$ if and only if $\zeta_0(\lambda)$ or $\fb$ is
finite,  and that  a.s. $\tau^\infty\in\Tf^*$  if and  only if  $\zeta_1(\lambda)$ or
$\fb$ is finite.  \medskip

We give a representation of the distribution of $T^{(\lambda)}$ as the 
distribution of $\tau$ with a martingale weight. The proof of the
following lemma is elementary and thus left to the reader. 

\begin{lem}
   \label{lem:t-t=mart-t-gen}
   Let  $p$ be  a non-degenerate  super-critical offspring
   distribution with finite mean. For $\lambda\in [-\infty , \lambda_c]$ such that
   $\zeta_0(\lambda)\in (0, +\infty )$, 
$h\in
\Np$ and $\bt \in \Tf^{(h)}$, we have with $k=z_h(\bt)$:
\begin{equation}
   \label{eq:t-t=mart-t-gen}
\P\left(r_{h}(T^{(\lambda)})=\bt\right)= 
\frac{\zeta_h(\lambda)^k}{\zeta_0(\lambda)} \,\,  \P\left(r_{h}( \tau)=\bt\right).
\end{equation}
For $\lambda_c<+\infty$ and $\zeta_0(\lambda_c)=+\infty$,   $h\in
\Np$, $k_0\in \Np$ and $\bt \in \T_{k_0}^{(h)}$, we have with $k=z_h(\bt)$:
\begin{equation}
   \label{eq:t-t=mart-t-infini}
\P\left(r_{h, k_0}(T^{(\lambda_c)})=\bt\right)= 
\frac{\zeta_h(\lambda_c)^k}{\zeta_1(\lambda_c)^{k_0}}
\,\,  \frac{\P\left(r_{h}(
  \tau)=\bt\right)}{ p(k_0)}\cdot
\end{equation}
\end{lem}
Notice that $\lambda_c<+\infty$ and $\zeta_0(\lambda_c)=+\infty$ occurs in the case of the
geometric offspring distribution studied in \cite{abd:llfgt}.

\subsection{A family of two-type  GW trees}
\label{sec:tau-2-type}
We    keep   notations    from   Section    \ref{sec:tau-lambda}.    For
$\lambda\in   (-\infty  ,   \lambda_c]$,  we   give  a   description  of
$T^{(\lambda)}$ using a two-type GW tree $\hat T^{(\lambda),\re}$.
\medskip

For     $h\in    \Nz$     and
$\lambda\in (-\infty , \lambda_c]$ such  that $\zeta_h(\lambda)$ is finite, we define
the probability distribution
$\hat  p_h^{(\lambda), \re}  =(\hat p_h^{(\lambda),  \re}(\ell), \ell\in
\N^*)$ by:
\begin{equation}
   \label{eq:def-p-l-e}
\hat  p_h^{(\lambda), \re}(\ell)=\frac{\left(\zeta_{h+1}(\lambda)-\fc\right)^\ell}{\left(\zeta_h(\lambda)-\fc\right)}
\frac{\ff^{(\ell)}(\fc)}{\ell!}\cdot
\end{equation}

Notice that  $\hat p_h^{(\lambda), \re}$  is indeed a probability  as by
the  Taylor-Lagrange expansion  at $\fc$  of $\ff$,  we have,  using
\reff{eq:f(ah)}                                                     that
$\sum_{\ell\ge 1}           \hat            p_h^{(\lambda),
  \re}(\ell)=(\ff(\zeta_{h+1}(\lambda))-\fc)/(\zeta_h(\lambda)-\fc)=1$.
For $\ell\in  \Nz^*$ such  that $\ff^{(\ell)}(\fc)>0$, we  also recall
the       $\ell$th-size       biased      probability       distribution
$\fp_{[\ell]}$, see Definition \reff{eq:fp-2}, with the convention that 
$\fp_{[\ell]}$ is the Dirac mass
at    $\ell$ if $\fc=0$. 
\medskip

We define  a two type  random tree  $\hat T^{(\lambda), \re} $ in the
next definition and write $T^{(\lambda),  \re}=\ske(\hat T^{(\lambda),  \re})$ for
the tree $\hat T^{(\lambda), \re} $ when one forgets the  types of the
vertices of $\hat T ^{(\lambda),  \re} $.

\begin{defi}\label{def:tau-lambda-e}
Let $p$ be a non-degenerate super-critical offspring distribution with finite mean.   Let     $\lambda\in    (-\infty     ,     \lambda_c]$.    The  labeled   tree
  $\hat T^{(\lambda), \re} $ is a two-type random tree whose vertices
  are  either  of type  $\rs$  (for  survivor)  or  of type  $\re$  (for
  extinct).

\begin{itemize}
   \item[(i)] 
If $\zeta_0(\lambda)=\zeta_1(\lambda)=+\infty $, then
$\hat  T^{(\lambda), \re}$  is  the regular  $\fb$-ary  tree and  all
its
vertices  are  of type  $\rs$  (and  thus there  is  no  vertex of  type
$\re$).

\item[(ii)] If  $\zeta_1(\lambda)<+\infty$, the
random tree $\hat T^{(\lambda), \re} $ is defined as follows:
\begin{itemize}
\item For  a vertex,  the number  of offsprings of  each type  and their
  positions  depend, conditionally  on  the vertices  of  lower or  same
  height, only on its own type (branching property).
   \item The root is of type $\rs$ with
     probability $(\zeta_0(\lambda)-\fc)/\zeta_0(\lambda)$. This probability is set to 1 if
     $\zeta_0(\lambda)=+\infty $. 
   \item A  vertex of type  $\re$ produces  only vertices of  type $\re$
     with sub-critical  offspring distribution $\fp$. 
   \item Recall that  only $\zeta_0(\lambda)$ might be infinite.  Let  $h\in \N$ such
     that      $\zeta_h(\lambda)$       is      finite.        A      vertex
     $u\in  \hat T^{(\lambda),  \re} $  at  level $h$  of type  $\rs$
     produces $\kappa^\rs(u)$  vertices of  type $\rs$  with probability
     distribution  $\hat  p_h^{(\lambda),  \re}  $  and  $\kappa^\re(u)$
     vertices    of   type    $\re$   such    that   conditionally    on
     $\kappa^\rs(u)=s_u\geq                                          1$,
     $k_u(  T^{(\lambda),  \re} )=\kappa^\rs(u)+\kappa  ^\re(u)$  has
     distribution $\fp_{[s_u]}$,  defined in \reff{eq:fp-2},  and the
     $s_u$  individuals of  type $\rs$  are chosen  uniformly at  random
     among the $k_u( T^{(\lambda),  \re})$ children.  More precisely,
     as   for   Definition   \ref{def:def-tau-theta},   we   denote   by
     $\cs_h=\{u\in T^{(\lambda), \re};\, |u|=h  \text{ and $u$ is
       of     type    $\rs$}\}$     the    set     of    vertices     of
     $\hat T^{(\lambda), \re}$  with type $\rs$ at  level $h\in \Nz$,
     and  we  have for $u\in \cs_h$: for all   $k_u\in  \Np$, 
     $s_u\in \{1,  \ldots, k_u\}$, and  $A_u\subset \{1, \ldots, k_u  \}$ such
     that $\sharp A_u=s_u$,
\begin{multline*}
\P\left(\kappa^\rs(u)+\kappa^\re(u)=k_u \text{ and } \cs_{h+1} \cap\{u1,
    \ldots, uk_u\}=uA_u \,|\,
    r_h(T^{(\lambda), \re}), \cs_h\right) 
\\
= \hat p_h^{(\lambda), \re}(s_u) \inv{\binom {k_u}{s_u}}
  \fp_{[s_u]}(k_u)= \frac{(\zeta_{h+1}(\lambda) -
  \fc)^{s_u}}{\zeta_h(\lambda)-\fc} \, \fc^{k_u-s_u}\, p(k_u).
\end{multline*}
	\item If $\zeta_0(\lambda)=+\infty $, then the
root, which is  of type $\rs$ a.s., has infinitely  many children of types
$\rs$ and  $\re$, each  children being, independently  from the  other, of
type $\rs$  with probability $(\zeta_1(\lambda)-\fc)/\zeta_1(\lambda)$.  That is for  $k_0\in \Np$
and $S_1\subset \{1, \ldots, k_0\}$:
\[
\P\left(\cs_1 \cap \{1, \ldots, k_0\} =S_1\right) =
\left(\frac{\zeta_1(\lambda)-\fc}{\zeta_1(\lambda)} \right)^{\sharp S_1}
\left(\frac{\fc}{\zeta_1(\lambda)}\right)^{k_0 -\sharp S_1}. 
\]
\end{itemize}
\end{itemize}
\end{defi}

Unless $\fa\geq 1$ or $\zeta_0(\lambda)=\zeta_1(\lambda)=+\infty $, conditionally on the fact that
the root  is of  type $\rs$,  a.s.  there exists  an infinite  number of
vertices of type  $\rs$ and of type $\re$.   By construction individuals
of type $\rs$  have a progeny which does not  suffer extinction, whereas
individuals  of  type  $\re$  have a  finite  progeny.   Informally  the
individuals of type $\rs$ in  $\hat T^{(\lambda), \re}$, if any, form
a backbone, on which are grafted, if
$\fa=0$, independent GW trees distributed as $\tau$ conditionally on the
extinction event  $\ce$.  This  is in  a sense  a generalization  of the
Kesten  tree, where  the backbone is reduced to an   infinite  spine  in the  case
$\fa\leq 1$.  We stress out that $\hat T^{(\lambda), \re}$, truncated
at level  $h$ can be  recovered from $r_{h}(T^{(\lambda),  \re})$ and
$\cs_h$ as all the ancestors of a vertex of type $\rs$ is also of a type
$\rs$ and  a vertex  of type  $\rs$ has  at least  one children  of type
$\rs$.  \medskip

The following result states that the random tree $T^{(\lambda)}$ can
be seen as the skeleton of a two-type GW tree. 

\begin{lem}
   \label{lem:=distrib3}
   Let $p$  be a  non-degenerate super-critical  offspring distribution
   with finite mean.
   Then,   for   $\lambda\in   (-\infty    ,   \lambda_c]$,   the   tree
   $T^{(\lambda),\re}$ is distributed as $T^{(\lambda)}$.
\end{lem}

\begin{proof}
  Let $\lambda\in  (-\infty , \lambda_c]$.   We first consider  the case
  $\zeta_0(\lambda)$ finite.   We assume $\fc>0$ (or  equivalently $\fa=0$).
  Let       $h\in        \Np$,       $\bt\in        \Tf^{(h)}$       and
  $S_h\subset       \{u\in       \bt;      \,       |u|=h\}$.        Set
  $k=z_h(\bt)=\sharp\{u\in \bt;  \, |u|=h\}$.   In order to  shorten the
  notations,     we     set     $\ca=S_h\bigcup \anc(S_h)$.     We     set,     for
  $\ell\in \{0,  \ldots, h-1\}$, $S_\ell=\{u\in \ca,  \, |u|=\ell\}$ the
  vertices at level $\ell$ which have  at least one descendant in $S_h$.
  For $u\in r_{h-1}(\bt)$, we set $s_u(\bt)=\sharp(\ca \bigcap u\Np)$, the number
  of  children of  $u$  having  descendants in  $S_h$.   We recall  that
  $\hat T^{(\lambda), \re}$  truncated at level $h$  can be recovered
  from   $r_{h}(T^{(\theta),   \re})$   and  $\cs_h$.    We   compute
  $\cc_{S_h}=\P(r_{h} (T^{(\lambda),  \re})=\bt ,\,  \cs_h=S_h)$.  We
  have by construction if $\sharp S_h>0$:
\begin{align}
\nonumber
   \cc_{S_h}
&= \frac{\zeta_0(\lambda)- \fc}{\zeta_0(\lambda)} \left[\prod_{u\in r_{h-1}(\bt),\, u\not\in \ca} \fp(k_u(\bt))
  \right]\, 
 \left[\prod_{u\in  \ca} \frac{\left(\zeta_{|u|+1}(\lambda) -
  \fc\right)^{s_u(\bt)}}{\zeta_{|u|}(\lambda)-\fc} \, 
  \fc^{k_u(\bt) - s_u(\bt)}\, p(k_u(\bt))\right]\\
\label{eq:calcul-Cs-ne}
&=\P(r_h(\tau) =\bt) \,
\fc^ {k - \sharp S_h}  \, \frac{\left(\zeta_{h}(\lambda) - \fc\right)^{\sharp S_h}}{\zeta_0(\lambda)}, 
\end{align}
where    we    used    that    for     a    tree    $\bs$,    we    have
$\sum_{u\in r_{h-1}(\bs)} k_u(\bs) - 1=  z_h(\bs) -1$ and that $\bs=\ca$
is tree-like with $z_h(\bs)=\sharp S_h$.  It is elementary to check that
Formula \reff{eq:calcul-Cs-ne} is also true when $S_h$ is empty, and the
root  is  thus  of  type  $\re$.   Since  $\cc_{S_h}$  depends  only  of
$\sharp S_h$, we shall write $\cc_{\sharp S_h}$ for $\cc_{S_h}$. We get:
\[
\P(r_{h} (T^{(\lambda), \re})=\bt )
= \sum_{i=0}^k   \binom{k}{i} \, \cc_i
=\P(r_h(\tau)=\bt)\frac{\zeta_h(\lambda)^k}{\zeta_0(\lambda)}\cdot
\]
We deduce  from \reff{eq:t-t=mart-t-gen} that $T^{(\lambda),\re}$ and
$T^{(\lambda)} $ have
the same distribution.  \medskip

The case $\zeta_0(\lambda)$  finite and $\fc=0$ (i.e. $\fa>0$) is  clear, as there is
no  vertex  of  type  $\re$  in $\hat  T^{(\lambda),  \re}$  and  the
offspring  distribution of  individuals of  type $\rs$  at level  $h$ in
$\hat T^{(\lambda), \re}$ given by \reff{eq:def-p-l-e}, that is:
\[
\hat  p_h^{(\lambda), \re}(\ell)=\frac{\left(\zeta_{h+1}(\lambda)-\fc\right)^\ell}{(\zeta_h(\lambda)-\fc)}
\frac{\ff^{(\ell)}(\fc)}{\ell!}
= \frac{\zeta_{h+1}(\lambda)^\ell}{\zeta_h(\lambda)} p(\ell),
\]
coincides with the offspring
distribution $\tilde p_h^{(\lambda)} (k)$ given in
\reff{eq:def-tilde-p} of individuals at level $h$ in $T^{(\lambda)}$. 
\medskip

We  consider  the  case  $\zeta_0(\lambda)=+\infty  $,
$\zeta_1(\lambda)$  finite and $\fc>0$.  Let 
$k_0,       h\in      \Np$,       $\bt\in      \T^{(h)}_{k_0}$       and
$S_h\subset       \{u\in       \bt;        \,       |u|=h\}$.        Set
$k=z_h(\bt)=\sharp\{u\in \bt; \, |u|=h\}$. Arguing  as in the case $\zeta_0(\lambda)$
finite, we get if $\fc>0$:
\[
\cc_{S_h}=\P(r_{h, k_0} (T^{(\lambda),  \re})=\bt ,\,  \cs_h=S_h)
= \frac{\P(r_{h} (\tau)=\bt)}{p(k_0)}\,  \fc^{k-\sharp S_h}
\frac{(\zeta_h(\lambda)-\fc)^{\sharp S_h}}{\zeta_1(\lambda)^{k_0}}, 
\]
and thus, writing $\cc_{\sharp S_h}$
for $\cc_{S_h}$ as the latter quantity depends only on $\sharp S_h$:
\[
\P(r_{h, k_0} (T^{(\lambda), \re})=\bt )
= \sum_{i=0}^k   \binom{k}{i} \, \cc_i
=\frac{\P(r_{h} (\tau)=\bt)}{p(k_0)}\,
\frac{\zeta_h(\lambda)^k}{\zeta_1(\lambda)^{k_0}}\cdot 
\]
Then use \reff{eq:t-t=mart-t-infini} to conclude. The sub-case $\fc=0$ is
handled in the same way as when $\zeta_0(\lambda)$ is finite. 
\medskip

Eventually,   we  consider   the  case   $\zeta_1=+\infty  $.   In  this   case
$T^{(\lambda),\re}$  and  $T^{(\lambda)}$   are  by definition regular  $\fb$-ary
trees, and they are thus a.s. equal.
\end{proof}

 For $\lambda>-\infty
$, we    denote    by
$T^{(\lambda),  *}$ the  tree-valued random  variable distributed  as
$ T^{(\lambda)}$ conditionally on the  non extinction event (which is
distributed   as  the   skeleton  of   $  \hat   T^{(\lambda),  \re}$
conditionally on the  root being of type $\rs$). Recall  the Kesten tree
$\tau^0$ defined in Section \ref{sec:Kesten}.

\begin{lem}
   \label{lem:cont-hat-t-l}
Let $p$ be  a non-degenerate  super-critical offspring  distribution
with finite mean. 
We have the following convergence in distribution:
\[
T^{(\lambda), *}\; \xrightarrow[\lambda\searrow - \infty
]{\textbf{(d)}}\;  \tau^0.
\]
\end{lem}
\begin{proof}
  Considering the  cases $\fa=0$ and $\fa\geq 1$, it is easy  to check
  that  the distributions  $\hat  p_h^{(\lambda), \re}$  over $\bar  \N$
  defined in \reff{eq:def-p-l-e}
  converge as  $\lambda$ goes to  $-\infty $  towards the Dirac  mass at
  $\max  (1, \fa)$.   This implies  the convergence  in distribution  as
  $\lambda$  goes   to  $-\infty  $  of   $\hat  T^{(\lambda),  \re}$
  conditionally on the  root being of type $\rs$  towards $\hat \tau^0$.
  Using that the extinction event of $T^{(\lambda), \re}$ corresponds
  to the  root of $\hat T^{(\lambda),  \re}$ being of type  $\rs$, we
  obtain the convergence of the lemma.
\end{proof}

\begin{rem}
   \label{rem:cv-bitype}
   In the proof  of Lemma \ref{lem:cont-hat-t-l}, we proved  in fact the
   convergence of the two-type random trees $\hat T^{(\lambda), \re}$
   conditionally on the  root being of type $\rs$  towards $\hat \tau^0$
   as $\lambda$ goes to $-\infty$, using the convergence in distribution
   of  the  probability  distribution $\hat  p^{(\lambda),  \re}$ as
   $\lambda$ goes to $-\infty $.
 
Similarly, considering carefully the three cases $\zeta_0(\lambda_c)$  finite;
$\zeta_0(\lambda_c)=+\infty $ and $\zeta_1(\lambda_c)$ finite;
$\zeta_0(\lambda_c)=\zeta_1(\lambda_c)=+\infty $, it is not very difficult to
check that $\hat T^{(\lambda), \re}$ converges in distribution towards $\hat
T^{(\lambda_c), \re}$  as $\lambda$ goes up towards $\lambda_c$. Then,
by considering only the skeleton, this allows to recover the 
convergence in distribution of $T^{(\lambda)}$ towards $\tau^\infty $
as $\lambda$ goes up to $\lambda_c$,
thus recovering  the continuity at $\lambda_c$  in Lemma \ref{lem:cont-t-l}. Notice that when
$\zeta_0(\lambda_c)=+\infty $, then the root  of $\hat
T^{(\lambda_c), \re}$ is a.s. of type $\rs$ and has infinitely many
children. 
\end{rem}

\subsection{Continuity   in   law   of   the  extremal   GW   trees   at
  $\theta=+\infty $}\label{sec:lim-infty}
\label{sec:cont-infty}
Recall that $T^{(\lambda),  *}$ is distributed   as
$ T^{(\lambda)}$ conditionally on the  non extinction event (which is
distributed   as  the   skeleton  of   $  \hat   T^{(\lambda),  \re}$
conditionally on the  root being of type $\rs$).

Recall  that   $\zeta_0(\lambda)>  \fc$  for  $\lambda>-\infty   $.  For
$\lambda\in (-\infty , \lambda_c]$,  such that $\zeta_0(\lambda)$ is finite,
we consider the function $g_\lambda$ defined by:
\[
g_\lambda(\theta)=\inv{\zeta_0(\lambda) -\fc} w(\theta) \expp{\lambda
  \theta}\ind_{(0, +\infty )}(\theta).
\]
Since, by definition, $\int g_\lambda= 1$, we deduce that $g_\lambda$ is
a probability  density. Let $\Theta_\lambda$  be a random  variable with
density $g_\lambda$. We consider the random tree $\tau^{\Theta_\lambda}$
and  the  random  two-type   tree  $\hat  \tau^{\Theta_\lambda}$,  which
conditionally    on    $\{\Theta_\lambda=\theta\}$    are    distributed
respectively  as  $\tau^\theta$  and  $\hat \tau^\theta$.  We  have  the
following representation.

\begin{prop}
   \label{prop:=tau-Ql}
   Let $p$  be a  non-degenerate super-critical  offspring distribution
   with finite mean.
   Then,   for   $\lambda\in   (-\infty    ,   \lambda_c]$ such that
   $\E[\expp{\lambda W}]$ is finite, we have that $\tau^{\Theta_\lambda}$
(resp. $\hat  \tau^{\Theta_\lambda}$) is distributed as 
 $T^{(\lambda),*}$ (resp. as $\hat T^{(\lambda), \re}$ conditionally on
the root being of type $\rs$). 
\end{prop}

\begin{proof}
  Let $h\in \N^*$, $\bt\in \Tf^{(h)}$ and $S_h\subset \{u\in \bt;\,
  |u|=h\}$ with $S_h$ non empty. We recall that the distribution of $\hat \tau^\theta$ up to
  generation $h$ is
  completely characterized by $r_h(\tau ^\theta)$ its skeleton up to
  level $h$ and by the set $\cs_h$ of vertices at generation $h$ which
  are of type $\rs$. We still denote by $\cs_h$ the  vertices of
  $\tau^{\Theta_\lambda}$ at generation $h$ which
  are of type $\rs$.  We have with $k=z_h(\bt)$ and $\ell=\sharp S_h$:
\begin{align}
\nonumber
   \P(r_h(\tau^{\Theta_\lambda}) =\bt, \cs_h=S_h)
&= \int  \P(r_h(\tau^{\theta}) =\bt, \cs_h=S_h)\, g_\lambda(\theta)\,
  d\theta\\
\nonumber
&= \P(r_h(\tau) =\bt) \,
\fc^ {k - \ell} \inv{\zeta_0(\lambda) - \fc} \int w^{*\ell}(\mu^h \theta)
  \expp{\lambda \theta} \mu^h d\theta\\ 
\nonumber
&= \P(r_h(\tau) =\bt) \,
\fc^ {k - \ell} \inv{\zeta_0(\lambda) - \fc} \E\left[\expp{\lambda \mu^{-h}
  \sum_{i=1} ^ \ell W_i}\prod_{i=1}^\ell \ind_{\{W_i>0\}}\right] \\
   \label{eq:tQlSh}
&= \P(r_h(\tau) =\bt) \,
\fc^ {k - \ell} \frac{(\zeta_h(\lambda)-\fc)^\ell}{\zeta_0(\lambda) - \fc},
\end{align}
where we used \reff{eq:calcul-Cs} for the second equality, that $(W_i,
i\in \Np)$ are independent random variables distributed as $W$ for the third one and the
definition of $\zeta_h$ given in \reff{eq:def-ah}. Then use
\reff{eq:calcul-Cs-ne} and that the root of $\hat T^{(\lambda), \re}$ is
of type $\rs$ with probability $(\zeta_0(\lambda)-\fc)/\zeta_0(\lambda)$ to get that:
\[
   \P(r_h(\tau^{\Theta_\lambda}) =\bt, \cs_h=S_h)
=   \P(r_h(T^{(\lambda), \re}) =\bt, \cs_h=S_h| \text{ type of
  $\partial$ is $\rs$}).
\]
Since  $\hat \tau^{\Theta_\lambda}$ up to level $h$
is characterized by 
$ \tau^{\Theta_\lambda}$ and $\cs_h$, and similarly for $\hat
T^{(\lambda), \re}$, we deduce from the previous equality that 
$\hat \tau^{\Theta_\lambda}$ is distributed as $\hat
T^{(\lambda), \re}$ conditionally on its root being of type
$\rs$. Then, forgetting about the types, we deduce that $ \tau^{\Theta_\lambda}$ is distributed as $
T^{(\lambda), *}$. 
\end{proof}

When  $\lambda$   goes  to   $-\infty  $,  we   get  that   the  measure
$g_\lambda(\theta) \,  d\theta$ converges  weakly to  the Dirac  mass at
$0$. We deduce that $\Theta_\lambda$ converges in distribution towards 0
as  $\lambda$ goes  to $-\infty  $.   We then  recover from  Proposition
\ref{prop:=tau-Ql}  and  Corollary  \ref{cor:cvttq} the  convergence  in
distribution of $T^{(\lambda),*}$, that is of $T^{(\lambda), \re}$
conditionally  on the  non-extinction event,  towards $\tau^0$  given in
Lemma \ref{lem:cont-hat-t-l}.  \medskip

If $\E[\expp{\lambda_c W}]=+\infty $ (and thus $\lambda_c>0$) or
equivalently $\ff(R_c)=+\infty $, then when
$\lambda$ goes up to $\lambda_c$ we get that $\Theta_\lambda$ converges
in   distribution  towards   $+\infty  $.    We deduce from  Lemma
\ref{lem:cont-t-l} the
following corollary. 

\begin{cor}
   \label{cor:cv-t-inf}
Let $p$  be a  non-degenerate super-critical  offspring distribution
whose generating function blows-up (that is $\ff(R_c)=+\infty $). Then, if $(\tau_\theta, \theta\in [0,
\infty ))$ converges in distribution as $\theta$ goes to infinity, then
the limit is the distribution of $\tau^\infty $. 
\end{cor}

\begin{rem}
   \label{rem:TI-extrem}
If $R_c=+\infty $, then the tree $\tau^\infty $ has all its
nodes with degree $\fb\in \bar \N$. Since  the distribution of
$\tau^\infty $ is maximal in the convex set of probability distributions on $\T_\infty
$, we get  that the distribution of $\tau^\infty $ is the limit in
distribution of
a sub-sequence $(\tau_{\theta_n}, n\in \N)$ with $\lim_{n\rightarrow \infty
} \theta_n=+\infty $. 
\end{rem}

We  are  able  to  prove  the stronger  result  on  the  convergence  in
distribution of $(\tau_\theta, \theta\in [0, \infty ))$ as $\theta$ goes
to  infinity  in   the  particular  case  of   the  geometric  offspring
distribution   (in   this   case   $\lambda_c$   is   positive   finite,
$\E[\expp{\lambda_c    W}]=+\infty   $    and   $\fb=\infty    $),   see
\cite{abd:llfgt}. The next proposition, which is a direct consequence of
the  convergence of  $\rho_{\theta, r}$  as $\theta\rightarrow+\infty  $
given  in Lemma  \ref{lem:cv-rho}, asserts  that  it also  holds if  the
offspring  distribution has  a  finite support  which  is the  so-called
Harris  case (in  this  case $\fb<\infty  $  and $\lambda_c=+\infty  $).
Otherwise, the general case is open.

\begin{prop}
\label{prop:cvtq-tinfty-gen}
Let $p$  be a  non-degenerate super-critical  offspring distribution
with finite support, that is  $\fb<+\infty $ (Harris case). Then we have  the following
  convergence in distribution:
\[
\tau^\theta\; \xrightarrow[\theta\rightarrow \infty ]{\textbf{(d)}}\; 
\tau^\infty.
\]
\end{prop}

\subsection{A remark on an other trees family}
We provide in this section an   alternative  description of
$T^{(\lambda)}$ using a  two-type GW tree  $\hat
T^{(\lambda),\rn}$.  \medskip

We assume that  $\lambda_c>0$. 
Notice that the sequence $(\zeta_h(\lambda), h\in \N)$ defined in \reff{eq:def-ah}
is non-increasing and
$\zeta_h(\lambda)>1$ for all $h\in \N$, $\lambda\in (0, \lambda_c]$. 
Furthermore, as $R_c>1$, we get that $\ff^{(\ell)}(1)$ is finite for
all $\ell\in \N$.    For $h\in \Nz$
and $\lambda\in  (0 ,  \lambda_c]$ such that  $\zeta_h(\lambda)$ is  finite, we
define                          the                          probability
$\hat  p_h^{(\lambda), \rn}$ as $\hat p_h^{(\lambda),  \re}$ in
\reff{eq:def-p-l-e} but with $\fc$ replaced by 1. That is for $\ell\in
\Np$:
\begin{equation}
   \label{eq:def-p-l-n}
\hat  p_h^{(\lambda), \rn}(\ell)=\frac{(\zeta_{h+1}(\lambda)-1)^\ell}{(\zeta_h(\lambda)-1)}
\frac{\ff^{(\ell)}(1)}{\ell!}\cdot
\end{equation}

For $\ell\in \Nz$ such that  $\ell\leq \fb$, we recall the $\ell$th-size
biased     probability     distribution     of    $p$     defined     in
\reff{eq:def-biased-p}.     We   define    a    two type  random    tree
$\hat  T^{(\lambda),  \rn}  $  in   the  next  definition  and  write
$T^{(\lambda),  \rn}=\ske(\hat T^{(\lambda),  \rn})$  as the  tree
$\hat T^{(\lambda), \rn} $ when one  forgets the types of
the vertices of $\hat T ^{(\lambda), \rn} $.

\begin{defi}\label{def:tau-lambda-n}
Let $p$ be a non-degenerate super-critical offspring distribution such that
$\lambda_c>0$. 
Let     $\lambda\in    (0     ,     \lambda_c]$. 
We define  a labeled   random tree  $\hat T^{(\lambda), \rn} $,  whose vertices
are either of  type $\rs$ (for survivor) or of  type $\rn$ (for normal).

\begin{itemize}
\item[(i)] If $\zeta_0(\lambda)=\zeta_1(\lambda)=+\infty $, then
$\hat T^{(\lambda), \rn}$  is the regular $\fb$-ary tree  and all its
vertices are of type $\rs$ (and thus  there is no vertex of type $\rn$).

\item[(ii)] If   $\zeta_1(\lambda)<+\infty$   , the
random tree  $\hat T^{(\lambda),  \rn} $ is  defined as  follows:
\begin{itemize}
\item For  a vertex,  the number  of offsprings of  each type  and their
  positions  depend, conditionally  on  the vertices  of  lower or  same
  height, only on its own type (branching property).
   \item The root is of type $\rs$ with
     probability $(\zeta_0(\lambda)-1)/\zeta_0(\lambda)$. This probability is set to 1 if
     $\zeta_0(\lambda)=+\infty $. 
   \item A  vertex of type  $\rn$ produces  only vertices of  type $\rn$
     with super-critical  offspring distribution
     $p$.
   \item Recall that  only $\zeta_0(\lambda)$ might be infinite.  Let  $h\in \N$ such
     that $\zeta_h(\lambda)$ is finite.  A  vertex $u\in \hat T^{(\lambda), \rn} $
     at level  $h$ of  type $\rs$  produces $\kappa^\rs(u)$  vertices of
     type        $\rs$        with       probability        distribution
     $\hat p_h^{(\lambda),  \rn} $ and $\kappa^\rn(u)$  vertices of type
     $\rn$  such   that  conditionally  on   $\kappa^\rs(u)=s_u\geq  1$,
     $k_u(  T^{(\lambda),  \rn} )=\kappa^\rs(u)+\kappa  ^\rn(u)$  has
     distribution ${p}_{[s_u]}$, defined in \reff{eq:def-biased-p}, and
     the $s_u$ individuals of type  $\rs$ are chosen uniformly at random
     among the $k_u( T^{(\lambda),  \rn})$ children.  More precisely if  we   denote   by
$\cs_h=\{u\in T^{(\lambda), \rn};\,  |u|=h \text{ and $u$  is of type
  $\rs$}\}$ the  set of  vertices of
$\hat T^{(\lambda), \rn}$  with type $\rs$ at level  $h\in \Nz$, 
and  we  have for $u\in \cs_h$: for all   $k_u\in  \Np$, 
     $s_u\in \{1,  \ldots, k_u\}$, and  $A_u\subset \{1, \ldots, k_u  \}$ such
     that $\sharp A_u=s_u$,
\begin{multline*}
\P\left(\kappa^\rs(u)+\kappa^\rn(u)=k_u \text{ and } \cs_{h+1} \cap\{u1,
    \ldots, uk_u\}=uA_u \,|\,
    r_h(T^{(\lambda), \rn}), \cs_h\right) 
\\
=\hat p_h^{(\lambda), \rn}(s_u) \inv{\binom {k_u}{s_u}}
  \fp_{[su]}(k_u)
=\frac{(\zeta_{h+1}(\lambda) - 1)^{s_u}}{\zeta_h(\lambda)-1} \, \fc^{k_u-s_u}\, p(k_u).
\end{multline*}

If $\zeta_0(\lambda)=+\infty $, then the
root, which is  of type $\rs$ a.s., has infinitely  many children of type
$\rs$ and  $\rn$, each  children being, independently  from the  other, of
type $\rs$  with probability $(\zeta_1(\lambda)-1)/\zeta_1(\lambda)$.  That is for  $k_0\in \Np$
and $S_1\subset \{1, \ldots, k_0\}$:
\[
\P\left(\cs_1 \cap \{1, \ldots, k_0\} =S_1\right) =
\left(\frac{\zeta_1(\lambda)-1}{\zeta_1(\lambda)} \right)^{\sharp S_1}
\left(\frac{1}{\zeta_1(\lambda)}\right)^{k_0 -\sharp S_1}. 
\]
\end{itemize}
\end{itemize}
\end{defi}

The  main difference  with  $\hat T^{(\lambda),  \re}$  is that  the
individuals of type $\rs$ in  $\hat T^{(\lambda), \rn}$, if any, form
a  backbone on  which  are  grafted, if  $\fa=0$,  independent GW  trees
distributed as $\tau$ (instead of $\tau$ conditionally on the extinction
event $\ce$ in $\hat T^{(\lambda), \re}$).

The following result  states that the random  tree $T^{(\lambda)}$ can
also be seen as the skeleton of this new two-type GW tree. Its proof, which follows the proof of Lemma
\ref{lem:=distrib3}, is left to the reader. 

\begin{lem}
   \label{lem:=distrib4}
   Let $p$  be a  non-degenerate super-critical  offspring distribution
   such that $\lambda_c>0$. 
   Then,   for   $\lambda\in   (0  ,   \lambda_c]$,   the   tree
   $T^{(\lambda),\rn}$ is distributed as $T^{(\lambda)}$.
\end{lem}

\begin{rem}
   \label{rem:divers} 
Recall that  $\zeta_0(\lambda)> 1$ for  $\lambda\in (0 , \lambda_c]$. For
  $\lambda\in (0 , \lambda_c]$,  such that
$\zeta_0(\lambda)$ is finite, we consider the function $h_\lambda$ defined by:
\[
h_\lambda(\theta)=\inv{\zeta_0(\lambda) -1} w(\theta) \left(\expp{\lambda
  \theta}-1\right)\ind_{(0, +\infty )}(\theta).
\]
Since, by definition, $\int h_\lambda= 1$, we deduce that $h_\lambda$ is
a probability density.  Let $\Theta'_\lambda$  be a random variable with
density     $h_\lambda$.      We     consider    the     random     tree
$\tau^{\Theta'_\lambda}$     and     the    random     two-type     tree
$\hat      \tau^{\Theta'_\lambda}$,      which     conditionally      on
$\{\Theta'_\lambda=\theta\}$    are    distributed    respectively    as
$\tau^\theta$  and  $\hat  \tau^\theta$.    Computation  similar  as  in
\reff{eq:tQlSh} gives  that for $h\in  \N^*$ and $\bt\in  \Tf^{(h)}$, with
$k=z_h(\bt)$,  and $S_h\subset  \{u\in  \bt;\, |u|=h\}$  with $S_h$  non
empty and $\ell=\sharp S_h$:
\[
   \P(r_h(\tau^{\Theta'_\lambda}) =\bt, \cs_h=S_h)
= \P(r_h(\tau) =\bt) \,
\fc^ {k - \ell} \frac{(\zeta_h(\lambda)-\fc)^\ell- (1-\fc)^\ell}{\zeta_0(\lambda) - 1}\cdot
\]
Similar computations as in \reff{eq:calcul-Cs-ne}
give that:
\[
   \P(r_h(T^{(\lambda), \rn}) =\bt, \cs_h=S_h)
= \P(r_h(\tau) =\bt) \,
 \frac{(\zeta_h(\lambda)-1)^\ell}{\zeta_0(\lambda)}\cdot
\]
Summing over all non-empty subsets $S_h$ of $\{u\in \bt;\,
  |u|=h\}$, gives that:
\[
    \P(r_h(\tau^{\Theta'_\lambda}) =\bt)
= \P(r_h(\tau) =\bt) \,
\frac{\zeta_h(\lambda)^k-1}{\zeta_0(\lambda) - 1}
=  \P(r_h(T^{(\lambda), \rn}) =\bt|\, \text{root is of type $\rs$}).
\]
Thus  the   random  tree  $\tau^{\Theta_\lambda'}$  is   distributed  as
$T^{(\lambda), \rn}$ conditionally on the root being of type $\rs$. 
%whereas $\hat \tau^{\Theta_\lambda'}$ is   not distributed  as $\hat
%T^{(\lambda), \rn}$ conditionally on the root being of type $\rs$. 
\end{rem}

\section{Convergence of conditioned  sub-critical GW tree}
\label{sec:sub-critical}
In this section, we consider a  sub-critical GW tree $\tau$ with general
non-degenerate offspring  distribution $p=(p(n), n\in \Nz)$  with finite
mean $\mu\in (0, 1)$.  To avoid trivial  cases, we assume that
$p(0)+p(1)<1$. We denote by $\ff$ the generating function of
$p$.    We     assume   that there     exists    $\kappa>1$     such    that
$\ff(\kappa)=\kappa$  and   $\ff'(\kappa)<+\infty  $.  Since
$\ff$ is strictly convex, $\kappa$, when it exists, is unique.  Those
assumptions are trivially satisfied if the radius of convergence of
$\ff$ is infinite. This is also the case for geometric offspring
distribution studied in \cite{abd:llfgt}. 

Define $\bar \ff(t)=\ff(\kappa t)/\kappa$ for $t\in [0,1]$ and note that
$\bar  \ff$ is  the generating  function of  a super-critical  offspring
distribution     $\bar     p=(\bar     p(n),     n\in     \Nz)$     with
$\bar p(n)=\kappa^{n-1} p(n)$. The mean $\bar  \mu$ of $\bar p$ is equal
to $\ff'(\kappa)$; the  fixed point $\bar \fc\in (0, 1)$  of $\bar f$ is
given by  $\bar \fc=1/\kappa$; and $\bar f'(\bar \fc)=\mu$.

 We have that $\bar \fp$  defined by \reff{eq:def-fp}
(with $p$ replaced by $\bar p$) is equal to $p$ by construction.  Notice
that we  are in the  Schröder case  and that $p$  is of type  $(L_0,0)$ as
$\bar p(0)>0$.  Let $\bar \tau$  be the corresponding  super-critical GW
tree.    It  is   elementary  to   check   that  for   $h\in  \Np$   and
$\bt\in \Tf^{(h)}$, we have with $k=z_h(\bt)$:
\begin{equation}
   \label{eq:t-bt}
\P(r_h(\tau)=\bt)=
\kappa^{k -1} \P(r_h(\bar\tau)=\bt).
\end{equation}

Recall that $Z_n=z_n( \tau)$, and set $\bar Z_n= z_n( \bar \tau)$. 
Following Section \ref{sec:gen-gen}, let $(c_n, n\in \Nz)$ be a sequence
with $c_0>0$  such  that  $\left(\kappa^{Z_n}  \expp{-Z_n/c_n}, n\in  \Nz\right)$  or
equivalently  $\left(\expp{-\bar   Z_n/c_n},  n\in  \Nz\right)$   is  a
martingale.   This  sequence is  increasing  positive  and
unbounded. Furthermore, the sequence $(c_{n+1}/c_n, n\in \Nz)$ increases
towards $\bar \mu=\ff'(\kappa)$.

We consider a sequence $(a_n, n\in \Np)$ of  integers such that
$\P(Z_n=a_n)>0$ (see Remark \ref{rem:deftn}). We denote by 
$\tau_n$ (resp. $\bar \tau_n$) a  GW tree distributed as $\tau$
(resp. $\bar \tau$)  conditionally on
$\{Z_n= a_n\}$ (resp. $\{\bar Z_n=a_n\}$). 
Clearly if $a_n=0$ for $n$ large enough, then $(\tau_n, n\in \Np)$
converges in distribution towards  $\tau$. So only the case $a_n$
positive for $n\in \N^*$ is of interest. 
 
It is straightforward to deduce from \reff{eq:t-bt} that for $n\geq
h\geq 1$ and $\bt\in \Tf^{(h)}$:
\begin{equation}
   \label{eq:bart-law}
\P(r_h(\tau_n)=\bt)=
 \P(r_h(\bar\tau_n)=\bt).
\end{equation}
Let $\theta\in  (0, +\infty  )$.  Let $\bar  \tau^\theta$ be  defined as
$\tau^\theta$ in Definition \ref{def:def-tau-theta}  where $p$ has to be
replaced by $\bar p$,  and $\fp$ is then equal to  $p$.  When $\fb$, the
upper  bound   of  the  support   of  $p$,   is  finite,  we   denote  by
$\bar  \tau^\infty  $ the  deterministic  regular  $\fb$-ary tree.   Let
$\bar  \tau^0$ be  defined as  the  Kesten tree  $\tau^0$ in  Definition
\ref{def:Kesten}  where  $\fp$   is  equal  to  $p$.    We  deduce  from
Propositions   \ref{prop:cv-fat-gene},  \ref{prop:cv-not-fat-gene}   and
\ref{prop:cv-not-vfat-gene}, \reff{eq:bart-law} and the characterization
\reff{eq:cv-loi} of the convergence in $\Tf$ the following result.
\begin{prop}
   \label{prop:cv-fat-gene-sub}
   Let $p$ be a  non-degenerate sub-critical offspring distribution with
   generating  function $\ff$  such that  $\fb\geq 2$ and  suppose
   that    there    exists    (a   unique)    $\kappa>1$    such    that
   $\ff(\kappa)=\kappa$   and   $\ff'(\kappa)<+\infty  $.    Let
   $\theta\in       [0,       +\infty       )$.        Assume       that
   $\lim_{n\rightarrow\infty  }  a_n  /c_n=\theta$,  $a_n>0$ and  
 $\tau_n$ is well defined  for  all
   $n\in \N^*$. Then,  we  have the
   following convergence in distribution:
\begin{equation}
   \label{eq:cv-tau-n-sous}
\tau_n\; \xrightarrow[n\rightarrow \infty ]{\textbf{(d)}} \;  \bar \tau^\theta. 
\end{equation}
If $\fb$  is finite,
then \reff{eq:cv-tau-n-sous} holds also for $\theta=\infty $.
\end{prop}

In the  sub-critical regime, the  local convergence of $\tau_n$  and the
identification  of the  limit if  any when  1 is  the only  root of  the
equation $\ff(\kappa)=\kappa$ is an open question.

\section{Ancillary results}
\label{sec:anc}
We  adapt the  proof of  Theorem 1  in \cite{fo:ld}.   Recall that  $W$,
conditionally  on $\{W>0\}$  has a  positive continuous  density $w$  on
$(0, +\infty )$. We shall use the following well known result.
\begin{lem}
   \label{lem:fw-integ}
Let $X$ be a real random variable with a continuous density. Let $a<b$
be elements of $\{\lambda\in \R; \, \E[\expp{\lambda X}]<+\infty
\}$. For  $z\in \C$ such that $\fR(z)\in K=[a,b]$, the Laplace transform
$g(z)=\E[\expp{zX}]$ is well defined and we have:
\[
\lim_{|t|\rightarrow+\infty } \sup _{u\in K} \frac{|g(u+it)|}{g
(u)}=0.
\]
Let $t_0>0$. There exists $\eta\in (0, 1)$ such that
for all $u\in K$, $ t\in \R$ with   $|t|\geq t_0$, we
have:
\begin{equation}
   \label{eq:majo-phi-t}
|g (u+it)|\leq  (1-\eta) g (u).
\end{equation}
\end{lem}

Recall the function $\tilde \varphi(z)=\E[\expp{zW}]$  is well
defined for $z\in \C$ such that  $\fR(z)\in \ck=\{\lambda\in \R; \,
\E[\expp{\lambda W}]<+\infty \}$. The next Lemma is a direct consequence
of \reff{eq:majo-phi-t}. 
\begin{lem}
   \label{lem:majo-f-c}
   Let    $p$    be    a   non-degenerate    super-critical    offspring
   distribution with finite mean. Let $a<0\leq b$ such
   that $K_0:=[a,b]\subset \ck$. Let $t_0>0$. There exists $\eta\in (0,
   1)$ such that for all $u\in K_0$, $t\in \R$ with $|t|\geq t_0$:
\begin{equation}
   \label{eq:majo-f-c}
|\tilde \varphi (u+it)  | \leq (1-\eta) \tilde  \varphi(u) .
\end{equation}
\end{lem}
\begin{proof}
  Set $\ca=\{(u,t);\, u\in K_0 \text{ and } |t|\geq t_0\}$. 
According   to   \reff{eq:majo-phi-t},   with  $X$   replaced   by   $W$
conditioned  on $\{W>0\}$,  there exists  $\eta'\in (0,  1)$ such  that
$|\tilde \varphi(u+it) -\fc|\leq (1-\eta')  (\tilde \varphi(u) -\fc)$ for
all  $(u,t)\in \ca$.     Taking
$\eta=\eta'(1- \fc/\tilde \varphi(a))\in (0,1)$ so that $\eta'\fc\leq
(\eta'-\eta)\tilde \varphi(u)$ for all $u\in K_0$, we get for all $(u,t)\in \ca$:
\[
|\tilde \varphi(u+it)|\leq  |\tilde \varphi(u+it) -\fc|+\fc 
\leq (1-\eta')  (\tilde \varphi(u) -\fc) + \fc= 
 (1-\eta')
\tilde \varphi(u) +\eta' \fc\leq (1-\eta) \tilde \varphi(u).
\] 
This gives the result. 
\end{proof}

The next lemma, see Lemma 16  in \cite{fw:lta}, is used for  the Fourier
inversion formula  of   $w^{*\ell}$.  Set 
\begin{equation}
   \label{eq:def-ck'}
\ck'=\{\lambda\in \R;\, \tilde \varphi'(\lambda)<+\infty
\}.
\end{equation}
 Notice that $\ck'\subset \ck$ and $\ck'\bigcup \{\lambda_c\}=\ck$. 
\begin{lem}
   \label{lem:fw-Fourier}
   Let    $p$    be    a   non-degenerate    super-critical    offspring
   distribution with finite mean. Let $a<0\leq b$ such
   that $K_0:=[a,b]\subset \ck$.  If  $\ell\in \N ^*$  is such that
   $\ell  >1/\alpha$, then we have:
\begin{equation}
   \label{eq:int-unif-f}
\sup_{u\in K_0} \int_\R |\tilde \cl (u+it)-\fc|^\ell \, dt<+\infty .
\end{equation}
If $\alpha<+\infty $  and if $K_0\subset \ck'$, 
then we have:
\begin{equation}
   \label{eq:int-unif-f'}
\sup_{u\in K_0} \int_\R |\tilde \cl' (u+it)| \, dt<+\infty .
\end{equation}
\end{lem}
Notice that the proof of Lemma \ref{lem:fw-Fourier} insures that $\tilde \cl (u+it)-\fc$
is not $L^1$ if $\ell\leq 1/\alpha$. This  dichotomy  appears already in
the proof of Lemma 9 from \cite{ds:lltgwp}. Recall that, as $p$ is super-critical, we write
$\fm=f'(\fc)\in [0, 1)$.

\begin{proof}
  The inequality \reff{eq:int-unif-f}  in the Böttcher case  is given in
  Lemma 16  in \cite{fw:lta}. So,  we now consider 
  the Schröder  case, that is  $\fm>0$.  In this  case, there
  exists    an    analytic    function     ${\rm    S}$    defined    on
  $\overset{\circ}D=\{z\in  \C, \,  |z|<1\}$ such  that the  convergence
  $\lim_{n\rightarrow+\infty } \fm^{-n}  ( \ff_n(z)-\fc)= {\rm S}(z)$
  holds  uniformly  on any  compact  subset  of $\overset{\circ}D$,  see
  \cite{ah:bp}  Corollary 3.7.3\footnote{Notice  Corollary 3.7.3  stated
    for   $z\in  \fc+   (1-\fc)\overset{\circ}D$  in   fact  holds   for
    $z\in \overset{\circ}D$ according to Lemma 3.7.2 in \cite{ah:bp}, as
    $\lim_{n\rightarrow\infty         }          f_n(z)=\fc$         for
    $z\in  \overset{\circ}D$.}.  Since  the functions  are analytic,  we
  also                            deduce                            that
  $\lim_{n\rightarrow+\infty } \fm^{-n} \ff_n'(z)= {\rm S}'(z)$ holds
  uniformly on any compact subset of $\overset{\circ}D$. We deduce 
from \reff{eq:Lap-W} and Remark \ref{rem:def-lc}
that $\tilde  \varphi(z)=f_k\left(\tilde
  \varphi(\mu^{-k}  z)\right)$ and thus for $k\in \N^*$ and $z\in \C$
such that $\varphi'(\mu^{-k} \fR(z))<+\infty $:
\begin{equation}
   \label{eq:form-f'n}
\tilde    \varphi'(z)=\mu^{-k}     f_k'\left(\tilde    \varphi(\mu^{-k}
  z)\right)\varphi'(\mu^{-k} z).
\end{equation}

  There   exists    $\varepsilon\in   (0,1)$,   such   that    for   all
  $z\in   \overset{\circ}D$    with   $|z-\fc|<\varepsilon(1-\fc)$   and
  $k\in  \N^*$, we  have  $|f_k(z)  -\fc|\leq \fm^k/\varepsilon$  and
  $|f'_k(z) |\leq \fm^k/\varepsilon$.  Since  $0\in K_0$, we get that
  if   $u\in   K_0$,  then   $u\mu^{-k}\in   K_0$.    Thanks  to   Lemma
  \ref{lem:fw-integ}  (with   $X$  replaced  by  $W$   conditionally  on
  $\{W>0\}$),  we   can  take  $k_0\in   \N$  large  enough   so  that
  $|\tilde  \varphi(\mu^{-k}u +it)  - \fc|\leq  \varepsilon(1-\fc)$, and
  thus  $\tilde \varphi(\mu^{-k}u  +it) \in  \overset{\circ}D$, for  all
  $k\geq   k_0$,   $u\in  K_0$   and   $t\geq   \mu^{k_0}$.  Then,   for
  $k\geq k_0\geq 0$ and $u\in K_0$, we get, with $\mu^ks=t$, that:
\begin{align}
\nonumber
\int_{\mu^{k+k_0}}^{\mu^{k+k_0+1}}  |\tilde \cl (u+it)-\fc|^\ell \, dt
&= \int_{\mu^{k+k_0}}^{\mu^{k+k_0+1}}  |\ff_k(\tilde \cl (\mu^{-k}(u+it)))-\fc| ^\ell \, dt\\
\nonumber
&= \mu^k\int_{\mu^{k_0}}^{\mu^{k_0+1}}  |\ff_k(\tilde \cl (\mu^{-k}u+is))-\fc| ^\ell \, ds\\
\label{eq:int-fk0}
&\leq  \mu^{k_0+1}  \varepsilon^{-\ell} (\mu \fm^{\ell})^k,
\end{align}
as well as, using  \reff{eq:form-f'n},
\begin{align}
\nonumber
\int_{\mu^{k+k_0}}^{\mu^{k+k_0+1}}  |\tilde \cl' (u+it)| \, dt
&=\mu^{-k} \int_{\mu^{k+k_0}}^{\mu^{k+k_0+1}}  
|\ff'_k(\tilde \cl (\mu^{-k}(u+it)))| \, |\tilde \cl' (\mu^{-k}(u+it))|  \, dt\\
\nonumber
&=\int_{\mu^{k_0}}^{\mu^{k_0+1}}  
|\ff'_k(\tilde \cl (\mu^{-k}u+is))| \, |\tilde \cl' (\mu^{-k}u+is)|  \, ds\\
\label{eq:int-ffk0}
&\leq  \mu^{k_0+1} \varepsilon^{-1} \fm^k \, \sup_{u\in K_0} \tilde
  \varphi'( u).
\end{align}
We deduce from \reff{eq:int-ffk0}  that  
the integral 
$\int_\R |\tilde \cl' (u+it)|\, dt$ is uniformly bounded  for $u\in
K_0$ as $\sup_{u\in K_0} \tilde
  \varphi'( u)$ is finite since $K_0\subset \ck'$, and from
\reff{eq:int-fk0} that 
the  integral 
$\int_\R |\tilde \cl (u+it)-\fc|^\ell \, dt$ is uniformly bounded  for $u\in
K_0$ as soon as 
$\mu
\fm^\ell=\mu^{1- \ell\alpha} <1$ that is  $\ell>1/\alpha$. 
\end{proof}

We give a similar result on the integrability of $\tilde \varphi_j$, the
Laplace transform of  $-W_j=-Z_j/c_j$. See (166) in  \cite{fw:lta} and a
variant of Lemmas 2  and 3 in \cite{ds:lltgwp}, see also  Lemmas 8 and 9
in       \cite{fw:ld}.      By       construction      the       process
$M=(M_n=\expp{-W_n},  n\in \N)$  is a  positive bounded  martingale with
respect  to  the  filtration   $(\cf_n=\sigma(Z_0,  \ldots,  Z_n),  n\in
\N)$.
It is  closed as it  converges a.s. towards $M_\infty  =\expp{-W}$.  Let
$g$ be  a convex non-negative function  defined on $(0, +\infty  )$.  We
deduce that $N^g=(N_n^g=g(M_n), n\in \N)$ is a positive sub-martingale which
converges a.s.  towards $N^g_\infty =g(M_\infty )$.  By Jensen inequality,
we    get     that    $N_n^g\leq    \E[g(M_\infty    )|\cf_n]     $.    If
$\E[g(M_\infty )]<+\infty $, then we get that: $N^g$ is uniformly
integrable, $N^g$ converges  in $L^1$  towards  $N^g_\infty $,
$\lim_{n\rightarrow\infty     }    \E[N^g_n]=\sup_{n\in \N }
\E[N^g_n]=\E[N^g_\infty ]<+\infty $. 
For $\lambda\in \ck$, consider the positive convex function 
$g(x)=     x^{-\lambda}$ defined on $(0,1)$ and set $\tilde
\varphi_n(\lambda)=N^g_n=\E[\expp{\lambda W_n}]$. We deduce that 
$\lim_{n\rightarrow+\infty }\tilde   \varphi_n(\lambda)
=\sup_{n\in \N }\tilde   \varphi_n(\lambda)=\tilde
\varphi(\lambda)$.  Using monotone
convergence, we get  that $\tilde \varphi_n(z)=\E[\expp{z W_n}]$
converges uniformly on compacts subsets of $\{z\in \C; \, \fR(z)\in \ck\}$
towards $\tilde \varphi(z)$ as $n$ goes to infinity. 

For $\lambda\in \ck'$, consider the positive convex function 
$\fg(x)=     -\log(x) x^{-\lambda}$ defined on $(0,1)$ and notice that $\tilde
\varphi'_n(\lambda)=N^\fg_n$. Arguing as for $g$, we get that
 $\tilde \varphi'_n(z)$
converges uniformly on compacts subsets of $\{z\in \C; \, \fR(z)\in
\ck'\}$ and that for $\lambda\in \ck'$:
\begin{equation}
   \label{eq:f'n-bound}
\lim_{n\rightarrow+\infty }\tilde   \varphi'_n(\lambda)
=\sup_{n\in \N }\tilde   \varphi'_n(\lambda)=\tilde
\varphi'(\lambda).
 \end{equation}

\begin{lem}
   \label{lem:fw-Fouriern}
   Let  $p$ be  a non-degenerate  super-critical offspring  distribution
   with finite  mean and type  $(L_0, r_0)$.  Let $a<0\leq b$  such that
   $K_0:=[a,b]\subset  \ck$.   Let  $t_1\in (0,  \pi  c_0/L_0)$.   There
   exists $\delta\in  (0,1)$ such  that for all  $u\in K_0$,  $t\in \R$,
   $n\in \N^*$, with $t_1\leq |t|\leq \pi c_n/L_0$:
\begin{equation}
   \label{eq:fw-Fourier2}
|\tilde \varphi_n (u+it)| \leq  (1-\delta) \tilde \varphi_n (u).
\end{equation}
If  $\alpha>1$, we have:
\begin{equation}
   \label{eq:fw-Fourier1}
\sup_{u\in K_0, n\in \N ^*} \int_{[\pm\pi c_n/L_0]} |\tilde
\varphi_n (u+it)-\fc| \, dt<+\infty. 
\end{equation}
If $\alpha<+\infty $ and $K_0\subset \ck'$,  we
have:
\begin{equation}
   \label{eq:fw-Fourier3}
\sup_{u\in K_0, n\in \N ^*} \int_{[\pm\pi c_n/L_0]} |\tilde
\varphi'_n (u+it)| \, dt<+\infty.
\end{equation}
\end{lem}

\begin{proof}
  Let  $t_1\in (0,  \pi c_0/L_0)$.   Because  of the  periodicity $L_0$,  we
  deduce that for  all $n\in \N^*$, $u\in K_0$, $0<|t|<  2\pi c_n/L_0$, we
  have  $|\tilde  \varphi_n(u+it)|/\tilde  \varphi_n(u)<1$.   Thanks  to
  Lemma \ref{lem:majo-f-c}, there exists  $\eta\in (0,1)$, such that for
  all      $u\in       K_0$,      $|t|\geq      t_1$,       we      have
  $|\tilde \varphi  (u+it) | \leq  (1-\eta) \tilde \varphi(u)  $.  Using
  the     uniform     convergence      on     compact     subsets     of
  $\{z\in  \C;  \,  \fR(z)\in   \ck\}$  of  $\tilde  \varphi_n$  towards
  $\tilde \varphi$  and $\tilde  \varphi> \fc$ on $K_0$, we  deduce that  for all
  $t_2>t_1$,  there   exists  $\eta'\in   (0,1)$  such  that   for  all
  $u\in  K_0$, $n\in  \N^*$, $\min(t_2,\pi  c_n/L_0)\geq |t|\geq  t_1$, we
  have:
  \begin{equation}
   \label{eq:majo-ftn1}
|\tilde  \varphi_n  (u+it)  | \leq  (1-\eta')  \tilde  \varphi_n(u).
 \end{equation} 
 Set     $t_1=\pi    c_0/     L_0$,     $t_2=\pi     c_0\mu    /L_0$     and
 $J_n=[\pi  c_0/  L_0,  \pi c_0c_n/L_0  c_{n-1}]\subset[t_1,  \min(t_2,  \pi
 c_n/L_0)]$.
 Using the uniform convergence on compacts of $\{z\in \C, \fR(z)\in \ck\}$ of $\tilde \varphi_n$
 towards  $\tilde \varphi$  and  $\tilde \varphi(0)=1$,  we deduce  from
 \reff{eq:majo-ftn1}    that    there   exists    $\varepsilon>0$    and
 $r_0\in (0,1)$  such that for  all $u\in K_0$ with $|u|\leq \varepsilon$,  $n\in \N^*$,
 $|t|\in J_n$, we have $|\tilde  \varphi_n  (u+it)  | \leq r_0<1$.
Thus, there exists $k_0\in \N$ such that $\sup_{K_0} |u|
c_{n-k_0}/c_n\leq \varepsilon$ for all $n\geq k_0$. Thus for all $k,n\in \N^*$
with $k+k_0\leq n$, $u\in K_0$, $|t|\in J_n$, we have:
\begin{equation}
   \label{eq:majo-fJ}
\val{\tilde \varphi_k\left(u\frac{c_k}{c_n} +it\right)}\leq  r_0<1.
\end{equation} 
\medskip

We consider now the Schröder case, that is $\fm>0$.  According to the
beginning of the proof of Lemma \ref{lem:fw-Fourier}, there exists a finite
constant $B$ such that for all $z\in \C$ such that $|z|\leq r_0$ and
$n\in \N^*$, we have:
\begin{equation}
   \label{eq:fnbg}
|f_n(z)-\fc|\leq  B \fm^n \quad\text{and}\quad
|f'_n(z)|\leq  B \fm^n.
\end{equation}
For $k\in \{1, \ldots, n\}$, set
$   J_{n,k}=\{t\in   \R,  \,   \pi   c_nc_0/L_0c_{k}   \leq  |t|\leq   \pi
c_nc_0/L_0c_{k-1}                                                     \}$,
so that $t\in  J_{n,k}$ implies $|t|c_k/c_n\in J_n$ as the sequence
$(c_k/c_{k-1}, k\in \N^*)$ is non-decreasing.   For $k,n\in \N^*$
with $k+k_0\leq n$, $u\in K_0$, $t\in J_{n,k}$, we deduce from
\reff{eq:majo-fJ} and \reff{eq:fnbg} that:
\begin{equation}
   \label{eq:majo-sur-Jnp}
\val{\tilde \varphi_n(u+it)-\fc}=\val{f_{n-k}\left(\tilde \varphi_k
  \left(u\frac{c_k}{c_n} +it\frac{c_k}{c_n} \right) \right)-\fc }\leq  B
\fm^ {n-k}. 
\end{equation}
and,                                                                with
$R_0=\sup_{n\in \N^*}  \sup_{u\in K_0} \val{\tilde
  \varphi'_n                      \left(u\right)
}$,  that:
\begin{equation}
  \label{eq:majo-sur-Jnp'}
\val{\tilde \varphi'_n(u+it)}
=\frac{c_k}{c_n} \val{f'_{n-k}\left(\tilde \varphi_k
  \left(u\frac{c_k}{c_n} +it\frac{c_k}{c_n} \right) \right)}\, 
\val{\tilde \varphi'_k
  \left(u\frac{c_k}{c_n} +it\frac{c_k}{c_n} \right) }
\leq
\frac{c_k}{c_n}  B
\fm^ {n-k} R_0. 
\end{equation}
Notice that $R_0=\sup_{u\in K_0} \tilde \varphi'(u)$ and it is finite if
$K_0\subset \ck'$.  

As $c_n/c_{k-1}\leq  \mu^{n-k+1}$, we get that $|J_{n,k}|\leq \pi c_0
c_n/L_0c_{k-1} \leq \pi c_0
\mu^{n-k+1}/L_0$. This implies that  for $k+k_0\leq n$:
\[
\int_{J_{n,k}} \val{\tilde \varphi_n(u+it)-\fc}\, dt\leq  \frac{B\pi c_0\mu}{L_0} (\mu
\fm)^{n-k}
\quad\text{and}\quad
\int_{J_{n,k}} \val{\tilde \varphi'_n(u+it)}\, dt\leq R_0 \frac{B\pi c_0\mu}{L_0} 
\fm^{n-k}.
\]
Since $[\pm \pi c_n/L_0]\subset [\pm \pi c_0 \mu^{k_0}/L_0] \cup \bigcup _{k=1}^{n-k_0} J_{n,k}$,
we deduce that for all $u\in K_0$:
\begin{equation}
   \label{eq:intfn}
 \int_{[\pm\pi c_n/L_0]} |\tilde \varphi_n (u+it)-\fc| \, dt
\leq \frac{\pi c_0}{L_0} \left(
 \mu^{k_0} \sup_{n\in \N^*} \tilde \varphi_n(\sup K_0)
+ \mu^{k_0} \fc +B\mu \sum_{k=1}^{n-k_0} (\mu\fm)^{n-k} 
\right) 
\end{equation}
and
\begin{equation}
   \label{eq:intfn'}
\int_{[\pm\pi c_n/L_0]} |\tilde \varphi'_n (u+it)| \, dt
\leq \frac{\pi c_0 R_0}{L_0} \left(
 \mu^{k_0} \sup_{n\in \N^*} \tilde \varphi'_n(\sup K_0)
 +B\mu \sum_{k=1}^{n-k_0} \fm^{n-k} 
\right) .
 \end{equation} 
The   the upper bound 
\reff{eq:intfn}  gives \reff{eq:fw-Fourier1} when  $\alpha>1$  that is
$\mu\fm<1$,  and the upper bound \reff{eq:intfn'} gives \reff{eq:fw-Fourier3}
when $\alpha<+\infty $ and $K_0\subset \ck'$.\medskip

We     now prove   \reff{eq:fw-Fourier1} in the   Böttcher  case,   that   is  $\fm=0$   and
$\alpha=+\infty          $.           Notice          then          that
$\sup_{|z|\leq r_0} |f_n(z)|\leq B\varepsilon_0^n$  for any $n\in \N ^*$
and $\varepsilon_0>0$  with some finite  constant $B$ depending  only on
$\varepsilon_0$.   Then we  obtain  \reff{eq:fw-Fourier1} using  similar
arguments as in the Schröder case.  \medskip

We now prove  \reff{eq:fw-Fourier2} in the Schröder  case.  There exists
$\eta''\in   (0,1/2)$   such  that   $\fc   <   (1-2  \eta'')^2   \tilde
\varphi(a)$.
We    can    choose    an    integer   $k'_0\geq    k_0$    such    that
$\fc + B\fm^{k'_0}  < (1- \eta'')^2 \tilde \varphi(a)$.   We can also
choose $n_0\geq k'_0$ large enough so that $c_{n_0}> c_0 \mu^{k'_0}$ and
$\inf_{n\geq    n_0}   \tilde    \varphi_n(a)\geq   (1-\eta'')    \tilde
\varphi(a)$. Notice that for $n\geq k'_0$:
\[
\left[\frac{\pi  c_0\mu^{k'_0} }{L_0 },
\frac{\pi c_n }{L_0 }\right] \subset\left[\frac{\pi c_n c_0}{L_0 c_{n-k'_0}},
\frac{\pi c_n }{L_0 }\right]=\bigcup _{k=1}^{n-k'_0} J_{n,k}. 
\]
 Using \reff{eq:majo-sur-Jnp}, we
get  that for $k\in \N^*$, $n\geq  n_0$
with $k+k'_0\leq n$, $u\in K_0$, $t\in J_{n,k}$:
\[
\val{\tilde \varphi_n(u+it)}\leq  \fc+ B \fm^{n-k}\leq  \fc+ B \fm^{k'_0} \leq
(1-\eta'')^2 \tilde \varphi(a)\leq
(1-\eta'') \tilde \varphi_n(a)\leq
(1-\eta'') \tilde \varphi_n(u).
\]
This                              gives                             that
$\val{\tilde \varphi_n(u+it)}\leq  (1-\eta'') \tilde  \varphi_n(u)$ for
all                            $u\in K_0$,
$t\in \left[\frac{\pi  c_0\mu^{k'_0} }{L_0 }, \frac{\pi  c_n }{L_0 }\right]$
and     $n\geq    n_0$.      This    and     \reff{eq:majo-ftn1}    with
$t_2=\pi  c_{n_0}/L_0>  \pi  c_0   \mu^{k'_0}/L_0$  complete  the  proof  of
\reff{eq:fw-Fourier2} in the Schröder case.  \medskip

The proof of \reff{eq:fw-Fourier2} in the Böttcher
case is similar and left to the reader. 
\end{proof}

\section{Results in the Harris case}
\label{sec:Harris}
We  present  detailed  proofs  of the  results,
because even if  they correspond to an adaptation of  the results known
in the  Böttcher case (see  \cite{fw:ld} and \cite{fw:lta}),  we believe that
the adaptation is not straightforward since in particular the Fourier
inversion of $w^{*\ell}$ is not valid if $\ell \alpha\leq 1$. 
We keep
notations from Sections \ref{sec:GW}
and \ref{sec:gen-gen}. Recall $\fb$ defined in \reff{eq:def-ab} is the
supremum of the support of the offspring distribution $p$. 
We assume $\fb<\infty $ (Harris case).   Following \cite{fo:ld} or \cite{bb:ldsbp}, we
define  the (right)  Böttcher  constant $\beta_H\in  (1,+\infty  ) $
by:
\[
\fb  =\mu^{\beta_H}.
\]

\subsection{Preliminaries}
Since $\fb$ is finite, the radius of convergence $R_c$ of $f$ is
infinite. 
According to Remark \ref{rem:def-lc},    we    deduce    that
$\lambda_c=+\infty   $, that is $W$ has all its exponential moments and
that for   every   $z\in    \C$, with 
$\tilde \cl(z)=\E[\expp{zW}] = \cl(-z)$:
\begin{equation}
   \label{eq:PC-eq}
\tilde \cl (z)= \ff (\tilde \cl (z/\mu)).
\end{equation}

We  define  the  function  $\tilde b$  on  its domain by:
\begin{equation}
   \label{eq:def-tilde-b}
\tilde b(z)=\log
(z)+\sum_{n=0}^{+\infty} \fb^{-n-1} \log \left(
  \frac{ \ff_{n+1}(z)}{ \ff_n(z)^\fb}\right) . 
\end{equation}
According to Lemma 2.5  in \cite{fo:ld}, 
for every $\delta\in (0,1)$, there exists a constant
$\theta=\theta(\delta)\in (0, \pi)$ such that $\tilde b$ is analytic on
the open set:
\begin{equation}
   \label{eq:def-tD}
   \tilde \cd(\delta)= \{z\in \C; \, 1+\delta<|z|< \delta^{-1}, \, |\arg(z)|<
   \theta\}. 
\end{equation}
Notice that the function $\tilde b$ is analytic and positive on $(1,\infty )$ and
satisfies on $(1, \infty )$:
\begin{equation}\label{eq:tilde-b}
\tilde b\circ  \ff=\fb \, \tilde b.
\end{equation}
According to Lemma 2.6 in \cite{fo:ld}, the function $\tilde b$
satisfies:
\begin{equation}
   \label{eq:prop-tb}
(s\tilde b'(s))'>0 \quad\text{on $(1, \infty )$}, \quad
\lim_{s\rightarrow 1+ } s\tilde b'(s)=0 \quad\text{and}\quad
\lim_{s\rightarrow+\infty } s\tilde b'(s)=1.
\end{equation}
In particular, the function $\tilde b$ is increasing on $(1, +\infty
)$. 
\medskip 

We  set $\tilde  \psi=\tilde b\circ  \tilde\cl$ on  $(0, +\infty  )$. We
directly recover Proposition  1 in \cite{bb:ldsbp}, where  it is assumed
that   $\fc=0$.   (We   could  have   used  directly   the  results   from
\cite{bb:ldsbp} using the generating function $\tilde f$ given by the so-called
Sevastyanov transform  of $f$: 
 $\bar  f(z)= [\ff(\fc+(1-\fc)z)-\fc]/(1-\fc)$, where $\bar
 f'(1)=\mu$, $\bar f'(0)=f'(\fc)$ and \reff{eq:PC-eq} also holds with
 $f$ replaced by $\bar f $. But this approach breaks down, when
 considering the upper large deviation for $Z_n$, see Section \ref{sec:upperZn}.) 
\begin{lem}
   \label{lem:psi}
   The  function  $\tilde \psi$  is  analytic,  increasing  and
   strictly     convex     on     $(0,     +\infty     )$, and 
\begin{equation}
   \label{eq:prop-tpsi}
\tilde \psi(s)=\tilde \psi(\mu s)/\fb \quad \text{on $(0, +\infty )$}, \quad
\lim_{s \rightarrow 0+} \tilde\psi'(s)=0  \quad\text{and}\quad
\lim_{s \rightarrow +\infty} \tilde \psi'(s)=+\infty .
\end{equation}
\end{lem}

\begin{proof}
   Since $\tilde  \cl$ is analytic on $\C$ and $\tilde \varphi((0,
+\infty ))=(1, +\infty )$, we get that $\tilde \psi$ is analytic on
$(0, +\infty )$. It is clear that $\tilde\psi$ is increasing as the composition of two increasing functions. Moreover,
using \reff{eq:prop-tb} as well as $\tilde
\varphi'(s)^2< \tilde \varphi''(s)\tilde \varphi(s)$ thanks to Cauchy-Schwartz
inequality, we have for every $s\in(0,+\infty)$,
\begin{align*}
\tilde\psi''(s) & =\tilde\varphi''(s)\tilde b'(\tilde\varphi(s))+\tilde\varphi'(s)^2\tilde b''(\tilde\varphi(s))\\
& \ge \frac{\tilde\varphi'(s)^2}{\tilde\varphi(s)}\bigl(\tilde b'(\tilde\varphi(s))+\tilde\varphi(s)\tilde b''(\tilde\varphi(s))\bigr)>0.
\end{align*}
We deduce that $\tilde \psi$ is strictly convex on $(0,+\infty)$. The functional equation $\tilde \psi(s)=\tilde \psi(\mu
s)/\fb$ is a direct consequence of \reff{eq:PC-eq} and \reff{eq:tilde-b}.
Then use that $W$ has an unbounded support to get that
$\lim_{s\rightarrow+\infty } \tilde\varphi'(s)/\tilde\varphi(s)=+\infty $ and
deduce the limits of $\tilde \psi'$ using \reff{eq:prop-tb}. 
\end{proof}

Recall Definition \reff{eq:def-tD} of $\tilde \cd(\delta)$. 
According to Lemma 2.5  in \cite{fo:ld}, 
there exists $\varepsilon=\varepsilon(\delta)\in (0, \tilde b(\delta))$ such that for all
$z\in \tilde \cd(\delta)$, we have:
\begin{equation}
   \label{eq:tfn-dev}
 f_n(z)= p(\fb)^{-1/(\fb -1)} \expp{\fb^n \tilde b(z)}
\left(1+ O(\expp{-\varepsilon \fb^n})\right).
\end{equation}
We have the following result (see Lemma 13 in \cite{fw:lta}). 

\begin{lem}
   \label{lem:majo-fn}
For all $s\in (1, +\infty )$ and all $n\in\N^*$, we have:
\[
f_n(s)<  p(\fb)^{-1/(\fb -1)} \exp \left\{\fb^n \tilde b(s)\right\}.
\]
\end{lem}
\begin{proof}
  We                                                                 set
\[
\tilde  b_N(z)=\frac{1 -   \fb^{-N}}{\fb  -1}\log( p(\fb))+  \log
(z)+\sum_{n=0}^{N-1}     \fb^{-n-1}     \log     \left(     \frac{
    \ff_{n+1}(z)}{ p(\fb) \ff_n(z)^\fb}\right)
\]
for $z\in  \bigcup_{\delta>0} \tilde \cd(\delta)$.   Notice that
$\tilde b_N(s)=\fb^{-N}  \log( f_N(s))$ for all  $s>0$.  
For $s>0$, we have:
\[
  \fb^{-N} \log(f_N(s))= \tilde b_N(s)
=\tilde b(s) - \frac{   \fb^{-N}}{\fb
  -1}\log(p(\fb)) - \sum_{n\geq N}      \fb^{-n-1}     \log     \left(
  \frac{
    \ff_{n+1}(s)}{ p(\fb) \ff_n(s)^\fb}\right)
\]
that is
\[
\log( f_N(s)) = \fb^{N} \tilde b(s)  - \inv{\fb -1}\log( p(\fb)) - \fb^N
\sum_{n\geq  N}  \fb^{-n-1}  \log \left(  \frac{  \ff_{n+1}(s)}{  p(\fb)
    \ff_n(s)^\fb}\right).
\]
For $s>1$, we have $p(\fb) 
f_n(s)^\fb<  f_{n+1}(s) $ so that:
\[
 \log( f_N(s))
< \fb^{N} \tilde b(s) - \inv{\fb
  -1}\log( p(\fb)) .
\]
This gives the result. 
\end{proof}

\subsection{Right tail of $w$}
We denote by $\tilde g$ the inverse of $\tilde \psi'$, which is one to one on
$(0, +\infty )$ by Lemma \ref{lem:psi}. 
For a given $v>0$,
the maximum of $uv-\tilde \psi(u) $ for $u\geq 0$ is uniquely reached at
$\tilde g(v)$: 
\begin{equation}
   \label{eq:min=tg}
\max_{u\geq 0} \left(uv-\tilde \psi(u) \right) = \tilde g(v)v - \tilde
\psi(\tilde g(v)). 
\end{equation}
We define the function $\tilde M$ for $v\in (0, +\infty )$ by:
\begin{equation}
   \label{eq:defM-R}
\tilde M(v)=v^{-\beta_H/(\beta_H-1)} \, \max_{u\geq 0} \left(uv-\tilde \psi(u) 
\right).
\end{equation}
According to  Proposition 2 in  \cite{bb:ldsbp}, the function $\tilde M$ is  analytic on $(0, +\infty )$. It is
positive     and      multiplicatively     periodic      with     period
$\mu^{\beta_H-1}=\fb/\mu$,   thanks  to   the  functional   equation  in
\reff{eq:prop-tpsi} and the definition of $\beta_H$,  (see also
Proposition 3 in  \cite{bb:ldsbp}).  \medskip

Mimicking the proof of Theorem 1  in \cite{fw:lta} (see also Remark 3 therein), we
set for $x\in [ \fb/\mu, \infty )$:
\begin{equation}
   \label{eq:def-y-R}
\tilde r(x)= \left\lfloor \frac{\log(x)}{\log(\fb/\mu)} \right\rfloor
\quad\text{and}\quad \tilde y(x)=x\left(\frac{\mu}{\fb}\right)^{r(x)}
= x \mu^{-r(x) (\beta_H-1)}=x \fb^{-r(x)(\beta_H-1)/\beta_H},
\end{equation}
so that $\tilde r(x)\geq 0$ and $\tilde y(x)\in [1, \fb/\mu)$. Notice that  $\tilde
r(x)\to+\infty$ as $x\to+\infty$. 
Let $\ell\in \Np$. We define
the positive functions for $y>0$:
\[
\tilde \cm_{1, \ell}(y)=  \frac{p(\fb)^{-\ell/(\fb -1)}}{\sqrt{2\pi
    \ell \tilde \sigma^2(y/\ell)}}\, y^{(\beta_H-2)/2(\beta_H-1)}
\quad\text{and}\quad
\tilde \cm_{2, \ell}(y)= \tilde \cm_{1, \ell}(y)\frac{y^{1/(\beta_H-1)} }
{\tilde g(y/\ell)},
\]
where $\tilde \sigma^2(y)=\tilde \psi''(\tilde g(y))>0$. For  $\ell\in \Np$ and $x\in [
\fb/\mu, +\infty )$, we set: 
\[
\tilde M_{1, \ell}(x)=\tilde \cm_{1, \ell}(\tilde y(x))
\quad\text{and}\quad
\tilde M_{2, \ell}(x)=\tilde \cm_{2, \ell}(\tilde y(x)). 
\]
By construction $x\mapsto \tilde y(x)$ is multiplicative periodic with period
$\mu/\fb=\fb^{\beta_H-1}$. We deduce that, for fixed $\ell\in \N^*$,
the functions $\tilde M_{1, \ell}$ and
$\tilde M_{2, \ell}$
are 
multiplicative periodic with period
$\mu/\fb$, positive, bounded and bounded away from 0. 
\medskip

We first state an upper bound on $w^{*\ell}$ whose proof is postponed to
Section \ref{sec:w<}. 
\begin{lem}
   \label{lem:w-maj}
Let $p$ be a non-degenerate super-critical offspring distribution with
$\fb<+\infty $. For all $u_1\geq 0$, there exists a
finite constant $C$ such that for all $\ell\in \N^*$, $x\geq  \fb/\mu$ and $u\in [0, u_1]$,  we
have with $r=\tilde r(x)$ and $y=\tilde y(x)$:
\begin{equation}
   \label{eq:w-maj}
w^{*\ell} (x) 
\leq  \frac{C\ell}{x} \fb^r \expp{-uy  \fb^{r}}
\ff_{r} (\tilde \cl(u)) ^\ell.
\end{equation}
\end{lem}

We  now  state  a  slightly  more   general  result  than  Remark  3  in
\cite{fw:lta}. (Notice in Remark 3  in \cite{fw:lta} that there is a misprint
in (21)  and (22) where  the power of $x$  in the exponential  should be
negative.)

\begin{lem}
   \label{lem:bott-wl-R}
   Let  $p$ be  a non-degenerate  super-critical offspring  distribution
   with $\fb<+\infty $. Let $\ell\in \N^*$.  As
   $x\nearrow +\infty $, we have:
\begin{align}
   \label{eq:w=i}
w^{*\ell}(x)
&\sim  \tilde M_{1,\ell}(x) \, x^{(2- \beta_H)/2(\beta_H-1)} \exp\left\{-
  \ell^{-1/(\beta_H-1)}\, x^{\beta_H/(\beta_H-1)} \tilde
  M(x/\ell)\right\},\\
\label{eq:wk>x}
w_\ell(x) & \sim w^{*\ell}(x),\\
   \P_\ell(W\geq x)
\label{eq:P>x}
&\sim  \tilde M_{2,\ell}(x) \,
   x^{-\beta_H/2(\beta_H-1)} \exp\left\{- 
     \ell^{-1/(\beta_H-1)}\,   x^{\beta_H/(\beta_H-1)} \tilde
  M(x/\ell)\right\} .
\end{align}
\end{lem}
Using Lemma 3.6.11 in \cite{ah:bp}, we could derive similar formula as
\reff{eq:w=i} for  the $j$-th derivative of $w$, for $j<\alpha$. 
The proof of Lemma \ref{lem:bott-wl-R} is given in Sections
\ref{sec:w=ii} and \ref{sec:P>x}.

\subsection{Proof of Lemma \ref{lem:cv-rho} in the Harris case}
\label{sec:H-f(c)}
Let $\ell\in \N^*$. Using \reff{eq:wr} and that $\ff^{(\fb)}(\fc)/\fb!=p(\fb)
$, we get for $x\geq  \fb/\mu$:
\begin{equation}
   \label{eq:+Rl}
\inv{\mu} w^{*\ell}(x/\mu)
=  \sum_{s=(s_1, \ldots, s_\ell)\in (\Np)^\ell}  w(x)^{*|s|_1}
\prod_{i=1}^\ell \frac{f^{(s_i)}(\fc)}{s_i!} 
= p(\fb)^\ell w^{*\fb \ell}(x) + R_\ell(x),
\end{equation}
where 
\[
R_\ell (x)=\sum_{s=(s_1, \ldots, s_\ell)\in (\Np)^\ell}  \ind_{\{|s|_1<
  \ell \fb\}} \,w(x)^{*|s|_1}
\prod_{i=1}^\ell \frac{f^{(s_i)}(\fc)}{s_i!} \cdot
\]
Using \reff{eq:w-maj}, we get for $u>0$ with $r=\tilde r(x)$ and
$y=\tilde y(x)$ defined in \reff{eq:def-y-R}:
\begin{align*}
R_\ell(x)
&\leq  C 
\expp{-uy \fb^{r}} \, \sum_{s=(s_1, \ldots, s_\ell)\in (\Np)^\ell}  \ind_{\{|s|_1<
  \ell \fb\}} \, \fb^{|s|_1} f_{r}(\tilde \cl(u))^{|s|_1}
  \prod_{i=1}^\ell \frac{s_i}{x}\frac{f^{(s_i)}(\fc)}{s_i!}  
\\
&\leq  \frac{C'}{x^\ell} \expp{-uy \fb^{r}} \, f_{r}(\tilde \cl(u))^{\ell \fb -1},
\end{align*}
for some finite constant $C'$ (depending on $u_1$ and independent of $x$
and $u\in [0, u_1]$). 
Using Lemma \ref{lem:majo-fn}, we get that  for all $u> 0$, $n\in \Np$:
\[
\ff_n(\tilde \cl(u))\leq  p(\fb)^{-1/(\fb -1)} \,\,  \exp \left\{\fb^n
  \tilde \psi(u)\right\}.
\]
Since $\tilde \varphi(u)\geq 1$, this gives with some constant $C''$ (depending on $u_1$ and independent of $x$
and $u\in [0, u_1]$):
\[
R_\ell(x)
\leq  \frac{C''}{x^\ell} 
\expp{\Gamma(x,u)}
\quad\text{with}\quad 
\Gamma(x,u)=
  (\fb \ell -1) \fb^r \tilde\psi(u)-uy \fb^{r}.
\]
We set $u^*=\tilde g(y/\fb \ell)$. We get:
\begin{align*}
   \Gamma(x,u^*)
&= \fb^{r+1} \ell\left[\tilde \psi(u^*) -u^*\frac{y}{\fb\ell}  \right]-
  \fb^r  \tilde\psi(u^*)\\ 
&=-\fb^{r+1} \ell
  \left(\frac{y}{\fb\ell}\right)^{\beta_H/(\beta_H-1)} \tilde M
  \left(\frac{y}{\fb\ell}\right) - \fb^r \tilde \psi(u^*) \\
&=- \ell^{-1/(\beta_H-1)} \fb^{-1/(\beta_H-1)} x^{\beta_H/(\beta_H-1)}
  \tilde M\left(\frac{x}{\mu\ell}\right)
-\fb^r \tilde \psi(u^*)\\
&=- \ell^{1/(\beta_H-1)} \left(\frac{x}{\mu}\right) ^{\beta_H/(\beta_H-1)}
  \tilde M\left(\frac{x}{\mu\ell}\right)
- \fb^r |\tilde\psi(u^*) |,
\end{align*}
where  we  used \reff{eq:min=tg}  and  \reff{eq:defM-R}  for the  second
equality;   that  $y=x   \fb^{-r(\beta_H-1)/\beta_H}$,  $\tilde   M$  is
multiplicative  periodic  with  period  $\fb/\mu$  for  the  third one;  and
$\fb=\mu^{\beta_H}$ and  $\tilde \psi$  is positive  for the  last one.  For
$x\in       [       \fb/\mu,       +\infty       )$,       we       have
$(y/\fb \ell)\in [1/\fb\ell,1/\mu\ell)$ and thus, as $\ell$ is fixed and
$\tilde g$ continuous positive, $u^*=\tilde g(y/\fb \ell)$ belongs to an
interval,  say   $[a,b]$,  with  $0<a<b<+\infty$.   This   implies  that
$c_0=\inf_{\{x\in [  \fb/\mu, +\infty )\}} |\tilde\psi(u^*)|>0$.   Notice also
that $c_2=\inf_{\{x>0\}}  \tilde M_{1, \ell}(x)$ is  positive as $\tilde
M_{1, \ell}$ is bounded away from 0.  Thus,
using    \reff{eq:w=i}, we deduce  that:
\[
R_\ell(x)
\leq  \frac{C''}{c_2} w^{*\ell}(x/\mu)
\expp{ -\ell \log(x)  - \fb^r c_0 - \frac{2-\beta_H
   }{2(\beta_H-1)} \log(x/\mu)}
\]
for $x$ large enough.
Recall    $r=\tilde    r(x)$    defined    in    \reff{eq:def-y-R}.     As
$x\rightarrow +\infty $ we have $r=\tilde r(x) \rightarrow +\infty $ and
$\log(x)\sim    \tilde   r(x)    \log(\fb/\mu)$.    Thus,   we    obtain
$R_\ell(x)=o(w^{*\ell}(x/\mu))$  as  $x\rightarrow +\infty  $.  Plugging
this in \reff{eq:+Rl} we get that:
\[
\lim_{x\rightarrow +\infty  } \mu \frac{w^{*\fb \ell}(x)}{w^{*\ell}(x/\mu)} \,
p(\fb)^\ell=1.
\]
From the definition of $\rho_{\theta,\ell}$ in \reff{eq:def-rho}, we
deduce that $\lim_{\theta\rightarrow +\infty } \rho_{\theta,\ell} (\fb, \dots
\fb)=1$.
This ends the proof of Lemma \ref{lem:cv-rho} in the Harris case.

\subsection{Upper large deviations for $Z_n$}
\label{sec:upperZn}
Recall Definition  \reff{eq:def-ck} of $\ck$ and  notations from Section
\ref{sec:anc},  and in  particular Definition \reff{eq:def-ck'}  of
$\ck'$. In the Harris case, we have $\ck=\ck'=\R$.  We recall that for      $j\in       \N^*$,
$\tilde     \cl_j(z)=\E[\expp{z    W_j}]=\ff_j(\expp{z/c_j})$,     with
$W_j=Z_j/c_j$, is well defined for $z\in \C$
and that 
$\tilde  \cl_j$ converges  uniformly  on the  compacts  of $\C$  towards
$\tilde \cl$ as $j$ goes to infinity.  Elementary computations give that
$\lim   _{u\rightarrow+\infty   }  \tilde   \cl_j'(u)/\tilde   \cl_j(u)=
\fb^j/c_j$.

We  consider the  functions  $\tilde  \psi_j=\tilde b\circ\tilde  \cl_j$
defined  on some  open  neighborhood  of $(0,  +\infty  )$  in $\C$  for
$j\in  \N^*$.
Following Lemma \ref{lem:psi}, it is  easy to
check that  the functions  $\tilde \psi_j$  are analytic  on $(0,  +\infty )$,
positive, increasing, strictly convex and that:
\[
\lim_{x \rightarrow 0+} \tilde \psi_j'(x)=0  \quad\text{and}\quad
\lim_{x \rightarrow +\infty} \tilde \psi_j'(x)= \frac{\fb^j}{c_j}\cdot
\]
Let $\tilde g_j$ be the inverse of $\tilde \psi'_j$ defined  on $(0, \fb^j/c_j)$.
In  particular,  for a  given  positive  $v<\fb^j/c_j$, the  minimum  of
$\tilde  \psi_j(u)  -  uv  $  for  $u\geq  0$  is  uniquely  reached  at
$\tilde  g_j(v)$.   Using that  $\tilde  \psi_j$  converges uniformly,  on
compacts  sub-sets of a neighborhood in $\C$  of $(0, +\infty )$, towards  $\tilde  \psi$, that  $\tilde  b$  and  thus
$\tilde \psi_j$  and $\tilde  \psi$ are  analytic, we  get that  for any
compact of  $(0, +\infty )$  and $j$  large enough, the  strictly convex
functions  $\tilde  \psi_j$  and their  derivatives  converge  uniformly
towards   the   strictly  convex   function   $\tilde   \psi$  and   its
derivatives. We deduce  that for any compact $K$ of  $(0, +\infty )$ and
$j$ large enough  (more precisely $j$ such that  $\fb^j/c_j> \sup (K)$),
$\tilde g_j$ is well defined on $K$ and converges uniformly towards $\tilde g$ on $K$.  \medskip

We  consider the  following  general setting.   Let  $\ell\in \N^*$  and
$a_n\in      [\ell     c_n/c_0,      \ell     \fb^n)$      such     that
$\limsup_{n\rightarrow\infty     }a_n/\ell     \fb^n     <1$.      Since
$\fb>\mu>c_{r+1}/c_r$ for  all $r\in  \N$, we  deduce that  the sequence
$(c_{n-l} \fb^{l},  \, 0\leq  l\leq n)$  is increasing.   Therefore, the
integer
$l_n=  \sup\{l\in \{0,  \ldots, n\},  \, c_{n-l}  \ell \fb^{l}  \leq c_0
a_n\}$
is well-defined and strictly less than $n$. Set $j_n=n-l_n\geq 1$ and
$y_n$ such that:
\begin{equation}
   \label{eq:def-jnln}
a_n=y_n \,  c_{j_n}\,  \ell\fb^{l_n},
\end{equation}
so that $y_n\in  [1/c_0,  \fb  c_{j_{n}-1}/c_0  c_{j_n}  )$.   Notice that the
conditions    $\lim_{n\rightarrow\infty     }a_n/c_n=+\infty    $    and
$a_n<  \ell \fb^n$  imply that  $\lim_{n\rightarrow\infty }  l_n=+\infty
$. The sequence $(j_n, n\in \N^*)$ may be bounded or not.

As   $c_{r+1}/c_r<\fb$    for   all   $r\in   \N$,    we   deduce   that
$y_n< \fb c_{j_n-1}/c_0c_{j_n} < \fb^{j_n}/c_{j_n}$. Thus, we can define
$\tilde          u_{n,\ell}^*=\tilde          g_{j_n}(y_n)$          and
$\tilde  \sigma^2_{n,\ell}=\tilde   \psi_{j_n}''(\tilde u_{n,\ell}^*)>0$.

\begin{lem}
   \label{lem:H-Pz}
   Let  $p$ be  a non-degenerate  super-critical offspring  distribution
   with   $\fb<\infty   $ and  type $(L_0,r_0)$.   Let
   $\ell\in               \Np$.                Assume               that
   $\lim_{n\rightarrow\infty              }a_n/c_n=\infty             $
   and $\limsup_{n\rightarrow\infty     }a_n/\ell      \fb^n
   <1$. Then, we have, with $\lim_{n\rightarrow\infty } \tilde\varepsilon_{n, \ell}=0$:
\[
\P_\ell(Z_n=a_n)
= \frac{L_0\,  p(\fb)^{-\ell/(\fb -1)}}
{c_{j_n}\sqrt{2\pi \, \ell \fb^{l_n} \, \tilde \sigma^2_{n, \ell}}} \, \exp \left\{
\ell\fb^{l_n} ( \tilde \psi_{j_n}( \tilde u_{n,\ell}^*) + \tilde u_{n,\ell}^* y_n) 
  \right\} (1+ \tilde \varepsilon_{n,\ell})\ind_{\{a_n=\ell r_0^n(\mod L_0)\}}.
\] 
\end{lem}
The proof, detailed in Section \ref{sec:H-Pz} is in the spirit of  the
proof of (175) in \cite{fw:lta}. 
We end this section with the following strong ratio limit.
\begin{lem}
   \label{lem:rapp-Zb}
   Let  $p$ be  a non-degenerate  super-critical offspring  distribution
   with     $\fb<\infty    $ and type $(L_0, r_0)$.    Let $\ell\in \N^*$.  Assume     that
   $\lim_{n\rightarrow\infty              }a_n/c_n=\infty            $, 
   $\limsup_{n\rightarrow\infty     }a_n/\ell      \fb^n     <1$, and
   $a_n=\ell r_0^n(\mod L_0)$ for all $n\in \N^*$. Then, we have:
\begin{equation}
   \label{eq:rapp-Zb}
\lim_{n\rightarrow\infty }\frac{\P_{ \ell\fb^h} (Z_{n-h}=a_n)}{\P_\ell(Z_n=a_n)}
= p(\fb)^{-(\fb^h-1)\ell/(\fb-1)}.
\end{equation}
\end{lem}

\begin{proof}
  Let $\ell\in \N^*$. Assume that   $a_n\in [\ell c_n/c_0, \ell
  \fb^n)$ and
  $a_n=\ell               r_0^n(\mod               L_0)$          for
  all $n\in \N^*$      and
  $\limsup_{n\rightarrow\infty  }a_n/\ell \fb^n  <1$.  An  estimation of
  $\P_{\ell}(Z_{n}=a_n)$ is given in  Lemma \ref{lem:H-Pz}.  We now give
  an  estimation  of  $\P_{\ell'}(Z_{n'}=a_n)$ with  $n'=n-h$  for  some
  $h\in \N  ^*$ and  $\ell'=\fb^h \ell$.  Recall  \reff{eq:def-jnln} and
  the definition of $l_n$, $j_n$ and $y_n$.  We have:
\[
a_n=y'_{n}c_{j'_{n}}\, \ell' \fb^{l'_n}=
y'_{n}c_{j'_{n}}\, \ell \fb^{l'_n+h},
\]
with                        $j'_n+l'_n=n'=n-h$                       and
$l'_n=  \sup\{l\in \{0,  \ldots, n'=n-h\},  \, c_{n-h-l}  \ell \fb^{l+h}
\leq                             c_0                             a_n\}$.
From  the definition  of $l'_n$,  we  deduce that  $l'_n=l_n-h$ so  that
$\ell' \fb^{l'_n}=\ell \fb^{l_n}$, $j'_n=j_n$  and thus $y'_n=y_n$. This
gives   that  $\tilde   g_{j_n}(y_n)=\tilde  g_{j'_n}(y'_n)$   and  thus
$\tilde    u^*_{n',    \ell'}=\tilde    u^*_{n,\ell}$   as    well    as
$\tilde \sigma^{2}_{n', \ell'}=\tilde  \sigma^{2}_{n, \ell}$.  Thanks to
Remark \ref{rem:deftn},  we have  $\P_{ \ell\fb^h}  (Z_{n-h}=a_n)>0$ and
$\P_\ell(Z_n=a_n)>0$ for  $n$ large.   We deduce  \reff{eq:rapp-Zb} from
Lemma \ref{lem:H-Pz}.
\end{proof}

\subsection{Proof of Lemma \ref{lem:w-maj}} 
\label{sec:w<}
Let  $\ell\in \N^*$ be fixed.  
We  deduce from  Lemma
\ref{lem:fw-Fourier} and  the Fourier inversion formula for
$xw^{*\ell}(x)$ that for $x>0$, 
$v\in \R$:
\begin{equation}
   \label{eq:Fourier-w}
w^{*\ell}(x)=-\frac{i\ell}{2\pi x} \int_\R \tilde \cl' (v+is)\, (\tilde
\cl (v+is)-\fc)^{\ell-1} \expp{-(v+is)x}\, ds. 
\end{equation}
We now follow  closely the proofs  from (120) to (148) of \cite{fw:lta}. 
Recall notations for $\tilde r(x)$ and $\tilde y(x)$  given in
\reff{eq:def-y-R}. 
Using \reff{eq:form-f'n} and \reff{eq:PC-eq}, 
we get with $r=\tilde
r(x)$, $y=\tilde y(x)$ and setting
$u=\mu^{-r} v $ and $t=\mu^{-r} s$:
\begin{align}\nonumber
w^{*\ell}(x)\!
&=-\frac{i\ell\mu^{-r}}{2\pi x} \int_\R  \!\tilde
  \varphi'\left(\frac{v+is}{\mu^r}\right)  f'_r\left(\tilde
  \cl \left(\frac{v+is}{\mu^r}\right)\right)\, \left( f_r\left(\tilde
  \cl \left(\frac{v+is}{\mu^r}\right)\right)- \fc\right)^{\ell-1}\!\!\!\! \expp{-(v+is)x}ds\\
\label{eq:wl-fourier}
&=- \frac{i\ell }{2\pi x} \int_\R H(u+it)\, dt,
\end{align}
where 
\begin{equation}
   \label{eq:def-H}
H(z)= \tilde
  \cl' (z)\, f'_r\left(\tilde
  \cl (z)\right)\, \left( f_r\left(\tilde
  \cl (z)\right)-\fc\right)^{\ell-1} \expp{-z\fb ^r y}. 
\end{equation}
Since $\tilde \varphi(z)-\fc=\E[\expp{zW}\ind_{\{W>0\}}]$, we deduce
that $|\tilde \varphi(z)-\fc|\leq  \tilde \varphi(\fR(z)) - \fc$. 
The Stevastyanov transformation of the generating function $f$ given
by $\bar f(z)=[f(\fc+(1-\fc)z)-\fc]/[1-\fc]$ is a generating function,
and the $r$-th iterate of $\bar f$ is  $\bar
f_r(z)=[f_r(\fc+(1-\fc)z)-\fc]/[1-\fc]$. Since $\bar f_r$ is a
generating function, we get that $|\bar f_r(z)|\leq  \bar f_r(|z|)$ and
thus $|f_r(\fc+z)-\fc|\leq  f_r(\fc+|z|)- \fc$. Using this last equality
with $z$ replaced by $\tilde \varphi(z)-\fc$, we get that:
\begin{equation}
   \label{eq:majo-f'phiz}
\val{f_r(\tilde \varphi(z))- \fc}
\leq  f_r(\fc+\val {\tilde \varphi(z)- \fc})- \fc 
\leq  f_r(\tilde \varphi(\fR(z))) - \fc\leq  f_r(\tilde \varphi(\fR(z))).
\end{equation}
Since $|f'_r(z)|\leq f'_r(|z|)\leq  \fb^r f_r(|z|)/|z|$, we get:
\[
|H(z)|\leq  \inv{\tilde \varphi(\fR(z))} \val{\tilde
  \cl' (z)}  \, \fb^r  f_r\left(\tilde
  \cl (\fR(z))\right)^{\ell} \expp{-\fR(z)\fb ^r y}. 
\]
Since $\tilde \varphi(u)\geq 1$ and $C:=\sup_{u\leq u_1} \int |\tilde \varphi'(u+it)
|\, dt<+\infty $, thanks to 
\reff{eq:int-unif-f'}, we deduce that:
\[
\int_\R H(u+it)\, dt \leq  C \fb^{r}  f_r\left(\tilde
  \cl (u)\right)^{\ell} \expp{-u\fb ^r y}. 
\]
Then use \reff{eq:wl-fourier} to conclude. 

\subsection{Proof of \reff{eq:w=i} in Lemma \ref{lem:bott-wl-R}} 
\label{sec:w=ii}
We keep notations from Section \ref{sec:w<}. 
Set $u_0=\tilde  g(1/\ell)$, $u_1=\tilde g(\fb/(\ell\mu))$ and $K=[u_0, u_1]$. Since  $u_0>0$, we have
$\tilde   \varphi(u_0)>1$.  Let   $\delta\in  (0,   1)$  be   such  that
$1+\delta<\tilde   \varphi(u_0)<\tilde \varphi(u_1)<\delta^{-1}$.    From   the
continuity   of  $\tilde   \varphi$   on  $\C$,   and   the  fact   that
$\tilde \varphi(K)\subset \tilde \cd  (\delta)$, we deduce there exists
$t_0>0$ such that for all $(u,t)\in K':=K\times [-t_0, t_0]$, we
have $\tilde \varphi(u+it) \in \tilde \cd(\delta)$, and thus $\tilde
\psi$ is analytic on an open neighborhood of $\{u+it;\, (u,t)\in K'\}$.
Since $\tilde \psi(u)>0$ and $\tilde \psi''(u)>0$ for $u>0$, we can  take $t_0$ small enough so
that $\fR ( \tilde \psi(u+it))>0 $ for $(u,t) \in K'$ and:
\begin{equation}
   \label{eq:majo-y3}
t_0\,  \sup_{(v,s)\in K'}|\tilde \psi'''(v+is)|\leq \inf_{v\in K}
\tilde \psi''(v).
\end{equation}
Recall $H$ defined in \reff{eq:def-H}. We shall study the asymptotics of
$\int_\R H(u+it)\, dt$ for large $x$. Condition \reff{eq:majo-y3}
will be used later on to study the main part of $\int_{|t|\leq t_0} H(u+it)\, dt$. 
\medskip 

\subsubsection*{First step: the tail part}
We first consider the tail part:
\[
I(t_0)=\val{\int_{|t|\geq  t_0} H(u+it) \, dt}.
\]
As
$\tilde  \varphi(u_0)>1$,  we  can  take $\eta$  small  enough  so  that
$(1-\eta)  \tilde \varphi(u_0)+  \eta \fc>1$ and \reff{eq:majo-f-c}
holds on $\ca=\{(u,t);\, u\in K \text{ and } |t|\geq t_0\}$. 
Using the first inequality in \reff{eq:majo-f'phiz}, we get for all
$(u,t)\in \ca$:
\[
 \val{f_r(\tilde \varphi(u+it))- \fc}
\leq  f_r((1-\eta) \tilde \varphi(u))- \fc\leq f_r((1-\eta) \tilde \varphi(u)) .
\]
We get for all $(u,t)\in \ca$ that $\val{\tilde \varphi(u+it)}\leq |\tilde \varphi(u+it)
-\fc|+\fc \leq  (1-\eta) \tilde \varphi(u)$ and, using $|f'_r(z)|\leq f'_r(|z|)\leq
\fb^r f_r(|z|)/|z|$,  that:
\[
\val{f'_r(\tilde \varphi(u+it))}
\leq \val{f'_r\left((1-\eta)\tilde \varphi(u)\right)}
\leq   \fb^r \frac{f_r((1-\eta)
\tilde \varphi(u))}{(1-\eta)
\tilde \varphi(u)} \cdot
\]
Using \reff{eq:def-H} and then Lemma \ref{lem:majo-fn}, we deduce that for all $(u,t)\in \ca$:
\begin{align*}
|H(u+it)|
&\leq  \inv{(1-\eta)\tilde \varphi(u)} \val{\tilde
  \cl' (u+it)}  \, \fb^r  f_r\left((1-\eta)\tilde
  \cl (u)\right)^{\ell} \expp{-u\fb ^r y}\\
&\leq  \inv{(1-\eta)\tilde \varphi(u_0)} \val{\tilde
  \cl' (u+it)}  \, \fb^r p(\fb)^{-\ell/(\fb-1)}  \expp{\ell \fb^r \tilde
  b\left((1-\eta)\tilde \varphi(u)\right)-u\fb ^r y}.
\end{align*}
Since $\tilde b$ is increasing, there
exists $\varepsilon'>0$ (depending on $u_0, u_1$ and $t_0$) such that  for
$u\in K$, 
\[
 \tilde b
\left((1-\eta) \tilde \varphi(u)\right)  )\leq \tilde b (\tilde \varphi(u))
  -\varepsilon'=\tilde \psi(u) -\varepsilon'.
\]
We get that for all $(u,t)\in \ca$:
\[
|H(u+it)|\leq  \frac{p(\fb)^{-\ell/(\fb-1)} }{(1-\eta')\tilde \varphi(u_0)} \val{\tilde
  \cl' (u+it)}  \, \fb^r  \expp{\ell \fb^r \tilde
\psi(u) -u\fb ^r y - \ell \fb^r \varepsilon'}.
\]
Using  \reff{eq:int-unif-f'} in Lemma \ref{lem:fw-Fourier},  we get,  for
some finite constant $c$ (depending on $u_0$, $u_1$, $t_1$ and $\ell$),
that for all $u\in K$ and $x>0$:
\begin{equation}
   \label{eq:majoI0}
I(t_0) \leq  c  \, \fb^r  \expp{\ell \fb^r \tilde
\psi(u) -u\fb ^r y - \ell \fb^r \varepsilon'}.
\end{equation}

\subsubsection*{Second step: the main part}
We now consider the main  part $J(t_0)=\int_{|t|\leq t_0} H(u+it)\, dt$.
An integration by part gives:
\[
 J(t_0)=\inv{\ell} \left[\left(f_r\left(\tilde \varphi(u+it)\right)-
    \fc\right)^\ell \expp{-(u+it)y\fb^r}\right]_{t=-t_0}^{t=t_0}
+\frac{iy\fb^r}{\ell}J_1(t_0),
\]
with
\[
J_1(t_0)=\int_{[\pm t_0 ]}  \left(f_r\left(\tilde
    \varphi(u+it)\right)-\fc\right)^\ell \expp{-(u+it)y\fb^r}\, dt.
\]
Arguing as in the first step, we get:
\begin{equation}
   \label{eq:majo-J1-J0}
\val{J(t_0) - \frac{iy\fb^r}{\ell}J_1(t_0)}
\leq  \frac{p(\fb)^{-\ell/(\fb-1)}}{\ell}  \expp{\ell \fb^r \tilde
\psi(u) -u\fb ^r y - \ell \fb^r \varepsilon'}.
\end{equation}

Now  $J_1(t_0)$ is  handled
 as in
\cite{fw:lta} from (128) to (139). 
By definition of $\delta$ and $t_0$,
we   get  that   $\tilde   \varphi(u+it)\in   \tilde  \cd(\delta)$   for
$(u,t)\in  K'$.  Use  \reff{eq:tfn-dev}, $\fR(\tilde  \psi(u+it))>0$ for
$(u,t)\in K'$  and that $\lim_{r\rightarrow+\infty }  |f_r(z)|=+\infty $
on $\tilde \cd(\delta)$, to get there exists $\varepsilon>0$ such that,
uniformly in  $u\in K$:
\begin{equation}
   \label{eq:main-H}
J_1(t_0)
= p(\fb)^{-\ell/(\fb-1)} \left(1+ O(\expp{-\varepsilon \fb^r})\right)
D(u),
\end{equation}
with 
\[ 
D(u)=\int_{-t_0}^{t_0} \expp{ \fb^{r}\left(\ell \tilde \psi(u+it) -(u+it)y \right)}
\, dt.
\]
We have for $(u,t)\in K'$:
\[
\tilde \psi(u+it)=\tilde \psi(u)+ it \tilde \psi'(u) -
\frac{t^2}{2} \tilde \psi''(u) + h(t,u),
\]
with          $|h(t,u)|\leq          t^3          C_3^+/6$,          and
$C_3^+=\sup_{(v,s)\in   K'}  |\tilde   \psi'''(v+is)|<+\infty  $.    Let
$C_2^- =\inf_{v\in K}  |\tilde \psi''(v)|$ which is  a positive constant
as  $\tilde \psi$  is increasing  and  strictly convex  on $(0,  +\infty
)$.
Recall  that by  definition  of $t_0$,  see  \reff{eq:majo-y3}, we  have
$t_0C_3^+\leq C^-_2$.

We define $\tilde u_\ell^*$ as $\tilde g(y/\ell)$, so that $\tilde u_\ell^*\in [u_0,
u_1]$ and we set $\tilde\sigma_\ell^2=  \tilde \psi''(\tilde u_\ell^*)$. We get:
\[
\ell\tilde \psi(\tilde u_{\ell}^* +it) -(\tilde u_{\ell}^*+it)y
= \ell \tilde \psi(\tilde u_{\ell}^*) - \tilde u_{\ell}^*y - \frac{t^2}{2}
\ell \tilde
\sigma^2_{\ell} +  \ell h(t, \tilde u_{\ell}^*),
\]
with $|h(t, \tilde u_{\ell}^*)|\leq  t^3 C_3^+/6$ and 
$|h(t, \tilde u_{\ell}^*)|\leq t^2 \tilde
\sigma^2_{\ell}/6$ for $t\in [-t_0, t_0]$. For $x$ large enough (and
thus $r$ large enough), we consider the
decomposition  $D(\tilde u_{\ell}^*)=D_1+D_2$ with:
\[
D_1=\int _{- r \fb^{-r}}^{r \fb^{-r}}
\expp{ \fb^{r}\left(\ell \tilde \psi(\tilde u_{\ell}^*+it)
    -(\tilde u_{\ell}^*+it)y \right)} 
\, dt.
\]
Using that $|h(t, \tilde u_{\ell}^*)|\leq  r^3 \fb^{-3r/2}
C_3^+/6$ for $|t|\leq r \fb^{-r/2}$, we get with $s=\sqrt{\ell \fb^{r}
  \tilde \sigma^2_{ \ell}}\, t$:
\begin{align*}
   D_1
&= \expp{\fb^{r}\left(\ell\tilde \psi(\tilde u_{\ell}^*)
    -\tilde u_{\ell}^* y \right)} 
\int _{- r\fb^{-r/2}}^{r \fb^{-r/2}}
\expp{- \ell\fb^{r} \tilde \sigma^2_{\ell} t^2/2 + \ell \fb^{r}
  h(t, \tilde u_{\ell}^*)}\, dt\\
&= \expp{ \fb^{r}\left(\ell\tilde \psi(\tilde u_{\ell}^*)
    -\tilde u_{\ell}^*y\right)} 
\int _{- r \fb^{-r/2}}^{r \fb^{-r/2}}
\expp{- \ell \fb^{r} \tilde \sigma^2_{ \ell} t^2/2 }\,dt \,
  \left(1+O( r^{3} \fb^{-r/2} )
\right)\\
&= \inv{\sqrt{\ell \fb^{r} \tilde \sigma^2_{\ell}}}
\expp{\fb^{r}\left(\ell \tilde \psi(\tilde u_{\ell}^*)
    -\tilde u_{\ell}^*y \right)} 
\int _{- r \tilde \sigma_{\ell}\sqrt{\ell}}^{r \tilde \sigma_{
  \ell}\sqrt{\ell}} 
\expp{-  s^2/2 }\,ds \,
  \left(1+O( r^{3} \fb^{-r/2} )
\right)\\
&= \ci \times
  \left(1+O( r^{3} \fb^{-r/2})
\right),
\end{align*}
with
\[
\ci=\frac{\sqrt{2\pi}}{\sqrt{\ell \fb^{r} \tilde \sigma^2_{ \ell}}}
\exp\left\{ \fb^{r}\left(\ell \tilde \psi(\tilde u_{\ell}^*)
    -\tilde u_{\ell}^*y \right)\right\} .
\]

We now give an upper bound on $|D_2|$. 
Since $|h(t, \tilde u_{\ell}^*)|\leq t^2 \tilde
\sigma^2_{ \ell}/6$, we deduce that for
$t\in [-t_0, t_0]$:
\[
\fR\left(\ell \tilde \psi(\tilde u_{\ell}^*+it) -(\tilde u_{\ell}^*+it)y \right)
\leq  \ell \tilde \psi(\tilde u_{\ell}^*) - u_{\ell}^*y-\ell \frac{t^2}{3} \tilde
\sigma^2_{\ell}.
\]
This implies that:
\begin{align*}
 |D_2|
&\leq  \expp{\fb^{r}\left(\ell \tilde \psi(\tilde u_{\ell}^*)
    -\tilde u_{\ell}^* y \right)} \int _{|t|\in [r \fb^{-r/2}, t_0]}
\expp{- \ell \fb^{r} \tilde \sigma^2_{ \ell} t^2/3}\, dt\\
&\leq 2t_0 \expp{ \fb^{r}\left(\ell\tilde \psi(\tilde u_{\ell}^*)
    -\tilde u_{\ell}^* y \right)} \expp{- \ell \, r^2 \tilde
  \sigma^2_{ \ell}/3}\\
&=\ci\times  O( r^{3} \fb^{-r/2}).
\end{align*}
This gives that $D(\tilde u_\ell^*)=\ci \times \left(1+O( r^{3} \fb^{-r/2})
\right)$. Use \reff{eq:main-H}, \reff{eq:majoI0}, \reff{eq:majo-J1-J0}
to get that:
\[
\int_\R H(\tilde u^*_\ell +it)\, dt=
\frac{iy\fb^r}{\ell}\, p(\fb)^{-\ell /(\fb-1)} \, \ci \times
  \left(1+O( r^{3} \fb^{-r/2})
\right).
\]
Then  use \reff{eq:wl-fourier},  the  definition  of $\tilde  u^*_\ell$,
which                            implies                            that
$\tilde u^*_\ell (y/\ell)- \tilde \psi (\tilde u^*_\ell) =\max_{u\geq 0}
((uy/\ell)  - \tilde  \psi(u))  = (y/\ell)^{\beta_H/(\beta_H-1)}  \tilde
M(y/\ell)$
with $y=y(x)$, and then the periodicity of $\tilde M$ to conclude.

\subsection{Proof of \reff{eq:wk>x} and \reff{eq:P>x} in Lemma
  \ref{lem:bott-wl-R}} 
\label{sec:P>x}
From \reff{eq:wk=w}, we get $w_\ell(x)=w^{*\ell}(x) + R(x)$, with:
\[
R(x)=\sum_{j=1}^{\ell -1} \binom{\ell}{j} \fc^{\ell-j}
w^{*j}(x).
\]
Using \reff{eq:w-maj}  and then Lemma \ref{lem:majo-fn},  we deduce there
exits a finite constant $c$ such that for all $x\geq \fb/\mu$ and $u\in K$:
\[
R(x) \leq   \frac{c}{x} \fb^r \expp{-uy  \fb^{r}}
\ff_{r} (\tilde \cl(u)) ^{\ell-1}
\leq \frac{c}{x} p(\fb)^{-(\ell-1)/(\fb-1)} \fb^r \expp{(\ell-1) \fb^r \tilde
  \psi(u) -uy  \fb^{r}}. 
\]
Taking $u=\tilde u^*_\ell$ and $\ci$ defined in Section \ref{sec:w=ii},
we get that $R(x)=\ci \times O(\expp{- \fb^r \tilde \psi(\tilde
  u^*_\ell)/2})=o(w^{*\ell}(x))$. This implies that $w_\ell(x) \sim
w^{*\ell}(x)$ as $x$ goes to infinity. This gives \reff{eq:wk>x}.
\medskip

An exact computation using \reff{eq:wk=w}  and \reff{eq:Fourier-w}
leads to:
\[
w_\ell(x)= -\frac{i\ell \fc}{2\pi x} \int_\R \tilde \cl' (v+is)\, \tilde
\cl (v+is)^{\ell-1} \expp{-(v+is)x}\, ds. 
\]
By definition, we have $\P_\ell(W\geq x)=\int_x^{+\infty } w_\ell(x') \,
dx'$. Arguing as in Section \ref{sec:w=ii}, with in particular the
integration by part (in $s$) for the main part, it is easy to get that:
\[
\P_\ell(W\geq x) \sim -\frac{i\ell \fc}{2\pi x} \int_\R \frac{\tilde
  \cl' (v+is)}{v+is}\, \tilde 
\cl (v+is)^{\ell-1} \expp{-(v+is)x}\, ds 
\]
as well as \reff{eq:P>x}. The details are left to the reader.

\subsection{Proof of Lemma \ref{lem:H-Pz}}
\label{sec:H-Pz}
  Recall  $a_n=y_n  c_{j_n}\ell   \fb^{l_n}>0$.  Using  Fourier  inversion
  formula, we have for $v>0$:
\begin{align*}
\P_\ell(Z_n=a_n)   
&= \frac{L_0}{2\pi} \int_{[\pm \pi/L_0]} f_{n}
\left(\expp{v+is}\right)^\ell \expp{-(v+is)a_n} \, ds\\
&= \frac{L_0}{2\pi} \int_{[\pm\pi/L_0]} \left( f_{n}
\left(\expp{v+is}\right)^\ell-\fc^\ell\right) \expp{-(v+is)a_n} \, ds
\end{align*}
since either $\fa\geq 1$ and thus $\fc=0$, or $\fa=0$ and $a_n= 0\ (\mod  L_0)$.
Setting $v=u/c_{j_n}>0$, $s=t/c_{j_n}$ and $H_{l,j}(z)=f_{l}\left(\tilde
  \varphi_j
(z)\right)^\ell-\fc^\ell$, we get using $l_n+j_n=n$:
\begin{equation}
   \label{eq:PZ=an}
\P_\ell(Z_n=a_n)   = \frac{L_0}{2\pi c_{j_n}} \int_{[\pm
  c_{j_n}\pi/L_0]}  H_{l_n, j_n}(u+it) 
\expp{-(u+it)y_n \ell \fb^{l_n}} \, dt.
\end{equation}
We now explicit the range of the possible choice for $u$ we shall consider.
Without loss of generality, we can assume that there exists $\delta_0>0$ such that 
$\sup_{n\in \N^* } a_n/\ell \fb^n <1 -\delta_0$. 
The restriction to $\R$ of the domain  of
definition      of      $\tilde g_{j}$ is $  D_{j}=(0, \fb^{j}/c_{j}  )$. 
Set $F_{j}=[1/c_0, \fb c_{j-1}/c_0c_{j}]$ for $j\geq 2$ 
and $F_1=[1/c_0, (1-\delta_0)\fb/c_1]$. 
From the uniform convergence of  $\tilde g_{j}$ towards $\tilde g$ on compact sets
of  $(0,  \infty  )$  and the fact that
$F_{j}\subset D_{j}$ for all $j\in \N^*$ and $\bigcup _{j\in \N ^*}
F_j\subset [1/c_0, \fb/c_1]$,  we  deduce that there exists
$0<u_0<u_1<+\infty $ such that  for  all
$j\in  \N^*$  and  all  $y\in  F_{j}$,  we have $\tilde g_{j}(y)\in
 K:=[u_0,  u_1]$.
Since $y_n\in F_{j_n}$,   we    deduce   that   the   sequence
$(\tilde u^*_{n,\ell}, n\in  \N^*)$ belongs  to   $ K$. 

\subsubsection{Preliminary upper bounds}
Using  the  continuity  of  $  \tilde  \varphi_{j}$  and  their  uniform
convergence towards  $\tilde \varphi$  as $j$ goes  to infinity,  we get
that there   exists  $t_0>0$,   $\delta\in  (0,   1)$  such   that  for   all
$(u,t)\in   K':=K\times  [-t_0,   t_0]$   and  $j\in   \N^*$,  we   have
$\tilde  \varphi_{j}(u+it)\in   \tilde  \cd(\delta)$ and $m_0=\inf
\{\fR(\tilde \psi_j(u+it); \, (u,t)\in K', j\in \N^*\}>0$.
We set 
$\tilde C_3^+=\sup_{j\in  \N^*}\sup_{(u,t)\in   K'}  |\tilde  \psi_j'''(u+it)|$
which is a finite  constant since the derivative of  $\tilde \psi_j$ converges
uniformly  on  $K'$  towards  the   derivative  of  $\tilde  \psi$.  Let
$\tilde C_2^- =\inf_{j\in \N^*}\inf_{u\in K}  |\tilde \psi_j''(u)|$ which is a
positive  constant  since  the  derivative  of  $\tilde  \psi_j$  converges
uniformly  on $K$  towards the  derivative  of $\tilde  \psi$ and  that
$\tilde \psi_j$  as well  as $\tilde \psi$  are increasing  and strictly
convex on $(0, +\infty )$. Taking a smaller $t_0$ if necessary, we can assume that:
\begin{equation}
   \label{eq:t02}
t_0\tilde C_3^+\leq \tilde C^-_2.
\end{equation}

  We   deduce  from
\reff{eq:tfn-dev}   and    the   definition   of    $\tilde\psi_j$   and
$\tilde \varphi_j$  that there exits  $\varepsilon>0$ and a finite
constant $C$ such that  for all
$l, j\in \N^*$, $(u,t)\in K'$:
\[
f_l(\tilde \varphi_j(u+it)) =p(\fb)^{-1/(\fb -1)} \expp{\fb^l \tilde \psi_j(u+it)}
\left(1+ R(u,t,l,j)\right)
\]
and $\sup_{(u,t)\in K', j\in \N^*} |R(u,t,l,j)| \leq C\expp{-\varepsilon
  \fb^l}$.
Since $m=\inf\{\tilde \psi_j(u); \,u\in K, j\in \N^*\}>0 $, taking
$\varepsilon$ smaller than $m$ if necessary, we get that:
\[
H_{l,j}(u+it) =p(\fb)^{-\ell/(\fb -1)} \expp{\ell\fb^l \tilde \psi_j(u+it)}
\left(1+ R'(u,t,l,j)\right)
\]
and
$\sup_{(u,t)\in K',  j\in \N^*} |R'(u,t,l,j)|  \leq C'\expp{-\varepsilon
  \fb^l}$
for some  finite constant $C'$.   Since $f_l(\tilde \varphi(u))>1> \fc$ for $u>0$,  we deduce
  from   Lemma  \ref{lem:majo-fn}   that   for  all   $l,  j\in   \N^*$,
  $u\in (0, +\infty )$:
\begin{equation}
   \label{eq:majo-Hlj}
 0<H_{l,j}(u)\leq f_l(\tilde \varphi_j(u))^\ell\leq  p(\fb)^{-\ell/(\fb -1)} \exp \left\{\ell \fb^l \tilde
   \psi_j(u)\right\}. 
\end{equation}

\subsubsection{The tail part}
We first bound the tail of the integral which appears  in \reff{eq:PZ=an}:
\[
I_{l,j}(t_0)=\val{\int_{|t|\in [t_0, c_j\pi/L_0]}  H_{l,j}(u+it)
  \expp{-(u+it)y \ell \fb^l}\, dt},
\]
where $y$ belongs to $[1/c_0,  \fb  c_{j-1}/c_0  c_{j}  )$. 
Using an integration by parts, we get:
\[
\int_{|t|\in [t_0, c_j\pi/L_0]}  H_{l,j}(u+it)\expp{-(u+it)y \ell \fb^l}
\, dt= I^{+1}_1-I_1^{-1}+I_2, 
\]
where, for $\epsilon\in \{+1,-1\}$
\[
I_1^\epsilon=\left[iH_{j,l}(u+it) \frac{\expp{-(u+it)y \ell \fb^l}}{y \ell
     \fb^l}\right]^{\delta \frac{cj\pi}{L_0}}_{\delta t_0}
\quad\text{and}\quad 
I_2=-i\int_{|t|\in [t_0, c_j\pi/L_0]}  H'_{l,j}(u+it)\frac{\expp{-(u+it)y
    \ell \fb^l}}{y \ell
     \fb^l} \, dt. 
\]

Set $\ca_j=\{(u,t)\in \R^2;\, u\in K, t_0\leq |t|\leq  c_j\pi/L_0\}$. 
According to \reff{eq:fw-Fourier2}, there exists $\delta\in (0,1)$ such
that for all $j\in \N^*$ and $(u,t)\in \ca_j$:
\begin{equation}
   \label{eq:majo-d-jfj}
|\tilde \varphi_j (u+it)| \leq  (1-\delta) \tilde \varphi_j (u).
\end{equation}
Taking $\delta$ small enough, we can assume that $m_1=\inf\{(1-\delta) \tilde
\varphi_j(u); \, j\in \N^*, u\in K\}> 1$. 
We have $H_{l,j}(z)=g(1)-g(0) =\int_0^1 g'(s) \, ds$, with
$g(s)=f_l(s\tilde \varphi_j(z) + (1-s)\fc)^\ell$. 
We get:
\[
|g'(s)|\leq  \val{\tilde \varphi_j(z) - \fc}
\ell f_l(s(1-\delta) \tilde \varphi_j(u) +(1-s)\fc)^{\ell
  -1}f_l'(s(1-\delta) \tilde \varphi_j(u) +(1-s)\fc).
\]
We deduce that for all $l,j\in \N^*$ and $z=u+it$ with $(u,t)\in \ca_j$:
\[
|H_{l,j}(z)|\leq   \val{\tilde \varphi_j(z) - \fc} \frac{f_l((1-\delta) \tilde
  \varphi_j(u))^\ell - \fc^\ell}{1-\fc}\leq 
2\tilde \varphi_j(u) \frac{f_l((1-\delta) \tilde 
  \varphi_j(u))^\ell }{1-\fc}\cdot
\]
Using Lemma \ref{lem:majo-fn}, we get there exists a constant $C$ such
that for all $l,j\in \N^*$ and $(u,t)\in \ca_j$:
\[
|H_{l,j}(z)|\leq C\exp \left\{\ell \fb^l
  \tilde b ( (1-\delta)\tilde
\varphi_j(u)  )\right\}.  
\]
Using that $\tilde b$ is analytic and increasing on $(1, +\infty )$  and
$m_1>1$, 
we deduce that there exists $\varepsilon'>0$ such that for all $j\in
\N^*$, $u\in K$:
\[
\tilde b ( (1-\delta)\tilde
\varphi_j(u) )\leq  \tilde \psi_j(u) - \varepsilon'.
\]
We deduce  that for all $l,j\in
\N^*$, $
(u,t)\in \ca_j$:
\[
| H_{l, j}(u+it) |
 \, dt \leq  C \exp \left\{\ell \fb^l
  \tilde \psi (u)   - \ell \fb^l \varepsilon' \right\}. 
\]
This gives that for all $u\in K$, $l,j\in \N^*$:
\begin{equation}
   \label{eq:majointH-I1}
|I_1^{\pm 1}|\leq  \frac{2  C}{y\ell \fb^l}\,  \expp{\ell \fb^l
  (\tilde \psi (u)  - u y) - \ell \fb^l \varepsilon' }. 
\end{equation}

We have $H'_{l,j}(z)=\ell \tilde \varphi'_j(z) f'_l\left(\tilde
  \varphi_j(z)\right)f_l\left(\tilde \varphi_j(z)\right)^{\ell -1}$. 
For $(u,t)\in \ca_j$, we have using \reff{eq:f'n-bound},
\reff{eq:majo-d-jfj} and
$f'_l(|z|)\leq \fb^l f_l(|z|)/|z|$:
\[
|H'_{l,j}(u+it)|
\leq  \frac{\ell}{m_1}  \tilde |\varphi_j'(u+it)| \fb^l  f_l\left((1-\delta) \tilde
    \varphi_j(u)\right)^\ell.
\]
Arguing as in the upper bound on $I_1^\pm$, we get there exists a
finite constant $C$ such that for all $l,j\in \N^*$, $u\in K$
\[
|I_2|\leq  
\frac{ C}{y}\,  \expp{\ell \fb^l
  (\tilde \psi (u)  - u y) - \ell \fb^l \varepsilon' }\, 
\int_{[\pm c_j\pi/L_0]}  |\tilde\varphi_j'(u+it)| \, dt .
\]
Then use \reff{eq:fw-Fourier3}, to conclude that $
|I_2|\leq (C/y)\,  \expp{\ell \fb^l
  (\tilde \psi (u)  - u y) - \ell \fb^l \varepsilon' }$ for some finite
constant $C$.
This and \reff{eq:majointH-I1} give there exists a
finite constant $C$ such that for all $l,j\in \N^*$, $u\in K$:
\begin{equation}
   \label{eq:majointH-t0}
I_{l,j}(t_0)\leq  \frac{C}{y}\,  \expp{\ell \fb^l
  (\tilde \psi (u)  - u y) - \ell \fb^l \varepsilon' }.
\end{equation}

\subsubsection{The main part}
The main part is  handled as in \cite{fw:lta} from (168) to
(172), see also \cite{fo:ld}. 
For $(u,t)\in K'$, we have $\tilde
\varphi_j(u+it) \in \tilde \cd(\delta)$
and we deduce from \reff{eq:tfn-dev}, that  there exists
$\varepsilon>0$ such that for $(u,t)\in K'$, $l,j\in \N^*$:
\[
 \int_{-t_0}^{t_0}  H_{l, j}(u+it) 
\expp{-(u+it)y \ell \fb^{l}} \, dt
= p(\fb)^{-\ell/(\fb-1)} \left(1+ O(\expp{-\varepsilon \fb^l})\right)
D(j,l,u),
\]
and
\[ 
D(j,l,u)=\int_{-t_0}^{t_0} \expp{\ell \fb^{l}\left(\tilde \psi_j(u+it) -(u+it)y \right)}
\, dt.
\]
where $O(\expp{-\varepsilon \fb^l})=R(u,t,j,l,y)$ and there exists some
finite constant $C$ such that for all $l\in \N^*$, we have $\sup_{j\in \N^*}
\sup_{y\in F_j, (u,t)\in K'}  |R(u,t,j,l,y)|\leq C \expp{-\varepsilon
  \fb^l}$. 
We have for $(u,t)\in K'$:
\[
\tilde \psi_j(u+it)=\tilde \psi_j(u)+ it \tilde \psi_j'(u) -
\frac{t^2}{2} \tilde \psi_j''(u) + h_j(t,u),
\]
with          $|h_j(t,u)|\leq           t^3          \tilde C_3^+/6$,     
 Recall that $\tilde
u_{n,\ell}^*$ belongs to $K$. 
With the definition of $\tilde u_{n,\ell}^*$, we get that:
\[
\tilde \psi_{j_n}(\tilde u_{n,\ell}^* +it) -(u_{n,\ell}^*+it)y_n
=  \tilde \psi_{j_n}(\tilde u_{n,\ell}^*) - u_{n,\ell}^*y_n - \frac{t^2}{2} \tilde
\sigma^2_{n, \ell} +  h_{j_n}(t, \tilde u_{n,\ell}^*),
\]
with $\tilde \sigma^2_{n, \ell}=
\tilde \psi_{j_n}''(\tilde u^*_{n, \ell})$, 
$|h_{j_n}(t, \tilde u_{n,\ell}^*)|\leq  t^3 \tilde C_3^+/6$. We consider the
decomposition  $D(j_n, l_n,\tilde u_{n,\ell}^*)=D_1+D_2$ with:
\[
D_1=\int _{- l_n \fb^{-l_n/2}}^{l_n \fb^{-l_n/2}}
\expp{\ell \fb^{l_n}\left(\tilde \psi_{j_n}(\tilde u_{n,\ell}^*+it)
    -(\tilde u_{n,\ell}^*+it)y \right)} 
\, dt.
\]
Using that $|h_{j_n}(t, \tilde u_{n,\ell}^*)|\leq  l_n^3 \fb^{-3l_n/2}
\tilde C_3^+/6$ for $|t|\leq l_n \fb^{-l_n/2}$, we get:
\begin{align*}
   D_1
&= \expp{\ell \fb^{l_n}\left(\tilde \psi_{j_n}(\tilde u_{n,\ell}^*)
    -\tilde u_{n,\ell}^* y_n \right)} 
\int _{- l_n \fb^{-l_n/2}}^{l_n \fb^{-l_n/2}}
\expp{- \ell \fb^{l_n} \tilde \sigma^2_{n, \ell} t^2/2 + \ell \fb^{l_n}
  h_{j_n}(t, \tilde u_{n,\ell}^*)}\, dt\\
&= \expp{\ell \fb^{l_n}\left(\tilde \psi_{j_n}(\tilde u_{n,\ell}^*)
    -\tilde u_{n,\ell}^*y_n\right)} 
\int _{- l_n \fb^{-l_n/2}}^{l_n \fb^{-3l_n/2}}
\expp{- \ell \fb^{l_n} \tilde \sigma^2_{n, \ell} t^2/2 }\,dt \,
  \left(1+O( l_n^{3} \fb^{-3l_n/2} )
\right)\\
&= \inv{\sqrt{\ell \fb^{l_n} \tilde \sigma^2_{n, \ell}}}
\expp{\ell \fb^{l_n}\left(\tilde \psi_{j_n}(\tilde u_{n,\ell}^*)
    -\tilde u_{n,\ell}^*y_n \right)} 
\int _{- l_n \tilde \sigma_{n, \ell}\sqrt{\ell}}^{l_n \tilde \sigma_{n,
  \ell}\sqrt{\ell}} 
\expp{-  s^2/2 }\,ds \,
  \left(1+O( l_n^{3} \fb^{-3l_n/2} )
\right)\\
&= \ci_n \times
  \left(1+O( l_n^{3} \fb^{-3l_n/2})
\right),
\end{align*}
with
\[
\ci_n=\frac{\sqrt{2\pi}}{\sqrt{\ell \fb^{l_n} \tilde \sigma^2_{n, \ell}}}
\expp{\ell \fb^{l_n}\left(\tilde \psi_{j_n}(\tilde u_{n,\ell}^*)
    -\tilde u_{n,\ell}^*y_n \right)} .
\]

We now give an upper bound for $|D_2|$. 
Thanks to \reff{eq:t02}, 
we have  $|h_{j_n}(t, \tilde u_{n,\ell}^*)|\leq t^2 \tilde
\sigma^2_{n, \ell}/6$ for $t\in [-t_0, t_0]$. We deduce that for
$t\in [-t_0, t_0]$:
\[
\fR\left( \tilde \psi_{j_n}(\tilde u_{n,\ell}^*+it) -(\tilde u_{n,\ell}^*+it)y_n \right)
\leq  \tilde \psi_{j_n}(\tilde u_{n,\ell}^*) - \tilde u_{n,\ell}^*y_n - \frac{t^2}{3} \tilde
\sigma^2_{n, \ell}.
\]
This implies that:
\begin{align*}
 |D_2|
&\leq  \expp{\ell \fb^{l_n}\left(\tilde \psi_{j_n}(\tilde u_{n,\ell}^*)
    -\tilde u_{n,\ell}^* y_n \right)} \int _{|t|\in [ l_n \fb^{-l_n/2}, t_0]}
\expp{- \ell \fb^{l_n} \tilde \sigma^2_{n, \ell} t^2/3}\, dt\\
&\leq 2t_0 \expp{\ell \fb^{l_n}\left(\tilde \psi_{j_n}(\tilde u_{n,\ell}^*)
    -\tilde u_{n,\ell}^* y_n \right)} \expp{- \ell l_n^2 \tilde
  \sigma^2_{n, \ell}/3}\\
&=\ci_n\times
  O( l_n^{3} \fb^{-3l_n/2}).
\end{align*}

\subsubsection{Conclusion}
To conclude, we deduce from \reff{eq:majointH-t0} with $y=y_n$ that:
\[
\int_{|t|\in [t_0, c_{j_n}\pi/L_0]} | H_{l_n, j_n}(\tilde u_{n,
  \ell}^*+it) \expp{-(\tilde u_{n,
  \ell}^*+it) y_n \ell \fb^{l_n}} |
 \, dt 
= \ci_n \times O(\expp{-\varepsilon \fb^{l_n}/2})
\]
This implies that:
\[
\int_{-
  \frac{c_{j_n}\pi}{L_0}}^{\frac{c_{j_n}\pi}{L_0}}  H_{l_n, j_n}(\tilde u_{n,
  \ell}^*+it) 
\expp{-(\tilde u_{n,
  \ell}^*+it)y_n \ell \fb^{l_n}} \, dt=
p(\fb)^{-\ell/(\fb-1)}\ci_n\times \left(1+ O( l_n^{3}
  \fb^{-3l_n/2}) \right). 
\]
Then use \reff{eq:PZ=an} to conclude.

\section{Results in the Bötcher case}
\label{sec:bottcher}
We  present  mostly the results without proof as their correspond either
to a
slight generalization of  \cite{fw:ld} and \cite{fw:lta} or can be
proven by mimicking the proof in the Harris case presented in Section
\ref{sec:Harris}.   Recall the   Böttcher  constant $\beta\in
(0,1 ) $ is defined  by
$\fa  =\mu^{\beta}$, where $\fa$ is the minimum of the support of $p$.
We assume $\fa\geq 2$. 
\subsection{Preliminaries}
We  define  the  function  $b$  on  its domain  which  is  a  subset  of
$\{z\in \C; \, 0<|z|< 1\}$ by:
\begin{equation}
   \label{eq:def-bz}
b(z)=\log(z)+ \sum_{n=0}^\infty \fa^{-n-1}
\log\left(\frac{\ff_{n+1}(z)}{\ff_n(z)^\fa}\right).
\end{equation}
 According to
Lemma 10 in \cite{fw:lta}, for every $\delta\in (0,1)$, there exists a constant
$\theta=\theta(\delta)\in (0, \pi)$ such that $b$ is analytic on the
open set:
\begin{equation}
   \label{eq:def-D}
\cd(\delta,\theta)= \{z\in \C; \, 0<|z|<  1-\delta, \, |\arg(z)|<
\theta\}.
\end{equation}
On $(0, 1)$,
the function $b$ is analytic, negative and satisfies $b\circ
\ff=\fa b$. We also have,
see Lemma 14 in \cite{fw:lta} that: 
\[
(sb'(s))'>0 \, \text{  for $s\in (0,1)$}, \quad
\lim_{s\nearrow 1} s b'(s)=+\infty \quad\text{and}\quad
\lim_{s\searrow 0} s b'(s)=1 . 
\]
Recall that $\cl$ denotes the Laplace transform of $W$. We also consider
the function $\psi=b\circ\cl$  defined on $(0, +\infty  )$. According to
Lemma  17  in   \cite{fw:lta},  the  function  $\psi$   is  analytic  on
$(0, +\infty )$ strictly decreasing, strictly convex and such that:
\[
\lim_{x \rightarrow 0+} \psi'(x)=-\infty  \quad\text{and}\quad
\lim_{x \rightarrow +\infty} \psi'(x)=0.
\]
Let $g$ be the inverse of $-\psi'$. In particular,  for a given $v>0$,
the minimum of $\psi(u) + uv $ for $u\geq 0$ is uniquely reached at
$g(v)$: 
\begin{equation}
   \label{eq:min=g}
\min_{u\geq 0} \left(\psi(u) + uv
\right) = \psi(g(v))+ g(v)v.
\end{equation}

\subsection{Left tail of $w$}
We define the function $M$ for $v\in (0, +\infty )$ by:
\begin{equation}
   \label{eq:defM}
M(v)= -v^{\beta/(1-\beta)} \, \min_{u\geq 0} \left(\psi(u) + uv
\right).
\end{equation}
The function $M$ is analytic on $(0, +\infty )$, see Proposition 3 in \cite{bb:ldsbp}, positive and multiplicatively
periodic with period $\mu^{1-\beta}$. For $x\in (0, \fa/\mu]$, we set:
\begin{equation}
   \label{eq:def-y}
r(x)= \left\lfloor \frac{\log(x)}{\log(\fa/\mu)} \right\rfloor
\quad\text{and}\quad y(x)=x\left(\frac{\mu}{\fa}\right)^{r(x)},
\end{equation}
so that $y(x) \in (\fa/\mu, 1]$. For  $\ell\in \Np$ and $y>0$, we define
the positive functions:
\[
\cm_{1, \ell}(y)= \frac{p(\fa)^{-\ell/(\fa -1)}}{\sqrt{2\pi
    \ell \sigma^2(y/\ell)}}\, y^{(2-\beta)/2(1-\beta)}
\quad\text{and}\quad
\cm_{2, \ell}(y)= \cm_{1, \ell}(y)\frac{y^{-1/(1-\beta)} }
{g(y/\ell)},
\]
where $\sigma^2(y)=\psi''(g(y))>0$. For  $\ell\in \Np$ and $x\in (0,
\fa/\mu]$, we set: 
\[
M_{1, \ell}(x)=\cm_{1, \ell}(y(x))
\quad\text{and}\quad
M_{2, \ell}(x)=\cm_{2, \ell}(y(x)). 
\]
By construction $x\mapsto y(x)$ is multiplicative periodic with period
$\fa/\mu=\fa^{1-\beta}$. We deduce that $M_{1, \ell}$ and $M_{2, \ell}$
are 
multiplicative periodic with period
$\fa/\mu=\fa^{1-\beta}$, positive, bounded and bounded away from 0.

Let  $\P_\ell$  be  the  distribution  of  $\sum_{i=1}^\ell  W_i$,  with
$(W_i,  i\in  \Np)$ independent  random  variables  distributed as  $W$.
Since $\fa>0$ and thus $\fc=0$, we get that $W$ has density $w$ and that
$\sum_{i=1}^\ell W_i$  has density $w^{*\ell}$.  Mimicking  very closely
the proof in \cite{fw:lta} stated for $\ell=1$, it is not very difficult
to check the following result. The verification is left to the reader.

\begin{lem}
   \label{lem:bott-wl}
Let $p$ be a non-degenerate super-critical offspring distribution with
finite mean and $\fa\geq 2$. Let $\ell\in \Np$.  As $x\searrow 0$, we have:
\begin{equation}
   \label{eq:w=0}
w^{*\ell}(x)\sim w_\ell(x)\sim M_{1,\ell}(x) \, x^{(\beta-2)/2(1-\beta)} \exp\left\{-
  \ell^{1/(1-\beta)}\, x^{-\beta/(1-\beta)} M(x/\ell)\right\}, 
\end{equation}
\begin{equation}
   \label{eq:P<x}
\P_\ell(W\leq x)\sim M_{2,\ell}(x) \, x^{\beta/2(1-\beta)} \exp\left\{-
  \ell^{1/(1-\beta)}\,   x^{-\beta/(1-\beta)} M(x/\ell)\right\}.
\end{equation}
\end{lem}

Using (118), (119),
(122)  (with $\ff$  replaced by  $\ff_\ell$), (123)  and (78)  in
\cite{fw:lta}, we also get the following upper bound, see also Lemma 
\ref{lem:w-maj} in the Harris case.

\begin{cor}
   \label{cor:w-maj-a}
Let $p$ be a non-degenerate super-critical offspring distribution with
finite mean and $\fa\geq 2$. There exists a
finite constant $C$ such that for all $\ell\in \Np$, $x>0$ and $u\geq 0$,  we
have with $r=r(x)$, $y=y(x)$:
\begin{equation}
   \label{eq:w-maj2}
w^{*\ell} (x) \leq  C \mu^{r} \frac{\expp{u y  \fa^{r}}} 
{\cl(u)} \, \ff_{r} (\cl(u))^\ell.
\end{equation}
\end{cor}

\subsection{Proof of Lemma \ref{lem:cv-rho} in the Böttcher case}
\label{sec:B-f(c)=0}
Mimicking the arguments given in Section \ref{sec:H-f(c)}, it is easy,
using Corollary \ref{cor:w-maj-a} to get that:
\[
\lim_{x\rightarrow 0+} \mu \frac{w^{*\fa \ell}(x)}{w^{*\ell}(x/\mu)} \,
p(\fa)^\ell=1.
\]
From the definition of $\rho_{\theta,\ell}$ in \reff{eq:def-rho}, we
deduce that $\lim_{\theta\rightarrow 0+} \rho_{\theta,\ell} (\fa, \dots
\fa)=1$.
This ends the proof of Lemma \ref{lem:cv-rho} in the Böttcher case.

\subsection{Lower large deviations for $Z_n$}
\label{sec:lowerZn}

For $j\in \N^*$, let  $\cl_j$ denote
the         Laplace          transform         of         $W_j=Z_j/c_j$:
$\cl_j(u)=\E[\expp{-u W_j}]=\ff_j(\expp{-u/c_j})$     for    $u\in     \C_+$,    where
$\C_+=\{u\in  \C, \,  \fR(u)\geq  0\}$.  Notice  that $\cl_j$  converges
uniformly on the compacts of $\C_+$ towards $\cl$, the Laplace transform
of  $W$,   as  $j$  goes   to  infinity. We also have that
$\cl_j'(u)/\cl_j(u)= -\E[W_j \expp{-uW_j}]/\E[\expp{-uW_j}]$ so that 
$\lim _{u\rightarrow+\infty } \cl_j'(u)/\cl_j(u)= - \fa^j/c_j$.

We    consider   the    functions   $\psi_j=b\circ\cl_j$    defined   on
$(0,  +\infty )$  for  $j\in \N^*$  and  the function  $\psi=b\circ\cl$.
According to Lemma 17 in  \cite{fw:lta}, the function $\psi$ is analytic
on $(0,  +\infty )$ strictly  decreasing, strictly convex and  such that
$\lim_{x        \rightarrow       0+}        \psi'(x)=-\infty$       and
$  \lim_{x \rightarrow  +\infty}  \psi'(x)=0$.  Mimicking  the proof  of
Lemma  17 in  \cite{fw:lta},  it is  easy to  check  that the  functions
$\psi_j$ are analytic  on $(0, +\infty )$  strictly decreasing, strictly
convex and such that:
\[
\lim_{x \rightarrow 0+} \psi_j'(x)=-\infty  \quad\text{and}\quad
\lim_{x \rightarrow +\infty} \psi_j'(x)=- \frac{\fa^j}{c_j}\cdot
\]
Let $g_j$ (resp.  $g$) be the  inverse of $-\psi'_j$ (resp. $-\psi'$) on
$(\fa^j/c_j, +\infty )$ (resp. on $(0, +\infty )$). In particular, for a
given $v>\fa^j/c_j$, the  minimum of $\psi_j(u) + uv $  for $u\geq 0$ is
uniquely reached at $g_j(v)$. Using  that $\psi_j$ converges uniformly on
compact of $\C_+$  towards $\psi$,  that  $b$  and thus  $\psi_j$  and $\psi$  are
analytic, we get that  for any compact of $(0, +\infty  )$,  the strictly  convex functions  $\psi_j$ and  their derivatives
converge uniformly towards  the strictly convex function  $\psi$ and its
derivatives. We deduce  that for  any compact of  $(0, +\infty )$, $g_j$ converges uniformly towards $g$.  \medskip

We consider the following general setting. Let
$\ell\in  \N^*$ and  $a_n\in
(\ell\fa^n, \ell c_n/c_0]$ such that $\liminf_{n\rightarrow\infty }a_n/\ell
\fa^n >1$. 
Since $\fa<c_{r+1}/c_r < \mu$ for all $r\in \N$, we deduce that the
sequence $(c_{n-l} \fa^{l}, \, 0\leq l\leq n)$ is decreasing. 
Therefore, the integer $l_n= \sup\{l\in \{0, \ldots, n\},
\, c_{n-l} \ell \fa^{l}
\geq c_0 a_n\}$ is well-defined and strictly less than $n$. Set $j_n=n-l_n\geq 1$ and 
\[
a_n=y_n \,  c_{j_n}\,  \ell\fa^{l_n},
\]
 with
$y_n\in (\fa c_{j_{n}-1}/c_0c_{j_n}, 1/c_0]$. 
Notice that the conditions $\lim_{n\rightarrow\infty }a_n/c_n=0$ and $a_n> \ell
\fa^n$ imply that $\lim_{n\rightarrow\infty } l_n=+\infty $. The
sequence $(j_n, n\in \N^*)$ may be bounded or not.

As   $\fa<c_{r+1}/c_r$    for   all   $r\in   \N$,    we   deduce   that
$y_n> \fa c_{j_n-1}/c_0c_{j_n} > \fa^{j_n}/c_{j_n}$. Thus, we can define
$u_{n,\ell}^*=g_{j_n}(y_n)$                                          and
$\sigma^2_{n,\ell}=\psi_{j_n}''(u_{n,\ell}^*)$.  Mimicking  very closely
the proof  of (175) in \cite{fw:lta}  (which is stated for  $\ell=1$ and
$\lim_{n\rightarrow\infty } j_n=\infty  $), it is not  very difficult to
check  the  following slightly  more  general  result. The  verification,
which can also be seen as a  direct adaptation of the detailed proof of
Lemma \ref{lem:H-Pz}, is left to the reader.
\begin{lem}
   \label{lem:bott-Pz}
Let $p$ be a non-degenerate super-critical offspring distribution  with
finite mean, $\fa\geq 2$   and  type $(L_0,r_0)$. Let $\ell\in \Np$. Assume that 
  $\lim_{n\rightarrow\infty }a_n/c_n=0$ and    $\liminf_{n\rightarrow\infty
  }a_n/\ell \fa^n >1$. Then, we have, with $\lim_{n\rightarrow\infty }
  \varepsilon_{n, \ell}=0$: 
\[
\P_\ell(Z_n=a_n)
= \frac{L_0\,  p(\fa)^{-\ell/(\fa -1)}}
{c_{j_n}\sqrt{2\pi \, \ell \fa^{l_n} \, \sigma^2_{n, \ell}}} \, \exp \left\{
\ell\fa^{l_n} ( \psi_{j_n}( u_{n,\ell}^*) + u_{n,\ell}^* y_n) 
  \right\} (1+ \varepsilon_{n,\ell}(1))\ind_{\{a_n=\ell r_0^n(\mod L_0)\}}.
\] 
\end{lem}

We end this section with the following strong ratio limit, whose proof is similar
to the proof of Lemma \ref{lem:rapp-Zb}. 
\begin{lem}
   \label{lem:rapp-Za}
   Let  $p$ be  a non-degenerate  super-critical offspring  distribution
   with   finite    mean and $\fa\geq 2$.    Assume   that
   $\lim_{n\rightarrow\infty                }a_n/c_n=0                $,
   $\liminf_{n\rightarrow\infty     }a_n/\ell      \fa^n     >1$     and
  $a_n=\ell r_0^n(\mod L_0)$ for all $n\in \N^*$. Then, we have:
\begin{equation}
   \label{eq:rapp-Za}
\lim_{n\rightarrow\infty }\frac{\P_{ \ell\fa^h} (Z_{n-h}=a_n)}{\P_\ell(Z_n=a_n)}
= p(\fa)^{-(\fa^h-1)\ell/(\fa-1)}.
\end{equation}
\end{lem}

% ******************
% * DETAILED PROOF *
% ******************

% We now give an estimation of $\P_{\ell'}(Z_{n'}=a_n)$ with $n'=n-h$  for
% some $h\in \N ^*$ and $\ell'=\fa^h \ell$. 
% We have:
% \[
% a_n=y'_{n}c_{j'_{n}}\, \ell' \fa^{l'_n}=
% y'_{n}c_{j'_{n}}\, \ell \fa^{l'_n+h},
% \]
% with                        $j'_n+l'_n=n'=n-h$                       and
% $l'_n=  \sup\{l\in \{0,  \ldots, n'=n-h\},  \, c_{n-h-l}  \ell \fa^{l+h}
% \geq                                            c_0                 a_n\}$.
% From  the definition  of $l'_n$,  we  deduce that  $l'_n=l_n-h$ so  that
% $\ell' \fa^{l'_n}=\ell \fa^{l_n}$, $j'_n=j_n$  and thus $y'_n=y_n$. This
% gives       that       $g_{j_n}(y_n)=g_{j'_n}(y'_n)$      and       thus
% $u^*_{n',         \ell'}=u^*_{n,\ell}$         as        well         as
% $\sigma^{2}_{n',   \ell'}=\sigma^{2}_{n,   \ell}$.    We   deduce   from
% \reff{eq:Z=0} that, as $n$ goes to infinity:
% \[
% \frac{\P_{ \ell\fa^h} (Z_{n-h}=a_n)}{\P_\ell(Z_n=a_n)}
% \sim p(\fa)^{-(\fa^h-1)\ell/(\fa-1)}.
% \]
% Taking $\ell=1$ gives \reff{eq:an=0-B}. 

% *************************
% * END OF DETAILED PROOF *
% *************************

\subsection{Proof of Proposition \ref{prop:cv-not-fat-gene} in the
Böttcher case}
\label{sec:proof-B}

For $h\in \N$, we have $\P(r_h(\tau)=r_h(\bt_\fa))=p(\fa)^{(\fa^{h}
  -1)/(\fa-1)}$. We deduce from \reff{eq:ph} and  the  convergence
characterization \reff{eq:cv-loi}, using that $\bt_\fa$
has  a.s.  an  infinite  height, that the proof of Proposition
\ref{prop:cv-not-fat-gene} is complete as soon as we prove  the
following strong
ratio limit. 
\begin{lem}\label{lem:srt-bottcher}
  Let $p$ be a non-degenerate super-critical offspring distribution with
  finite mean and such that $\fa\ge 2$. Assume that
  $\lim_{n\to+\infty}a_n/c_n=0$  and  that   $\P(Z_n=a_n)>0$  for  every
  $n\in\N$ (which implies  that $a_n\ge \fa^n$). Then,  we have for $h,k\in\N^*$:
\begin{equation}
   \label{eq:an=0-B}
\lim_{n\rightarrow\infty }  \frac{\P_{k}(Z_{n-h}=a_n)}{\P(Z_n=a_n)}=
p(\fa)^{-(\fa^{h} 
  -1)/(\fa-1)}\ind_{\{k=\fa^h\}}. 
\end{equation}
\end{lem}

In fact, it is  enough to prove \reff{eq:an=0-B} for $k=\fa^h$ as 
 $\P(Z_h=\fa^h)=p(\fa)^{-(\fa^{h} 
  -1)/(\fa-1)}$. It is also enough to consider the two cases:
 $\lim_{n\rightarrow\infty }
a_n/{\fa^n}=1$ and  $\liminf_{n\rightarrow\infty }
a_n/{\fa^n}>1$.

The case $\lim_{n\rightarrow\infty }
a_n/{\fa^n}=1$ is handled as in the Harris case, see the first part of
the proof of Proposition \ref{prop:cv-not-vfat-gene} in Section
\ref{sec:H-R}. The case  $\liminf_{n\rightarrow\infty }
a_n/{\fa^n}>1$ is a consequence of Lemma \ref{lem:rapp-Za}.

\bibliographystyle{abbrv}
\bibliography{biblio}

\end{document}